\documentclass{amsart}

\usepackage{graphicx}
\usepackage{multicol,multirow}
\usepackage{amsmath,amssymb,amsfonts}
\usepackage{mathrsfs}
\usepackage{amsthm}
\usepackage{rotating}
\usepackage{appendix}
\usepackage[numbers]{natbib}

\usepackage{amssymb,stmaryrd}
\usepackage{enumerate, url}
\usepackage{hyperref}
\usepackage{cleveref}

\newtheorem{theorem}{Theorem}[section]
\newtheorem{lemma}[theorem]{Lemma}

\newtheorem{remark}[theorem]{Remark}
\newtheorem{example}[theorem]{Example}
\newtheorem{proposition}{Proposition}

\numberwithin{equation}{section}

\newcommand{\R}{\mathbb{R}}
\newcommand{\N}{\mathbb{N}}

\newcommand{\rD}{\mathrm{D}}

\newcommand{\sym}{\mathrm{sym}}
\newcommand{\asym}{\mathrm{skew}}

\newcommand{\Herm}[1]{\mathfrak{p}(#1)}

\newcommand{\St}[2]{\mathrm{St}_{#1,#2}}  % Psi, field
\newcommand{\Sto}{\mathrm{St}}  % Psi, field

\newcommand{\Gr}[2]{\mathsf{Gr}_{#1,#2}}
\newcommand{\Gro}{\mathsf{Gr}}

\newcommand{\cX}{\mathcal{X}}

\newcommand{\cB}{\mathcal{B}}

\newcommand{\cE}{\mathcal{E}}

\newcommand{\cH}{\mathcal{H}}
\newcommand{\cI}{\mathcal{I}}
\newcommand{\Sd}[2]{\mathrm{S}^{+}_{#1, #2}}
\newcommand{\Sp}[1]{\mathrm{S}^{+}_{#1}}

\newcommand{\lb}{\llbracket}
\newcommand{\rb}{\rrbracket}
\newcommand{\cV}{\mathcal{V}}

\newcommand{\MVN}[1]{\mathtt{MVN}(n)}

\newcommand{\PiW}{\Pi_{\cW}}
\newcommand{\ttq}{\mathtt{q}}

\newcommand{\bbf}{\mathsf{f}}

\newcommand{\ft}{\mathsf{T}}
\newcommand{\rgrad}{\mathsf{rgrad}}

\newcommand{\cF}{\mathcal{F}}

\newcommand{\bcM}{\bar{\mathcal{M}}}

\newcommand{\cM}{\mathcal{M}}
\newcommand{\cQ}{\mathcal{Q}}

\newcommand{\bdD}{\mathring{\mathrm{D}}}

\newcommand{\cU}{\mathcal{U}}

\newcommand{\oo}{\mathfrak{o}}

\newcommand{\bB}{\bar{B}}

\newcommand{\sfa}{\mathsf{a}}

\newcommand{\crosscv}{\mathrm{cross}}

\newcommand{\sfg}{\mathsf{g}}

\newcommand{\rC}{\mathrm{C}}

\newcommand{\rA}{\mathrm{A}}

\newcommand{\cW}{\mathcal{W}}
\newcommand{\cN}{\mathcal{N}}

\newcommand{\cT}{\mathcal{T}}

\DeclareMathOperator{\Imag}{Im}
\DeclareMathOperator{\dI}{I}

\newcommand{\Lin}{\mathtt{Lin}}
\newcommand{\vxi}{v_{\xi}}
\newcommand{\veta}{v_{\eta}}
\newcommand{\mrGamma}{\mathring{\Gamma}}

\newcommand{\ttX}{\mathtt{X}}
\newcommand{\ttY}{\mathtt{Y}}

\newcommand{\ttZ}{\mathtt{Z}}
\newcommand{\ttW}{\mathtt{W}}

\newcommand{\tdX}{\widetilde{\mathtt{X}}}
\newcommand{\tdY}{\widetilde{\mathtt{Y}}}
\newcommand{\tdZ}{\widetilde{\mathtt{Z}}}

\newcommand{\ttP}{\mathtt{P}}

\newcommand{\qq}{\mathfrak{q}}

\newcommand{\tth}{\mathtt{h}}

\newcommand{\bdomg}{\mathring{\omega}}

\newcommand{\bdxi}{\mathring{\xi}}
\newcommand{\bdeta}{\mathring{\eta}}
\newcommand{\bomg}{\bar{\omega}}
\newcommand{\bxi}{\bar{\xi}}
\newcommand{\bareta}{\bar{\eta}}

\DeclareMathOperator{\SOO}{SO}
\DeclareMathOperator{\SE}{SE}

\DeclareMathOperator{\ttV}{V}
\DeclareMathOperator{\ttB}{B}
\DeclareMathOperator{\sA}{\mathsf{A}}

\DeclareMathOperator{\GammaH}{\Gamma_{\cH}}

\newcommand{\sAd}{\mathsf{A}^{\dagger}}
\DeclareMathOperator{\bRcB}{\widetilde{R}^{\cB}}

\DeclareMathOperator{\RcB}{R^{\cB}}
\DeclareMathOperator{\RcQ}{R^{\cQ}}

\DeclareMathOperator{\ttH}{H}
\DeclareMathOperator{\Rc}{R}

\DeclareMathOperator{\brD}{\bar{D}}

\DeclareMathOperator{\Tr}{Tr}

\DeclareMathOperator{\Null}{Null}
\DeclareMathOperator{\Hess}{Hess}

\DeclareMathOperator{\OO}{O}

\DeclareMathOperator{\egrad}{egrad}
\DeclareMathOperator{\Two}{II}

\begin{document}

\title[Geometry in global coordinates in mechanics and optimal transport]{Geometry in global coordinates with applications in mechanics and optimal transport}
\author{Du Nguyen}

\email{nguyendu@post.harvard.edu}
%\subjclass[2010]{Primary 53C05, 53C42, 	70H05, 70H45, 70H33, 53D05, 53Z30, 53Z50}

\address{
  organization={Independent},
  city={Darien},
  postcode={06820},
  state={CT},
  country={USA}
}

\begin{abstract}For a manifold embedded in an inner product space, we express geometric quantities such as {\it Hamilton vector fields, affine and Levi-Civita connections, curvature} in global coordinates. Instead of coordinate indices, the global formulas for most quantities are expressed as {\it operator-valued} expressions, using an {\it affine projection} to the tangent bundle. For a submersion image of an embedded manifold, we introduce {\it liftings} of Hamilton vector fields, allowing us to use embedded coordinates on horizontal bundles. We derive a {\it Gauss-Codazzi equation} for affine connections on vector bundles. This approach allows us to evaluate geometric expressions globally, and could be used effectively with modern numerical frameworks in applications. Examples considered include rigid body mechanics and Hamilton mechanics on Grassmann manifolds. We show explicitly the cross-curvature (MTW-tensor) for the {\it Kim-McCann} metric with a reflector antenna-type cost function on the space of positive-semidefinite matrices of fixed rank has nonnegative cross-curvature, while the corresponding cost could have negative cross-curvature on Grassmann manifolds, except for projective spaces.
\end{abstract}

\maketitle      
\section{Introduction}
\label{intro}
In geometric calculations, we often encounter submanifolds of an inner product space, defined by constraints. While local coordinates are often used theoretically, in practice, the dynamics of systems defined on a submanifold of $\cE:=\R^n$ could be expressed more conveniently in the global coordinates, (using Lagrange multipliers, for example), where both positions and velocities are expressed as points of $\cE$. Despite being widely used computationally, global coordinate treatments of fundamental concepts such as the identification of the cotangent bundle, connection, curvature, and lifting of submersions are not fully developed.

This paper is an attempt to treat the global coordinate approach systematically. We hope the contribution is timely, as, in recent years, the geometric approach is becoming important in engineering and data science applications, and an automated approach is favored. To compute differential geometric expressions systematically, the local chart approach is often difficult to use. As a chart is only a subset of a manifold, one may need to choose different charts to cover a manifold. Coordinate functions are usually computed by solving nonlinear equations, which are often hard. Further, geometric formulas often require using orthogonal bases, adding a layer of complexity. In contrast, the global approach works for any point of the manifold, involving {\it directional derivatives} of {\it operator-valued expressions} that are often {\it more convenient to work with}. A main finding of the paper is the {\it global formulas are simple generalizations of the local formulas} and can be used directly. Instead of vector fields, the formulas are in terms of a small number of operator-valued expressions that could be evaluated on (co)tangent vectors.

The effectiveness of this approach is immediate from the case of the sphere in \cref{expl:sphere}. The Levi-Civita connection and the curvature are computed immediately from the defining equation, with no trigonometric coordinate required, from global formulas in \cref{eq:LeviCivita,eq:rc1}. The global formula for Hamilton vector fields \cite{LeeLeokM}, surprisingly, is not well-known; with our extension, it implies a global formula for the Levi-Civita connection. Applying our approach, we introduce a new family of cost-function with nonnegative cross-curvature.

In \cite{KimMcCann}, Kim and McCann introduced a semi-Riemannian metric on a product of two manifolds given a cost function that plays an important role in {\it optimal transport}. It is shown \cite{KimMcCann,MTW,Loeper} that the smoothness of the associated {\it optimal map} is largely determined by the nonnegativity of the {\it MTW-tensor} or {\it cross-curvatures} (a form of sectional curvature, see \cref{sec:KimMcCann} for definition). Cost functions with nonnegative cross-curvature are of interest with a growing literature \cite{LeeMcCann,KimMcCann2012}. Using our framework,  in \cref{thm:crossSecFR}, we compute differential geometric measures for the Kim-McCann metric for a cost function similar to that of the {\it reflector antenna}, defined on the space of {\it fixed rank positive-semidefinite matrices}, and show the cross-curvature is nonnegative. We also study the induced geometry on {\it Grassmann manifolds}. We find the global formulas more intuitive and hope the readers will also be convinced through the examples presented.

Consider a smooth submanifold $\cQ$ of $\cE=\R^n$ for some integer $n$, where $\cE$ is equipped with an inner product. Assume $\cQ$ is defined by $k$ constraints ($k\in\N$), represented by an equation $\rC(q) = 0$, where $\rC$ is a (smooth) map from $\cE$ to $\cE_L=\R^k$. As a manifold, the tangent bundle $T\cQ$, the configuration space of position and velocity, is embedded naturally in $\cE^2$. Let $\rC'(q)$ be the Jacobian of $\rC$ at $q$, assuming smooth and regularity conditions, $T\cQ$ could be described as
\begin{equation}
  T\cQ = \{(q, v)\in \cE^2|\;\rC(q) = 0, \rC'(q)v = 0\}.
\end{equation}
%Acceleration could be identified with a point in $TT\cQ$ similarly.

A Hamiltonian function is a function on the cotangent bundle $T^*\cQ$, which is not embedded naturally in $\cE^2$. Using the inner product on $\cE$, $T^*\cQ$ could be identified with $T\cQ$, but this is not necessarily the most natural identification physically. We prove here an embedding of $T^*\cQ$ in $\cE^2$ is equivalent to a choice of a projection function $\Pi$, each $\Pi(q)$ ($q\in\cQ$) is an affine projection from $\cE$ onto the tangent space $T_q\cQ$, so $\Pi(q)\in\Lin(\cE, \cE)$, the space of linear map from $\cE$ to itself, and $\Pi(q)^2=\Pi(q)$. Then $T_q^*\cQ$ is identified with the range of the adjoint $\Pi(q)^{\ft}$ of $\Pi(q)$ on $\cE$.

As mentioned, the projection function $\Pi$ plays a central role in embedded mechanics and geometry, appearing in embedded generalizations of local formulas. A formulation of embedded Hamilton mechanics was given in \cite{LeeLeokM} for the case $\Pi(q)^{\ft} = \Pi(q)$. We {\it remove this restriction}, extending their result to affine projections and {\it simplify the formula} using the {\it Weingarten lemma} (\cref{eq:Wein}). The result could be expressed as a statement about Hamilton vector fields, where the Hamilton vector field of a function $H$ on $T^*Q$, extended to $\cE^2$ is given by
\begin{equation}
\ttX_H(q, p) =\begin{bmatrix}\Pi(q)H_p\\
    \Pi'(q; \Pi(q)H_p)^{\ft}p - \Pi(q)^{\ft}\{H_q +  \{X\mapsto \Pi'(q; X)^{\ft}p\}^{\ft}H_p \}\end{bmatrix}\label{eq:HamFlow}
\end{equation}
where $(H_q, H_p)$ is the gradient of $H$ as a $\cE^2$-function, $\Pi'(q, \xi)$ is the directional derivative of $\Pi$ in direction $\xi$, see \cref{thm:HVF}. The local coordinate formula $(H_p, -H_q)$ is a special case, where we identify an open chart in $\cQ$ with an open subset of its coordinate space $\R^{\dim\cQ}$, and $\Pi$ with the identity map. The pioneering work of \cite{LeeLeokM} inspires this result.

An important example \cite[section 3.7]{AbrahamMarsden} is a system with kinetic energy given by $H(q, p) =\frac{1}{2}p.\sfg(q)^{-1}p$, for $q\in\cQ$, with mass $\sfg(q)$ and momentum $p$. If the velocity $v$ is in $T_q\cQ$ and the momentum $p$ is in $T_q^*\cQ$, it is natural to think of the mass $\sfg(q)$ as a linear map from $T_q\cQ$ to $T_q^*\cQ$, which can be extended to an operator $\sfg(q)$ from $\cE$ to $\cE$ (see \ref{appx:Extend}). This induces a metric $v.\sfg(q) v$ on $\cQ$, making it a (semi)Riemannian manifold. If $\Pi(q)$ is the projection satisfying $\sfg(q)\Pi(q) = \Pi(q)^{\ft}\sfg(q)$, the cotangent space $T^*_q\cQ$ is identified with $\sfg(q)T_q\cQ = \Imag(\Pi(q)^{\ft})\subset \cE$. The Hamilton equations above imply a simple formula for the embedded geodesic equation, hence for the Levi-Civita connection (derived in \cite{Nguyen2020a}). The connection is expressed in terms of a global {\it Christoffel function} $\Gamma$ \cite{Edelman_1999}, a function from $\cQ$ to the space of bilinear functions $\Lin(\cE\otimes\cE, \cE)$, such that the covariant derivative of two vector fields $\ttX, \ttY$ is evaluated as $\nabla_{\ttX}\ttY = \rD_{\ttX}\ttY + \Gamma(\ttX, \ttY)$. We show  for two tangent vectors $\xi_1, \xi_2\in T_q\cQ$
\begin{equation}\Gamma(q; \xi_1, \xi_2) = -\Pi'(q; \xi_1)\xi_2 + \frac{1}{2}\Pi(q)\sfg(q)^{-1}(\sfg'(q; \xi_1)\xi_2 + \sfg'(q; \xi_2)\xi_1 - \chi_{\sfg}(q; \xi_1, \xi_2))\label{eq:LeviCivita}
\end{equation}
where $\chi_{\sfg}$ is an operator form of tensor index raising, which again reduces to the local Christoffel symbol formula \cite[section V.22.4]{Michor} when $\Pi$ is the identity map. This example is thus a Hamiltonian approach to \cite{Nguyen2020a}, where we applied this formula to problems in Riemannian optimization.

For a function $f(q)$ on $\cQ$ (a potential function) and a mass metric $\sfg$, the Euler-Lagrange equation \cref{eq:geodesic} for a {\it metric-potential} Hamiltonian of the form
\begin{equation}H(q, p) = \frac{1}{2}p . \sfg(q)^{-1}p + f(q)\label{eq:QKP}\end{equation}
is known to be $\ddot{q} + \Gamma(q; \dot{q}, \dot{q}) + \rgrad_f(q)=0$ \cite{AbrahamMarsden}, where $\rgrad_f$ is the Riemannian gradient of $f$, which also follows from the Hamilton method. This gives a global coordinate format of the Euler-Lagrange equation that is simple to use.

In general, a projection function $\Pi$ gives rise to a {\it torsion-free} affine connection (\cref{sec:curvature}), given by the Christoffel function $\Gamma^{\Pi}(\xi_1, \xi_2) = -\Pi'(q; \xi_1)\xi_2$. The projection gives a canonical way to extend a tangent vector at a point in $\cQ$ to a vector field on $\cQ$, allowing us to evaluate tensors defined by vector fields. In \cref{theo:cur_embed}, we express the curvature of an affine connection for tangent vectors $\xi_1, \xi_2$ and bundle vector $\xi_3$ in terms of the basis-independent global Christoffel function
$$\Rc(\xi_1, \xi_2)\xi_3 = \rD_{\xi_1}\Gamma(\xi_2, \xi_3) -
\rD_{\xi_2}\Gamma(\xi_1, \xi_3) + \Gamma(\xi_1, \Gamma(\xi_2, \xi_3)) - \Gamma(\xi_2, \Gamma(\xi_1, \xi_3)).
$$
This formula is the operator version of the well-known curvature formula \cite[section V.24.3]{Michor} using Christoffel symbols, but it also generalizes the Gau{\ss}-Codazzi equation \cite[section V.26.4]{Michor}(see \cref{rem:GaussCodazzi}). For an alternative method, we proved an {\it affine Gau{\ss}-Codazzi equation} for vector bundles using conjugate connections in \cref{subsec:secondfund}. Parts of the O'Neill tensors could be considered second fundamental forms in this setting, we apply this equation in \cref{theo:crossgrass}.

If we have a Riemannian submersion $\qq:\cQ\to\cB$ \cite{ONeill1966}, it is well-known the geometry of the manifold $\cB$ largely corresponds to the geometry of the horizontal bundle on $\cQ$, for example, geodesics on $\cB$ lift to horizontal geodesics on $\cQ$, there is a relationship between the curvatures in O'Neill's formula (ibid.).

We show if $\qq:\cQ\to\cB$ is a differentiable submersion with the vertical subbundle $\cV$, and $T\cQ$ splits to a direct sum $T\cQ=\cH\oplus\cV$ of horizontal and vertical bundles, we can lift a Hamilton vector field on $T^*\cB$ to a {\it horizontal Hamilton vector field} on $\cH^*$, and also provide an explicit formula for the Hamilton vector field similar to \cref{eq:HamFlow}. This is related to symplectic reduction. We suspect some of this is known, but could not locate a reference, thus, we believe a clear formulation of this lift is useful in applications.

For a Riemannian submersion, in \cref{prop:curvlift}, we prove a global formula for the {\it lift of the curvature tensor} in terms of the Christoffel function of the {\it induced Levi-Civita connection} on $\cH$ and {\it O'Neill's $\sA$-tensor}.

Our results could be used to derive several differential geometric expressions in both pure and applied mathematics. We provide examples of rigid body mechanics and the Grassmann manifolds. Our main application is to compute differential geometric measures of the {\it Kim-McCann metric} of a $\log$-type cost function, an analog of the {\it reflector antenna } cost  on the manifold of {\it positive semidefinite fixed rank matrices}. The rank-one case could be considered as a projective version of the classical model, cost functions of this format are related to information geometry \cite{WongInfo}. The non-commutativity of the novel higher-rank case makes the computation more difficult; the ability to verify the formulas numerically is helpful in deriving our result. We show in this case the cross-curvature is nonnegative (the optimal map is actually explicit). We also study the same cost function for {\it Grassmann manifolds}. Here, the higher-rank case does not have nonnegative cross-curvature. On the other hand, the Riemannian distant cost \cite{Loeper} for a Grassmann manifold has nonnegative cross-curvature on the diagonal, and we hope the method described in \cref{sec:reflectGrass} to be useful. Besides providing new examples of cost functions with nonnegative cross-curvature, these examples illustrate the framework provided here computationally.

Our method greatly simplifies the argument, if not the computation. We believe the global method will be useful in applications such as robotic or molecular dynamics, in instances where the projection could be computed effectively but not necessarily in closed form.

An important feature of the global coordinate approach advocated here is the use of {\it directional derivatives} on {\it operator-valued} functions for most formulas in this paper. With modern numerical libraries, taking directional derivative and composing Christoffel functions are done easily by several {\it automatic differentiation libraries}, we illustrate the examples with numerical implementations in \cite{NguyenVerify}.

\subsection{Outline of the paper}\label{subsec:outline}
In the next subsection, we introduce the main notations. In \cref{sec:BundleEmbeded}, we show cotangent bundle embeddings are characterized by projection functions, prove the Weingarten lemma, and describe the structure of tangent and cotangent bundles of a vector bundle. In \cref{sec:Hamechanics} we derive the equations for a Hamilton vector field. In \cref{sec:curvature} we describe {\it connections} and {\it connection maps} in global coordinates and derive the global curvature formula and the embedded geodesic equation. We study lifts of Hamilton vector fields for submersion in \cref{sec:HLift} and give a global lifting formula for Riemannian curvature with O'Neill's tensors. We show several application examples of our global formulas in \cref{sec:examples}: the {\it classical rigid body dynamics}, computing various geometric-mechanical quantities for the {\it Grassmann manifold}, and a {\it Kim-McCann metric on fixed-rank matrix manifolds}. We have a final discussion in \cref{sec:Discussion}. In the appendices, we prove a metric tensor of an embedded manifold extends to a metric operator.

\subsection{Notations}Without indicating otherwise, we consider an embedded manifold $\cQ\subset \cE$, where $\cE$ is a Euclidean space identified with $\R^n, n > 0$ is an integer. The inner product will be denoted by ``$.$'', or $\langle,\rangle_{\cE}$. In applications, the manifold $\cQ$ is often the solution set of a global constraint $\rC(q) = 0$ for $q\in\cQ$, where $\rC$ is a function from $\cE$ to another real vector space $\cE_L$ identified with $\R^k$, $k$ being the number of independent constraints. We will assume the Jacobian $\rC'(q)$ is of full rank for $q\in\cQ$, thus $\dim\cQ = n-k$. Our approach also works when a manifold is given parametrically.

The space of matrices of dimension $n\times m$ is denoted $\R^{n\times m}$ for $n,m\in\N$. We denote the space of linear maps between vector spaces $V, W$ by $\Lin(V, W)$. The differential of a map $f$ is denoted by $f'$, thus, the directional derivative at $q\in\cQ$ along the tangent direction $\xi\in T_q\cQ$ is denoted $f'(q; \xi)=f'(q)\xi$, the tangent space at $q$ is defined by the constraint $\rC'(q)v = 0$ on a tangent vector $v$. Similarly, the directional derivative of the projection function $\Pi$ defined below is denoted by $\Pi'(q; v)$ for $v\in T_q\cQ$.

We also use the notation $\rD_{\xi}f(q)$ for the directional derivative $f'(q;\xi)$ in case the latter is ambiguous. For example, if $\Gamma$ (the operator version of the Christoffel symbols) is a function from $\cQ$ to $\Lin(\cE\otimes\cE, \cE)$ , we write the directional derivative of $\Gamma$ in direction $v$ in variable $q\in\cQ$ as $\rD_v\Gamma(q; e_1, e_2)=\rD_v\Gamma(e_1, e_2)\in\cE$, evaluated at $e_1, e_2\in\cE$, as appeared in the curvature formula. When not specified, we understand the derivative is in the manifold variable $q\in\cQ$.

For an operator-valued function $F:q\mapsto F(q)$, we write $F^{\ft}$ for the operator-valued function $F^{\ft}:q\mapsto F(q)^{\ft}$. We write $A_{\asym} = \frac{1}{2}(A-A^{\ft}), A_{\sym} = \frac{1}{2}(A + A^{\ft})$ for a square matrix $A$.

If $\Psi$ is an operator-valued function defined on $\cQ$ and $e_1\in\cE$ is fixed, we have a linear function $X \mapsto \Psi'(q; X)e_1$ for $X\in T_q\cQ$. We denote by $\Phi:=\{X\mapsto \Psi'(q; X)e_1\}^{\ft}e_2$ an expression in $\cE$ such that for all $\Delta\in T_q\cQ, e_2\in\cE$
\begin{equation}\Delta . \{X\mapsto \Psi'(q; X)e_1\}^{\ft}e_2  = e_2 . \Psi'(q; \Delta)e_1.\label{eq:PsiX}
\end{equation}
Since $\Delta\mapsto e_2 . \Psi'(q; \Delta)e_1$ is a linear functional on $T_q\cQ$, there is a unique $\Phi(q, e_1, e_2)\in T_q\cQ$ satisfying the above, as the pairing is nondegenerate in $T_q\cQ$. If we only require $\Phi(q, e_1, e_2)\in\cE$, it is not unique (but easier to find). However, $\Pi(q)\Phi\in T_q\cQ$ is unique. This is an example of {\it index raising} that we will use in the paper. This type of index raising appeared in \cite{LeeLeokM}, it is the global version of the local orthogonal basis index raising.

\section{Preliminaries: bundles and embedded manifolds}\label{sec:BundleEmbeded}
Let $\cF$ be an inner product space with ``$.$'' denoting its inner product. Let $V\subset\cF, W\subset\cF$ be two subspaces. The pairing between $V$ and $W$ is nondegenerate if $v .W = 0$ for $v\in V$ implies $v=0$, and $V. w = 0$ for $w\in W$ implies $w=0$. Write this condition in terms of bases $\{v_i\}_{i=1}^{\dim_V}$  of $V$ and $\{w_j\}_{j=1}^{\dim_W}$ of $W$, the pairing matrix $M=(v_i .w_j)_{i,j}$ has rows and columns linearly independent, thus it is invertible, and any linear functional on either space is representable by a vector in the other space, thus $W$ could be identified with the dual space $V^*$ of $V$, and vice versa.

We call the two subspaces $V$ and $W$ of $\cF$ with nondegenerate pairing a pair of dual subspaces. They define a pair of adjoint affine projections. Recall an affine projection on $\cF$ is an idempotent linear operator $\Pi\in\Lin(\cF, \cF)$, $\Pi^2=\Pi$. If $\Imag(\Pi) = V$ then $\Pi$ is called a projection onto $V$. If $A$ is a linear operator on $\cF$, $A^{\ft}$ is the unique operator such that $A v. w = v. A^{\ft}w$ for $v, w\in \cF$.

\begin{lemma}\label{lem:proj1}Let $V$ and $W$ be a pair of dual subspaces of $\cF$ with $k = \dim V=\dim W$, so the pairing $V . W$ is nondegenerate. Then there exists a unique affine projection $\Pi\subset \Lin(\cF, \cF)$ such that $\Pi\cF = V, \Pi^{\ft}\cF=W$. If $\{v_i| i=1\cdots k\}$ is a basis of $V$, $\{w_i|i=1\dots k\}$ is the dual basis in $W$ (so $v_i.w_j = \delta_{ij}$, the Kronecker's delta), then for $f\in\cF$ we have
  $\Pi f = \sum_{i=1}^k (f. w_i)v_i$, $ \Pi^{\ft} f = \sum_1^k (f. v_i)w_i$.

  Conversely, given a subspace $V\subset \cF$ and a projection $\Pi$ onto $V$. Then $V$ and $\Pi^{\ft}\cF$ form a pair of dual subspaces.
  
   For $f\in\cE$, if $\Pi$ is a projection onto $V$ then
   \begin{equation}\Pi f = 0\quad\Leftrightarrow\quad v_*. f = 0\text{ for all }v_*\in\Pi^{\ft}\cF. \label{eq:Piq0}
\end{equation}  
\end{lemma}
\begin{proof} If $V, W$ is a dual pair with dual bases $\{v_i\}\{w_j\}, i,j=1\cdots k$ of $V, W$ as specified. Then $\Pi\cF = V, \Pi^{\ft}\cF = W$ implies the $j$-th coefficient of $\Pi f$ in the basis $\{v_i\}$ is $\Pi f . w_j = f . \Pi^{\ft}w_j = f. w_j$. Since $\Pi \cF = V$, this uniquely defines $\Pi f$, and $\Pi f = \sum_{i=1}^k (f. w_i)v_i=:v_f$. For $f, f_1\in \cE$
  $$(\Pi f) . f_1 = (\sum_{i=1}^k (f. w_i)v_i) . f_1 = \sum_{i=1}^k (f. w_i)(v_i . f_1) = f . \sum_{i=1}^k (v_i . f_1)w_i.$$
  Thus, $\Pi^{\ft}:f_1 \mapsto \sum_{i=1}^k (v_i . f_1)w_i$, and $v_i. w_j = \delta_{ij}$ implies $\Pi^2=\Pi$. The map $\Pi$ does not depend on bases, if the basis $\{v_i\}$ changes by a matrix $C=(c_{il})_{i,l=1}^k$,  $v_i=\sum c_{il}\tilde{v}_l$, ($\{\tilde{v}_l\}$ denotes the new basis), then $\{w_i\}$ relates to the new dual basis $\{\tilde{w}_m\}$ by $D=(d_{im})_{i,m=1}^k$, with $D^TC=I_k$ so
  $$\Pi f = \sum_{i=1}^k (f.w_i) v_i = \sum_{i, m, l} d_{im} (f. \tilde{w}_m) c_{il}\tilde{v}_l = \sum_{i=1}^k (f. \tilde{w}_l) \tilde{v}_l.$$

Let $w\in W=\Pi^{\ft}\cF$, then $w. V=\{0\}$ implies $w.f = \Pi^{\ft}w . f = w. \Pi f = 0$ for $f \in \cF$, thus $w=0$. Switching $V$ and $W$, if $v .W =0$ for $v\in V$ then $v=0$, thus we have a dual pair.

Equation (\ref{eq:Piq0}) follows from $v_*. f = \Pi^{\ft}v_*. f = v_*. \Pi f$ for $v_*\in W$.
\end{proof}
Assuming a vector bundle $\cW$ over $\cQ$ is embedded in the trivial bundle $\cQ\times\cF$ for a vector space $\cF$ , which we will equip with an inner product (such embeddings always exist for manifolds, see the map $\psi$ in the proof of \cite[Theorem III.8.11]{Michor}). We denote by $\cW_q$ the fiber over $q\in\cQ$.
\begin{lemma}\label{lem:proj2}1) Let $\cQ\subset \cE$ be a manifold and $\cF$ be an inner product space. Let $\cW$ and $\cW^{\dagger}$ be two subbundles of $\cQ\times\cF$ such that $(\cW_q, \cW_q^{\dagger})$ forms a dual pair for $q\in\cQ$. Then we can identify $\cW^{\dagger}$ with $\cW^*$, the dual bundle of $\cW$. If $\Pi(q)$ is the corresponding projection in \cref{lem:proj1} then $\Pi: q\mapsto \Pi(q)$ is a smooth function from $\cQ$ to $\Lin(\cF, \cF)$.

2) Let $\Pi:\cE\to \Lin(\cF,\cF)$ be a smooth map, such that for each $q\in\cQ$, $\Pi(q)$ is an affine projection of constant rank $k$ onto $\cW_q:=\Imag(\Pi(q))$. Then the set
  \begin{equation}\cW = \{(q, w)\in \cQ\times \cF|w\in\Imag(\Pi(q))\}
  \end{equation}
  is a subbundle of $\cQ\times \cE$ over $\cQ$. Dually,
  \begin{equation}\cW^{\dagger} = \{(q, w_*)\in \cQ\times \cF| w_*\in \Imag(\Pi(q)^{\ft}) \}
  \end{equation}
is a subbundle of $\cQ\times \cF$, which could be identified with $\cW^*$.
\end{lemma}
\begin{proof} For 1), $\Pi$ is well-defined from \cref{lem:proj1}, the basis independent part of that lemma guarantees $\Pi$ is independent of coordinate choice. It is smooth as the bundle projection maps $\pi_{\cW}:\cW\to\cQ$ and $\pi_{\cW^{\dagger}}:\cW^{\dagger}\to\cQ$ are smooth in each bundle chart.

  For 2), let $\pi_{\cW}$ be the map $\pi_{\cW}:(q, w)\mapsto q$ from $\cW$ to $\cQ$, then $\pi_{\cW}^{-1}(q) = \Imag(\Pi(q))$. To show $(\cW, \pi_{\cW})$ defines a vector bundle, choose a basis $B_q = \{v_i |i=1\cdots l\}$ of $\Imag(\Pi(q))$ at $q\in\cQ$, and let $\cU_q$ be an open set around $q$ such that $\Pi(y)B_q$ is linearly independent for $y\in \cU_q$, thus $\Pi(y)B_q$ is a basis of $\Imag(\Pi(y))$, and assume $\psi_{\cU_q}:\cU_q\to\R^{\dim\cQ}$ is a coordinate function (narrow the domain of $\cU_q$ if necessary). Set
  $$\hat{\psi}_{\cU_q}: \pi_{\cW}^{-1}(\cU_q)\mapsto\R^{\dim\cQ}\times\R^{\dim\Imag \Pi(q)},\;\;\;
  \hat{\psi}_{\cU_q}(y, \sum_i c_i\Pi(y)v_i) = (\psi_{\cU}(y), (c_i)_{i=1}^l)
  $$
then $\hat{\psi}_{\cU_q}$ is a coordinate function for $\pi_{\cW}^{-1}(\cU_q)$. The collection $(\cU_q, \hat{\psi}_{\cU_q})_{q\in\cQ}$ defines a bundle atlas \cite[section III.8.1]{Michor} as the transition functions are linear, giving $\cW$ a vector bundle structure. Replacing $\Pi$ by $q\mapsto \Pi(q)^{\ft}$, then $\cW^{\dagger}$ is a subbundle, duality follows from \cref{lem:proj1}.
\end{proof}
Since the inner product pairing is positive-definite on all subspaces, the pairing between $\cW$ to itself is nondegenerate, thus $\cW$ could be identified with its dual. In this case, $\Pi=\Pi^{\ft}$ is the orthogonal projection from $\cE$ to $\cW$.

As $\Pi$ is an operator-valued function, we can take directional derivatives. We have a generalization of the classical Weingarten lemma \cite[section V.26.1]{Michor} (as $\Pi'$ is the {\it second fundamental form} in the Riemannian embedding $\cQ\subset \cE$, see \cref{subsec:secondfund}).

\begin{lemma}[affine Weingarten] Assume the bundle $\cW\subset \cQ\times\cF$ is defined by the projection function $\Pi$ as in \cref{lem:proj2}. For $f\in \cW_q$, $\Delta_q\in T_q\cQ$
  \begin{equation}\label{eq:Wein}
    \Pi(q)\Pi'(q; \Delta_q)f = 0.
  \end{equation}
If $\cW = T\cQ$, $p \in T^*\cQ, v\in T\cQ$ and define $\{X\mapsto \Pi'(q; X)^{\ft}p\}^{\ft}$ by \cref{eq:PsiX} then
\begin{equation}
  \Pi(q)^{\ft}\{X\mapsto \Pi'(q; X)^{\ft}p\}^{\ft}v = 0.\label{eq:WeinX}
\end{equation}  
\end{lemma}
\begin{proof}Differentiate $\Pi(q)^2f = \Pi(q) f$, then use $\Pi(q)f = f$ to simplify the below
  $$\Pi'(q, \Delta_q)\Pi(q)f + \Pi(q)\Pi'(q, \Delta_q)f = \Pi'(q, \Delta_q) f.$$
  to \cref{eq:Wein}. Thus, $f_*. \Pi'(q; \Delta_q)f  = f_*. \Pi(q)\Pi'(q; \Delta_q)f  =0$ if $f_*\in \cW^*$, or
$$0 = \Pi'(q; \Delta_q)^{\ft}f_*. f = \Delta_q . \{X\mapsto \Pi'(q; X)^{\ft}f_*\}^{\ft} f.$$
This implies \cref{eq:WeinX}, if $\cW=T\cQ, f = v, f_*=p$.
\end{proof}  
\begin{lemma}\label{lem:higherTangent} Assume the subbundles $\cW, \cW^*\subset \cQ\times\cF$ are defined by the projection functions $\Pi_{\cW}$ and $\Pi_{\cW}^{\ft}$, and the tangent and cotangent bundles of $\cQ$ are defined by $\Pi$ and $\Pi^{\ft}$. The fibers of the tangent bundles of $\cW$ and $\cW^*$ at $(q, f)\in \cW, (q, f^*)\in \cW^*$ are defined by
\begin{gather}
    T_{q, f}\cW = \{(\Delta_q, \Delta_f)\in T_q\cQ\times\cF|\Pi_{\cW}'(q; \Delta_q)f + \Pi_{\cW}(q)\Delta_f = \Delta_f
    \}\label{eq:TW},\\
    T_{q, f^*}\cW^* = \{(\Delta_q, \Delta_{f^*})\in T_q\cQ\times\cF|
    \Pi_{\cW}'(q; \Delta_q)^{\ft}f^* + \Pi_{\cW}^{\ft}(q)\Delta_{f^*} = \Delta_{f^*}
  \}\label{eq:TWs}.
\end{gather}
Define the projection functions from $\cE\times\cF$ to $T_{q, f}\cW, T_{q, f^*}\cW^*,
T^*_{q, f}\cW, T^*_{q, f^*}\cW^*$ 
\begin{gather}
  \Pi_{T\cW}(q, f)(\omega, \phi) =   (\Pi(q)\omega, \Pi_{\cW}'(q; \Pi(q)\omega)f + \Pi_{\cW}(q)\phi)\label{eq:piTW},\\
  \Pi_{T\cW^*}(q, f^*)(\omega, \phi) =   (\Pi(q)\omega, \Pi_{\cW}'(q; \Pi(q)\omega)^{\ft}f^* + \Pi_{\cW}(q)^{\ft}\phi),\label{eq:piTWs}\\
  \Pi_{T^*\cW}(q, f)(\omega, \phi) =   (\Pi(q)^{\ft}\{\omega + \{\Delta\mapsto \Pi_{\cW}'(q; \Delta)f\}^{\ft}\phi\}, \Pi_{\cW}(q)^{\ft}\phi)\label{eq:pisTW},\\
  \Pi_{T^*\cW^*}(q, f^*)(\omega, \phi) = (\Pi(q)^{\ft}\{\omega + \{\Delta\mapsto \Pi_{\cW}'(q; \Delta)f^*\}^{\ft}\phi\}, \Pi_{\cW}(q)\phi)
\end{gather}
for $(\omega, \phi)\in \cE\times \cF$ with $\Pi_{T\cW}(q, f)^{\ft} = \Pi_{T^*\cW}(q, f)$, $\Pi_{T\cW*}(q, f^*)^{\ft} = \Pi_{T^*\cW^*}(q, f^*)$. Thus, the cotangent bundles $T^*\cW$ and $T^*\cW^*$ could be embedded in $(\cE\times\cF)^2$ with fibers at $(q, f)\in \cW$ and $(q, f^*)\in\cW^*$ as
\begin{gather}
    T_{q, f}^*\cW = (p_q, p_f)|\Pi(q)^{\ft}p_q = p_q,\Pi_{\cW}^{\ft}(q)p_f = p_f  \}\label{eq:TsW},    \\
    T_{q, f^*}^*\cW^* = (p_q, \Delta_f)| \Pi(q)^{\ft}p_q = p_q, \Pi_{\cW}(q)\Delta_f = \Delta_f  \}.    \label{eq:TsWs}
\end{gather}
\end{lemma}
\begin{proof}Differentiate the defining equation $\Pi(q)f = f$ with variables $(q, f)$ in direction $(\Delta_q, \Delta_f)\in T_{q, f}\cW$, gives us the constraints on the right-hand side of \cref{eq:TW}. Conversely, assume the constraint $\Delta_f = \Pi'(q, \Delta_q)f + \Pi(q)\Delta_f$ is satisfied for $\Delta_q \in T\cQ, \Delta_f\in\cF$. Consider a curve $\gamma(t)$ on $\cQ$ with $\gamma(0) = q$ and $\dot{\gamma}(0) = \Delta_q$. Let $\beta$ be the curve on $\cW$ defined by $\beta(t) = (\gamma(t), \Pi(\gamma(t))(f+t\Delta_f))$, then
  $$\dot{\beta}(0) = (\Delta_q, \Pi'(q, \Delta_q)f + \Pi(q)\Delta_f) = (\Delta_q, \Delta_f),$$
therefore $(\Delta_q, \Delta_f)\in T_qW$. This shows \cref{eq:TW} characterizes $T\cW$. Replace $\cW$ with $\cW^*$, \cref{eq:TWs} characterizes $T\cW^*$. We verify $\Pi_{T\cW}(q, f)(\omega, \phi)\in T_{q, f}\cW$, observing $\Pi(q)\omega \in T_q\cQ$ and
  $$\Pi_{\cW}'(q; \Pi(q)\omega)f + \Pi_{\cW}(q)\{\Pi_{\cW}'(q; \Pi(q)\omega)f + \Pi_{\cW}(q)\phi\} = \Pi_{\cW}'(q; \Pi(q)\omega)f + \Pi_{\cW}(q)\phi
  $$
  as $\Pi_{\cW}(q)\Pi_{\cW}'(q; \Pi(q)\omega)f=0$ by \cref{eq:Wein}. For $(\Delta_q, \Delta_f)\in T_{q, f}\cW$ we have
      $$  \Pi_{T\cW}(q, f)(\Delta_q, \Delta_f) =   (\Pi(q)\Delta_q, \Pi_{\cW}'(q; \Pi(q)\Delta_q))f + \Pi_{\cW}(q)\Delta_f) = (\Delta_q, \Delta_f)
  $$
  by the defining equation \cref{eq:TW}. This confirms $\Pi_{\cW}$ is a projection function to $T\cW$. Let $\omega_1\in\cE, \phi_1\in\cF$, then
  $$\begin{gathered}\Pi_{T\cW}(q, f)(\omega, \phi) . (\omega_1, \phi_1) = \Pi(q)\omega . \omega_1 +
    \Pi_{\cW}'(q; \Pi(q)\omega)f .\phi_1 + \Pi_{\cW}(q)\phi .\phi_1\\
    =\omega .\Pi(q)^{\ft}\omega_1 + \Pi(q)\omega .\{\Delta\mapsto \Pi_{\cW}'(q; \Delta)f \}^{\ft} \phi_1 +\phi. \Pi_{\cW}(q)^{\ft}\phi_1
  \end{gathered}$$
  which gives us the expression for $\Pi_{T\cW}(q, f)^{\ft}(\omega_1, \phi_1)$, and the expression for $\Pi_{T\cW^*}(q, f)^{\ft}$ follows similarly.

For \cref{eq:TsW}, from \cref{eq:pisTW}, $\Imag(\Pi_{T\cW}(q, f)^{\ft})\subset T_q^*\cQ\times \cW^*$, the equality follows from a dimension count. Replace $\cW$ with $\cW^*$ we get \cref{eq:TsWs}.
\end{proof}
\begin{remark}For $\cW=T\cQ$, if $\sfg$ is a {\it metric operator}, that means for $q\in\cQ$, $\sfg(q)$ is an invertible, symmetric operator on $\cE$, and {\it the pairing between $\sfg(q)T_q\cQ$ and $T_q\cQ$ is nondegenerate}, the projection function in \cref{lem:proj2} is called $\Pi_{\sfg}$. If $\cQ$ is defined by a constraint $\rC(q)=0$ where $\rC$ is a map from $\cE$ to $\cE_L$ of full rank then $\rC'(q)\sfg(q)^{-1}\rC'(q)^{\ft}$ is invertible, as if $\rC'(q)\sfg(q)^{-1}\rC'(q)^{\ft}a=0$ then $\sfg(q)^{-1}\rC'(q)^{\ft}a\in T_\cQ$, and
  $$\sfg(q)\eta .\sfg(q)^{-1}\rC'(q)^{\ft}a=\eta .\rC'(q)^{\ft}a=\rC'(q)\eta .a=0$$
  for all $\eta\in T_q\cQ$, thus $\sfg(q)^{-1}\rC'(q)^{\ft}a=0$, hence $\rC'(q)^{\ft}a=0$ and $a=0$ by the full-rank assumption. Note, {\it $\sfg$ is not assumed to be definite}. Thus, we can take
  \begin{equation}\Pi_{\sfg}(q)\omega = \omega - \sfg(q)^{-1}\rC'(q)^{\ft}(\rC'(q)\sfg(q)^{-1}\rC'(q)^{\ft})^{-1}\rC'(q)\omega\label{eq:ProjG}\end{equation}
as we see $\rC'(q)\Pi_{\sfg}(q)=0$ and $\sfg(q)\Pi_{\sfg}(q) = \Pi_{\sfg}(q)^{\ft}\sfg(q)$, identifying $\Imag(\Pi(q))$ with $\sfg(q)T_q\cQ$. A special case is when $\sfg(q)$ is the identity map, in this case $\Pi=\Pi^{\ft}$. This projection appears often in applications. The condition $\sfg(q)\Pi_{\sfg}(q) = \Pi_{\sfg}(q)^{\ft}\sfg(q)$ is called {\it metric compatible}.

Alternatively, for a local chart $(U, \phi)$ for $U\subset \cQ$, $\phi$ maps $U$ to an open subset of $\R^{d}$ ($d=\dim\cQ)$ then $\phi'(q)$ maps $T_q\cQ$  to $\R^{d}$. If we define $\psi(q)b=\phi'(q)^{-1}b\in T_q\cQ\subset\cE$ for $b\in\R^d$ and consider $\psi(q)$ as a map from $\R^d$ to $\cE$, with range $T_q\cQ$, then $\psi(q)^{\ft}$ maps $\cE$ to $\R^d$. The projection is given by
\begin{equation}\Pi_{\sfg}(q)\omega = \psi(q)(\psi(q)^{\ft}\sfg(q)\psi(q))^{-1}\psi(q)^{\ft}\sfg(q)\omega\label{eq:ProjG2}\end{equation}
as we can easily check $\psi(q)^{\ft}\sfg(q)\psi(q)$ is invertible if $\psi$ is of full rank and $\Pi_{\sfg}$ is onto $T_q\cQ$. A change of coordinate changes $\psi(q)$, but does not change $\Pi_{\sfg}$.

We make two observations that will be useful in \cref{sec:KimMcCann}. We note the restriction of a non-degenerate non-definite pairing $\sfg$ on $\cE$ to a subspace $\cV$ is nondegenerate if and only if there is a $\sfg$-compatible projection to $\cV$.  The arguments above prove the ``only if'' part. The ``if'' part is clear as if for $v\in\cV$, $v.\sfg v_1 = 0$ for all $v_1\in \cV$ then $v. \sfg e = v. \sfg \Pi e =0$ for all $e\in\cE$, where $\Pi$ is the metric-compatible projection. This shows $v=0$.

Also, let $\cV^{\perp_{\sfg}}$ to be the subspace of vectors orthogonal to $\cV$, then in the non-definite case $\cV^{\perp_{\sfg}}\cap \cV$ could be nonzero. However, the restriction of $\sfg$ to $\cV$ is nondegenerate if and only if $\cV^{\perp_{\sfg}}\cap \cV=\{0\}$, and this is equivalent to the restriction to $\cV^{\perp_{\sfg}}$ is nondegenerate.

If $\cQ$ is the unit sphere $q^{\ft}q = 1$ in $\cE$, and $\mathsf{a}\in\cE$ with $\|\mathsf{a}\| < 1$, then
  \begin{equation}\label{eq:PiExotic}
\Pi_{\sfa}(q):\omega\mapsto    \omega - (q + \sfa)\frac{q^{\ft}\omega}{1+q^{\ft} \sfa}
  \end{equation}
is a more exotic example of a projection to the tangent bundle, projecting $\omega$ to $T_q\cQ$ along direction $q+\sfa\neq 0$, with adjoint
$$  \Pi_{\sfa}(q)^{\ft}:\omega\mapsto    \omega - q\frac{(q+\sfa)^{\ft}\omega}{1+q^{\ft}\sfa}.$$
The cotangent space $T^*_q\cQ$ corresponds to $p\in\cE$ satisfying $(q+\sfa)^{\ft}p=0$.
\end{remark}  

\section{Tangent, cotangent bundle, and Hamiltonian mechanics}\label{sec:Hamechanics}
Consider a submanifold $\cQ\subset \cE$. We will fix a smooth projection function $\Pi$ onto $T\cQ$, and consider the embedding of $T^*\cQ\subset \cE^2$ defined by $\Pi^{\ft}$. We will express geometric and mechanical quantities on $\cQ$ using $\Pi$. Our main differential geometry reference is \cite{Michor}.

A {\it vector field} is a section of $T\cQ$, it could be considered as a function $\ttX$  from $\cQ$ to $\cE$ such that $\Pi(q)\ttX(q) = \ttX(q)$. For two vector fields $\ttX, \ttY$, the {\it Lie bracket} is
\begin{equation}[\ttX, \ttY] = \rD_{\ttX}\ttY - \rD_{\ttY}\ttX.\label{eq:lie_bracket}
\end{equation}
Here, at $q\in\cQ$, $\rD_{\ttX}\ttY$ is the directional derivative of $\ttY$ in direction $\ttX(q)$ evaluated at $q$, where $\ttY$ is considered as a map from $\cQ$ to $\cE$, so its directional derivative makes sense as an $\cE$-valued function, and similarly for $\rD_{\ttY}\ttX$. Since $\cQ$ is a submanifold, $[\ttX, \ttY]$ is a vector field, $\Pi[\ttX, \ttY] = [\ttX, \ttY]$.
\begin{lemma}[Torsion-free]\label{lem:torsionfree}Let $\xi, \eta\in T_q\cQ$ at fixed $q\in\cQ$. Define extensions of $\xi$ and $\eta$ to vector fields on $\cQ$ by $v_{\xi}(x)=\Pi(x)\xi, v_{\eta}(x) = \Pi(x)\eta$, $x\in\cQ$. Then
  \begin{equation}[v_{\xi}, v_{\eta}]_{x=q} =  \Pi'(q, \xi)\eta - \Pi'(q, \eta)\xi=0.
\label{eq:TorsionFree}\end{equation}
\end{lemma}
\begin{proof}By \cref{eq:lie_bracket}, $[v_{\xi}, v_{\eta}](q) =  \Pi'(q, \xi)\eta - \Pi'(q, \eta)\xi$. Since $[v_{\xi}, v_{\eta}](q)\in T_q\cQ$
  $$\Pi'(q, \xi)\eta - \Pi'(q, \eta)\xi = \Pi(q)(\Pi'(q, \xi)\eta - \Pi'(q, \eta)\xi) = 0$$
using the Weingarten lemma \cref{eq:Wein}.
\end{proof}
A corollary is if $(q, v, \Delta_q, \Delta_v)\in TT\cQ\subset\cE^4$ then $(q, \Delta_q, v, \Delta_v)\in TT\cQ\subset\cE^4$ as
$$\Delta_v = \Pi(q)\Delta_v + \Pi'(q;\Delta_q)v =\Pi(q)\Delta_v + \Pi'(q;v)\Delta_q.$$
Thus, the map $(q, v, \Delta_q, \Delta_v)\mapsto (q, \Delta_q, v, \Delta_v)$ maps $T_{q, v}T\cQ$ to $T_{q,\Delta_q}T\cQ$ is a bundle map, the {\it canonical flip} \cite[section III.8.13]{Michor} of $TT\cQ$.

A {\it $1$-form} on $\cQ$ is a section of $T^*\cQ$, or a function $\ttP:\cQ\to\cE$ satisfying
\begin{equation} \Pi^{\ft}(q)\ttP(q) = \ttP(q).\end{equation}
In this setting, the differential $df$ of a scalar function $f$ on $\cQ$ is represented as $q\mapsto\Pi(q)^{\ft}\egrad_f(q)$, with $\egrad_f$ is the gradient of $f$ as a function on $\cE$, as for a tangent vector $\xi\in T_q\cQ$
$$\Pi(q)^{\ft}\egrad_f(q).\xi = \egrad_f(q).\xi = f'(q;\xi) = df(q).\xi.$$

A {\it two-form} $\omega$ is an anti-symmetric pairing of two vector fields on $\cQ$. Identifying an $1$-form $\ttP$ and two  vector fields $\ttX, \ttY$ with functions from $\cQ$ to $\cE$, the {\it exterior derivative} \cite[III.9.8]{Michor} $d\ttP$ is a {\it two-form} evaluates on $\ttX, \ttY$ as
$$\begin{gathered} d\ttP(\ttX, \ttY) = \rD_{\ttX}(\ttP .\ttY) - \rD_{\ttY}(\ttP.\ttX) -\ttP([\ttX, \ttY])\\
 = (\rD_{\ttX}\ttP) .\ttY + \ttP. (\rD_{\ttX}\ttY) -
 (\rD_{\ttY}\ttP).\ttX - (\rD_{\ttX}\ttP).\ttY -\ttP.(\rD_{\ttX}\ttY - \rD_{\ttY}\ttX).
\end{gathered}$$
\begin{equation}\label{eq:dP}
d\ttP(\ttX, \ttY) = (\rD_{\ttX}\ttP) .\ttY - (\rD_{\ttY}\ttP).\ttX.
\end{equation}

With $\cW= T\cQ$, a vector field $\Delta$ on the manifold $T^*\cQ$ could be identified with a function $\Delta = (\Delta_q, \Delta_p)$ from $T^*\cQ$ to $\cE^2$. At $(q, p)\in T^*\cQ$, \cref{eq:TWs} requires
$$\Pi(q)\Delta_q(q, p) = \Delta_q(q, p);\quad\quad \Pi(q)^{\ft}\Delta_p(q, p) + \Pi'(q, \Delta_q)^{\ft}p = \Delta_p(q, p).
$$
A $1$-form on $T^*\cQ$ could be identified with a function $\ttP = (\ttP_q, \ttP_p)$ from $T^*\cQ$ to $\cE^2$ such that $(q, p, \ttP_q(q, p), \ttP_p(q, p))\in\cE^4$ satisfying \cref{eq:TsWs}, or $\ttP_q\in T_q^*\cQ, \ttP_p\in T_q\cQ$. The quadruple $(q, p, p, 0)$ satisfies this condition, and the {\it tautological 1-form} on $T^*\cQ$ corresponds to
\begin{equation} \theta:(q, p)\mapsto (p, 0)\in T_{q, p}^*T^*\cQ\subset\cE^2\end{equation}
which evaluates on a vector field $\ttX = (\ttX_q, \ttX_p)$ on $T^*\cQ$ as $(q, p)\mapsto p. \ttX_q(q, p)$. For two vector fields $\ttX = (\ttX_q, \ttX_p), \ttY = (\ttY_q, \ttY_p)$ on $T^*\cQ$, from \cref{eq:dP}, the {\it Poincar{\'e} $2$-form} is
$$\Omega(\ttX, \ttY) = -d\theta(\ttX, \ttY) = -(\rD_{\ttX_q, \ttX_p}(p, 0)).(\ttY_q, \ttY_p) + (\rD_{\ttY_q, \ttY_p}(p, 0)). (\ttX_q, \ttX_p) 
$$
  \begin{equation}
\Omega(\ttX, \ttY) = - \ttX_p.\ttY_q + \ttX_q.\ttY_p = \ttX_q.\ttY_p - \ttX_p.\ttY_q.
  \end{equation}
Thus, $\Omega$ is induced by the Poincar{\'e} $2$-form on $\cE^2$. We prove below the {\it symplectic pairing} $\Omega$ is nondegenerate on $T^*\cQ$, as well-known.
\begin{theorem}[Hamilton vector field]\label{thm:HVF} Assume $\cQ$ is embedded in $\cE$ and $T\cQ\subset \cQ\times\cE$ is represented by a projection function $\Pi$. For $(h_q, h_p)\in \cE^2$, set
  \begin{equation}\begin{gathered}e_q = \Pi(q)h_p,\\
e_p = \Pi'(q; e_q)^{\ft}p - \Pi(q)^{\ft}\{h_q +  \{X\mapsto \Pi'(q; X)^{\ft}p\}^{\ft}h_p \}
    \end{gathered}\label{eq:etangent}
  \end{equation}
then $(e_q, e_p)\in T_{q, p}T^*\cQ$ and for all $(\xi_q, \xi_p)\in T_{q, p}T^*\cQ$
\begin{equation}
(h_q, h_p).(\xi_p, \xi_q) =
h_q.\xi_q + h_p.\xi_p = e_q. \xi_p - e_p . \xi_q
= \Omega((e_q, e_p), (\xi_p, \xi_q)).\label{eq:OmgProj}
\end{equation}
Restricted to $T_{q, p}^*T^*\cQ$, the map $(h_q, h_p)\mapsto (e_q, e_p)$ in \cref{eq:etangent} has an inverse
\begin{equation} (e_q, e_p)\mapsto(s_q, s_p) := (-\Pi(q)^{\ft}e_p, e_q) \in T_{q, p}^*T^*\cQ\label{eq:OmgInv}
\end{equation}
hence, $(e_q, e_p)$ is the unique element in $T_{q, p}T^*\cQ$ satisfying \cref{eq:OmgProj}.

Thus, for a (Hamiltonian) function $H$ on $T^*\cQ$, extended to a neighborhood in $\cE^2$ with gradient $(H_q, H_p)$, the unique vector field on $T^*\cQ$ satisfying $\Omega(\ttX_H, \ttY) = \rD_{\ttY} H$ for all vector field $\ttY$ on $T^*\cQ$ is the Hamilton vector field below
\begin{equation}\label{eq:XH}
\ttX_H(q, p) =\begin{bmatrix}\Pi(q)H_p\\
    \Pi'(q; \Pi(q)H_p)^{\ft}p - \Pi(q)^{\ft}\{H_q +  \{X\mapsto \Pi'(q; X)^{\ft}p\}^{\ft}H_p \}\end{bmatrix}.
\end{equation}
  \end{theorem}
  \begin{proof}From \cref{eq:TWs}, $(q, p, e_q, e_p)\in T_{q, p}T^*\cQ$. For \cref{eq:OmgProj}
    $$\begin{gathered}e_q. \xi_p - e_p . \xi_q = \Pi(q)h_p .\xi_p -
(\Pi'(q; e_q)^{\ft}p - \Pi(q)^{\ft}\{h_q +  \{X\mapsto \Pi'(q; X)^{\ft}p\}^{\ft}h_p \}).\xi_q\\
      = h_p . \Pi(q)^{\ft} \xi_p +\Pi(q)^{\ft} h_q . \xi_q + \{X\mapsto \Pi'(q; X)^{\ft}p\}^{\ft}h_p .\Pi(q)\xi_q\\
      = h_p . \Pi(q)^{\ft} \xi_p + h_q . \xi_q + h_p.\Pi'(q; \Pi(q)\xi_q)^{\ft}p =h_p .\xi_p + h_q .\xi_q
    \end{gathered}$$
    where we use $\Pi'(q; e_q)^{\ft}p .\xi_q = p . \Pi'(q; e_q)\xi_q = 0$ by \cref{eq:Wein}, then the definition of $\{X\mapsto \Pi'(q; X)^{\ft}p\}^{\ft}h_p$ next, and $\xi_p = \Pi(q)^{\ft} \xi_p + \Pi'(q; \xi_q)^{\ft}p$ last.
    
    Conversely, given $(e_q, e_p)\in T_{q, p}T^*\cQ$, substitute \cref{eq:OmgInv} to the right-hand side of \cref{eq:etangent}, we recover $e_q, e_p$, as $\Pi(q)s_p = s_p = e_q$, while
using \cref{eq:WeinX,eq:Wein} and the defining equation of $T_{q, p}T^*\cQ$
$$\begin{gathered}\Pi'(q; e_q)^{\ft}p - \Pi(q)^{\ft}\{s_q + \{X\mapsto \Pi'(q; X)^{\ft}p\}^{\ft}s_p \} = \Pi'(q; e_q)^{\ft}p + \Pi(q)^{\ft}e_p = e_p\end{gathered}.$$
Thus, if the $\Omega$-pairing in \cref{eq:OmgProj} is zero for all $(\xi_q, \xi_p)\in T_{q, p}T^*\cQ$ then $(s_q, s_p)\in T_{q, p}^*T^*\cQ$ from \cref{eq:OmgInv} is zero, hence $(e_q, e_p)=0$, therefore $\Omega$ is  nondegenerate. The statement on $\ttX_H$ now follows from the definition of the gradient $(H_q, H_p)$.
  \end{proof}  
  The flow equation of the Hamilton vector field $\ttX_H$ is the {\it Hamilton equation}
  \begin{equation} (\dot{q}, \dot{p}) = \ttX_H(q, p).\label{eq:HamiltonFlow}\end{equation}
  \begin{remark}Lagrange mechanics is well-covered in \cite{LeeLeokM}. For Hamilton mechanics, the authors use the notation $\{\frac{\partial P^T(x)\mu}{\partial x} \}^T$ for our
    $\{X\mapsto \Pi'(x, X)^{\ft}p\}^{\ft}$. In equation 2.8 ibid., the term
    $P(x)^T\{\frac{\partial P^T(x)\mu}{\partial x} \}^TP(x)$ is zero by \cref{eq:WeinX} and the condition $\Pi=\Pi^{\ft}$ assumed there is not needed. The extension of a metric tensor to an operator in \ref{appx:Extend} satisfies $\Pi=\Pi^{\ft}$, however, it is convenient to work without this restriction.
  \end{remark}
\begin{example}For the sphere, to use the projection \cref{eq:PiExotic} in \cref{eq:XH}, note
  $$
    \begin{gathered}
      \Pi_{\sfa}'(q; X)^{\ft}p = -q \frac{ X^{\ft} p}{1 + q^{\ft}\sfa}\;\text{ for } X\in T_q\cQ,\\
      \{X\mapsto \Pi_{\sfa}'(q, X)^{\ft}p\}^{\ft}\omega = - \frac{\omega . q}{1 + q^{\ft}\sfa}p.\end{gathered}$$
from $\omega. \Pi_{\sfa}'(q, X)^{\ft}p =X . \frac{\omega . q}{1 + q^{\ft}\sfa}p$. When $\sfa =0$, $\Pi_{\sfa}(q) = \Pi_{\sfa}(q)^{\ft} = I - qq^{\ft}$, we recover \cite[equation 3.4]{LeeLeokM}
\begin{equation}\ttX_H = ((I - qq^{\ft}) H_p(q, p), -(I - qq^T)H_q(q, p) +(pq^{\ft} - qp^{\ft})H_p(q, p)).\label{eq:Hamiltonsphere}
\end{equation}    
\end{example}

\section{Affine connections and curvature of embedded manifolds}\label{sec:curvature}
The connection here is described in terms of a bundle map. As in \cite[Chapter 4]{Michor}, we describe a connection as a split of $T\cW$ to subbundles related to the sequence in \cref{eq:exact_seq} below. Let $\pi_T:T\cQ\to\cQ$, $\pi_T(q, \Delta_q) = q$ for $q\in \cQ, \Delta_q\in T_q\cQ$ be the tangent bundle projection. Let $\cW$ be a subbundle of $\cQ\times \cF$ for an inner product space $\cF$, and $\pi_{\cW}:\cW \to\cQ$ be the vector bundle projection $\pi_{\cW}(q, f) = q$ for $f\in \cW_q$. Define two (pullback) bundles over $\cW$, $\pi_{\cW}^*\cW$ and $\pi_T^*\cQ$
\begin{gather}
  \pi_{\cW}^*\cW :=\{(q, f, w)| q\in\cQ, (f, w)\in W_q^2 \subset \cF^2\},\\
  \pi_T^* T\cQ := \{(q, f, \Delta_q)|q\in\cQ, f\in W_q, \Delta_q\in T_q\cQ \}.
\end{gather}
We have a short exact sequence \cite[Vol. I, chap. 3, ex. 29]{Spivak} of bundles over $\cW$
\begin{equation}\label{eq:exact_seq}
  0\to\pi_{\cW}^*\cW\xrightarrow{\imath_{\cW}} T\cW\xrightarrow{\bar{\pi}} \pi_T^* T\cQ\to 0
\end{equation}
where $\imath_{\cW}$ maps $(q, f, w)$ to $(q, f, 0, w)\in T_{q, f}\cW$, and $\bar{\pi}$ maps $(q, f, \Delta_q, \Delta_f)$ satisfying \cref{eq:TW} to $(q, f, \Delta_q)$. The short exact sequence means $\bar{\pi}$ is onto, $\imath_{\cW}$ is injective and $\Imag(\imath_{\cW}) = \Null(\bar{\pi})$, which is verified directly. From \cref{eq:TW}
$$(\Delta_q, \Delta_f) = (0, \PiW(q)\Delta_f) + (\Delta_q, \PiW'(q; \Delta_q)f)\in \Imag(\imath_{\cW; q, f}) + \Null(\rC^{\PiW}_{q, f})$$
where $\rC^{\PiW}_{q, f}:(\Delta_q, \Delta_f) \mapsto \PiW(q)\Delta_f = \Delta_f - \PiW'(q; \Delta_q)f$ maps $T_{q, f}\cW$ to $\cW_q$, and $\Null(\rC^{\PiW}_{q, f})$ consists of elements of the form $(\Delta_q, \PiW'(q; \Delta_q)f)\in T_q\cQ\times\cF$. Note, $\rC^{\PiW}$ is a left-inverse of $\imath_{\cW}$, $\rC^{\PiW}_{q, f}\circ \imath_{\cW; q, f}= Id_{(\pi_{\cW}^*\cW)_{q,f}}$. Thus, $T_{q, f}\cW$ {\it splits} to a direct sum of $\Imag(\imath_{\cW; q, f})$ and $\Null(\rC^{\PiW}_{q, f})$. Another splitting with a decomposition
$$(\Delta_q, \Delta_f) = (0, \PiW(q)\Delta_f + \mrGamma(q; \Delta_q, f)) + (\Delta_q, \PiW'(q; \Delta_q)f - \mrGamma(q; \Delta_q, f))$$
requires $\mrGamma(q; \Delta_q, f)\in \cW_q$ for the first term to be in $\Imag(\imath_{\cW; q, f})$, while the second is in the null space of $\rC_{q, f}$, a map from $T_{q, f}W$ to $W_q$ of the form
\begin{equation}\rC_{q, f}(\Delta_q, \Delta_f) = \Delta_f - \PiW'(q; \Delta_q)f +\mrGamma(q; \Delta_q, f) = \PiW(q)\Delta_f + \mrGamma(q; \Delta_q, f)\label{eq:mrGamma0}
\end{equation}
with $\mrGamma$ is linear in $\Delta_q\in T_q\cW$, and in $f$. Then $\rC_{q, f}\circ \imath_{\cW; q, f}= Id_{(\pi_{\cW}^*\cW)_{q,f}}$. We will assume smoothness for $\mrGamma$. We call $\rC$ a {\it connection map} or {\it connector}. The term
\begin{equation}\label{eq:Gamma}
\Gamma(q; \Delta_q, f) = - \PiW'(q; \Delta_q)f +\mrGamma(q; \Delta_q, f)
\end{equation}
is called a {\it Christoffel function} \cite{Edelman_1999}. Here, $\Gamma$ is defined for $(\Delta_q, f)\in T_q\cQ\times \cW_q$. We will assume it extends bilinearly to $\cE\times\cF$ (this is always possible by projecting to $T_q\cQ\times\cW_q$ before applying $\Gamma$, but it is often simpler to just extend the algebraic expressions). Thus, $q\mapsto\Gamma(q)$ is a smooth map from $\cQ$ to $\Lin(\cE\otimes\cF, \cF)$ (the space of bilinear maps from $\cE\times\cF$ to $\cF$) characterized by
\begin{equation}
\rC((q, f, \Delta_q, \Delta_f)= \Delta_f +\Gamma(q; \Delta_q, f)\in \cW_q\text{ for }(\Delta_q, \Delta_f)\in T_{q, f}\cW.
\end{equation}
 {\it The covariant derivative below does not depend on the extension of $\Gamma$ to a bilinear map from $\cE\times \cF$ to $\cE$}. In practice, for example, in the curvature calculation below, we can simplify an expression for $\Gamma$ assuming $\Delta_q\in T_q\cQ$ and $f\in \cW_q$, then extend that expression algebraically. With $\Gamma$, hence $\rC$, we can define a {\it connection}, or {\it covariant derivative}, allowing us to take derivatives of a {\it section} $s$. Identify $s$ with a map $s:\cQ\to\cW\subset\cF$ (thus, $\PiW(q)s(q) = s(q)$), so its differential $\rD s$ maps $T\cQ$ to $\cT\cW$,
$$\rD s: (q, \Delta_q)\mapsto (q, s(q), \Delta_q, s'(q; \Delta_q)),$$
and the right-hand side satisfies \cref{eq:piTW}. In direction $\Delta_q$
\begin{equation}\label{eq:covar}
  (\nabla_{\Delta_q} s)(q) := (\rC\circ \rD s)(q, \Delta_q) = s'(q, \Delta_q) +\Gamma(q; \Delta_q, s(q))\in \cW_q.
\end{equation}
If $c$ is a vector field then $\nabla_c s$ is a section, evaluated at $q$ by \cref{eq:covar} with $\Delta_q = c(q)$. In particular, for $\rC=\rC^{\PiW}$ the covariant derivative is
\begin{equation}(\nabla_{c}^{\PiW}s)(q) = s'(q, c(q)) -\PiW'(q; c(q)) s(q) = \PiW(q)s'(q; c(q)).
\end{equation}
The definition of $\Gamma$ in terms of $\mrGamma$ depends on a particular choice of $\PiW$. Replacing $\PiW$ by $\Pi_{\cW,1}$ shifts $\Gamma(\Delta_q, f)$ to  $\Gamma(\Delta_q, f) + \PiW'(\Delta_q, f) -  \Pi'_{\cW,1}(\Delta_q, f)$.

To recover the bundle map $\rC$ from  a covariant derivative $\nabla^1$ in the local coordinate definition, given by Christoffel symbols organized to a bilinear map $\Gamma^1$ from $T_q\cQ\times \cW_q$ to $\cF$, consider a curve $c(t)$ with $c(0) = q, \dot{c}(0) =\Delta_q$ and consider the section $s(t) = \Pi_{\cW}(c(t))(f + t\Delta_f)$ over the curve $c(t)$. Then we can check
$$ \cW_q\ni(\nabla^1_{\dot{c}(t)}s(t))_{t=0} = \Pi'(q, \Delta_q)f + \Pi(q)\Delta_f + \Gamma^1(q; \Delta_q, f) = \Delta_f + \Gamma^1(q; \Delta_q, f).
$$
Thus, $\Gamma^1(q; \Delta_q, f)$ is a connection in our setup and the above defines a connection map.

For $q\in\cQ$, given two tangent vectors $\xi, \eta\in T_q\cQ$ and $w\in \cW_q$, the curvature of $\nabla$ is a tensor, i.e. for arbitrary extensions of $\xi, \eta$ to vector fields $\vxi, \veta$ and section $s_w$ near $q$, the element of $\cW_q$ defined below is independent of the choice of the extension
\begin{equation} (\Rc_{\xi\eta }w)_q := (\nabla_{\vxi} \nabla_{\veta} s_w - \nabla_{\veta} \nabla_{\vxi} s_w
- \nabla_{[\vxi, \veta]} s_w )_q.\end{equation}

\begin{theorem}\label{theo:cur_embed} Let $\cQ$ be a submanifold of $\cE$ and $\cW$ is a subbundle of $\cQ\times\cF$. Let $\nabla$ be a connection with Christoffel function $\Gamma:\cQ\to \Lin(\cE\otimes \cF,\cF)$. For $\xi, \eta \in T_q\cQ$, and $w \in \cW_q$, the curvature of $\nabla$ is given by
    \begin{equation}\label{eq:rc1}
      \Rc_{\xi,\eta}w = (\rD_{\xi}\Gamma)(\eta, w) - (\rD_{\eta}\Gamma)(\xi, w) + \Gamma(\xi, \Gamma(\eta, w)) - \Gamma(\eta, \Gamma(\xi, w)).
    \end{equation}
    Here, all expressions are evaluated at $q$, the directional derivatives are in $q$. Let $\Pi_{\cW}$ be a projection to $\cW$. If $\mathring{\Gamma}(q; \xi, w) =  \Pi'_{\cW}(q;\xi)w +\Gamma(q; \xi, w)$ as in \cref{eq:mrGamma0}, where $\Pi'_{\cW}$ is valid for $w\in\cF$ (not just on $\cW$),  then 
    \begin{equation}\label{eq:rc2}
      \Rc_{\xi,\eta}w = (\rD_{\xi}\mrGamma)(\eta, w) -
      (\rD_{\eta}\mrGamma)(\xi, w) + \Gamma(\xi,\Gamma(\eta, w)) -  \Gamma(\eta,\Gamma(\xi, w)).
    \end{equation}
\end{theorem}
\begin{proof} Let $\Pi$ be a projection to $T\cQ$. Let $\vxi(y) = \Pi(y)\xi, \veta(y)= \Pi(y)\eta$, and extend $w$ to a section $s_w(y) = \Pi_{\cW}(y)w$ for $y\in\cQ$. By \cref{eq:TorsionFree}, $[\vxi, \veta]_q = 0$ and
$$\begin{gathered}
    (\Rc_{\xi\eta }w)_q = (\nabla_{\vxi} \nabla_{\veta} s_w - \nabla_{\veta} \nabla_{\vxi} s_w- \nabla_{[\vxi, \veta]} s_w)_q\\
    =\{\rD_{\vxi}(\rD_{\veta}s_w +\Gamma(\veta, s_w))\}_q + \Gamma(\vxi, (\rD_{\veta}s_w +\Gamma(\veta, s_w)))_q -\\
    \{\rD_{\veta}(\rD_{\vxi}s_w +\Gamma(\vxi, s_w))\}_q - \Gamma(\veta, (\rD_{\vxi}s_w +\Gamma(\vxi, s_w)))_q\end{gathered}
$$
We have $(\rD_{\vxi}\rD_{\veta}s_w - \rD_{\veta}\rD_{\vxi}s_w)_q = (\rD_{[\vxi, \veta]}s_w)_q = 0$, thus
\begin{equation}\label{eq:rcm_alt}
  \begin{gathered}(\Rc_{\xi\eta }w)_q = \{\rD_{\vxi}(\Gamma(\veta, s_w))\}_q  + \Gamma(\vxi, (\rD_{\veta}s_w +\Gamma(\veta, s_w)))_q \\
- \{\rD_{\veta}(\Gamma(\vxi, s_w))\}_q - \Gamma(\veta, (\rD_{\vxi}s_w +\Gamma(\vxi, s_w)))_q.\end{gathered}\end{equation}
As $\Gamma$ is bilinear, we have
$$\rD_{\vxi}(\Gamma(\veta, s_w)) = (\rD_{\vxi}\Gamma)(\veta, s_w) + \Gamma(\rD_{\vxi}\veta, s_w) + \Gamma(\veta, \rD_{\vxi}s_w)
$$
Substitute the above and the similar combination for $(\veta, \vxi, s_w)$ to \cref{eq:rcm_alt}, the terms $\Gamma(\rD_{\vxi}\veta, s_w)_q$ and $\Gamma(\rD_{\veta}\vxi, s_w)_q$ from the similar combination cancel as $(\rD_{\vxi}\veta)_q = (\rD_{\veta}\vxi)_q$ by \cref{lem:torsionfree}. Thus
$$\begin{gathered}(\Rc_{\xi\eta }w)_q = (\rD_{\vxi}\Gamma)(\veta, s_w)_q + \Gamma(\veta, \rD_{\vxi}s_w)_q + \Gamma(\vxi, (\rD_{\veta}s_w +\Gamma(\veta, s_w)))_q\\
- (\rD_{\veta}\Gamma)(\vxi, s_w)_q - \Gamma(\vxi, \rD_{\veta}s_w)_q
  - \Gamma(\veta, (\rD_{\vxi}s_w +\Gamma(\vxi, s_w)))_q\\
= (\rD_{\vxi}\Gamma)(\veta, s_w)_q  + \Gamma(\vxi, \Gamma(\veta, s_w))_q -
(\rD_{\veta}\Gamma)(\vxi, s_w)_q
  - \Gamma(\veta, \Gamma(\vxi, s_w))_q
\end{gathered}
$$
or \cref{eq:rc1}. From $(\rD_{\xi}\rD_{\eta}\Pi_{\cW})w - (\rD_{\eta}\rD_{\xi}\Pi_{\cW})w =0$, \cref{eq:rc2} follows.
\end{proof}
\begin{remark}\label{rem:GaussCodazzi}If $\mrGamma=0, \Gamma(q;\xi,\eta)=-\Pi'(q;\xi)\eta$, \cref{eq:rc2} is equivalent to
\begin{equation}(\Rc_{\xi, \eta}w)_q . w_* = \Pi'(q;\eta) w.\Pi'(q;\xi)^{\ft}w_* - \Pi'(q;\xi)w.\Pi'(q; \eta)^{\ft}w_*\label{eq:GaussCod}\end{equation}
for a vector $w_*\in\cW^*$. If $\cW=T\cQ$ and $\Pi$ is given by the Euclidean projection to $T\cQ$ (thus, is self-adjoint), this is the Gau{\ss}-Codazzi equation (see \cref{subsec:secondfund}).
\end{remark}  

\begin{example}\label{expl:sphere}For the unit sphere $q^{\ft}q = 1$ in $\R^n$ with the projection $\Pi(q) = \dI_n - qq^{\ft}$, for two tangent vectors $\xi, \eta$, $\Pi'(q; \xi)\eta = -q\xi^{\ft}\eta$, $\Gamma^{\Pi}(q; \xi, \eta) = q\xi^{\ft}\eta$. The curvature of $\Gamma=\Gamma^{\Pi}$ at three tangent vectors $\xi, \eta, \phi$ is
  $$\begin{gathered}\Rc_{\xi,\eta}\phi = \rD_{\xi}\Gamma(\eta, \phi) - \rD_{\eta}\Gamma(\xi, \phi) + \Gamma(\xi, \Gamma(\eta, \phi)) -\Gamma(\eta, \Gamma(\xi, \phi)) \\
    = \xi\eta^{\ft}\phi - \eta\xi^{\ft}\phi + q\xi^{\ft}(q\eta^{\ft}\phi)^{\ft}
    - q\eta^{\ft}(q\xi^{\ft}\phi)^{\ft}=\xi\eta^{\ft}\phi - \eta\xi^{\ft}\phi.
  \end{gathered}$$
%It is well-known this also follows from the Gau{\ss}-Codazzi equation (\ref{eq:GaussCod}).
To use \cref{eq:rc2} with $\mrGamma=0$, we need $\Gamma=\mrGamma+\Pi'$ on $\cE$, or $\Gamma(q; \xi, \omega) =   q\xi^{\ft}\omega + \xi q^{\ft}\omega$
$$\begin{gathered}
 \Gamma(\xi, \Gamma(\eta, \phi)) =
q\xi^{\ft}(q\eta^{\ft}\phi + \eta q^{\ft}\phi) + \xi q^{\ft}(q\eta^{\ft}\phi + \eta q^{\ft}\phi)= \xi\eta^{\ft}\phi\\
\Rc_{\xi,\eta}\phi = \Gamma(\xi, \Gamma(\eta, \phi)) - \Gamma(\eta, \Gamma(\xi, \phi))=\xi\eta^{\ft}\phi - \eta\xi^{\ft}\phi.
\end{gathered}$$
\end{example}
Consider the case $\cW= T\cQ\subset \cQ\times\cE$ for $\cQ\subset\cE$. For two vector fields $\ttX, \ttY$ on $\cQ$, the torsion tensor of a connection $\nabla$ on $T\cQ$ is defined as
\begin{equation}Tor_{\nabla}(X, Y) = \nabla_{\ttX}\ttY - \nabla_{\ttY}\ttX - [\ttX, \ttY]
\end{equation}
From \cref{lem:torsionfree}, for any projection function $\Pi$ on $T\cQ$, $\nabla^{\Pi}$ is {\it torsion-free}.
\subsection{Conjugate connections, second fundamental forms, and the affine Gau{\ss}-Codazzi equation}\label{subsec:secondfund}
We start out with a more general setting. Consider a vector bundle $\cV$ over $\cQ$, let $\cV^*$ be the dual bundle and let ``$.$'' denote the pairing between $\cV$ and $\cV^*$. Let $\nabla^{\cV}$ be an affine connection on $\cV$. The conjugate connection $\nabla^{\cV^*}$ is the unique connection on $\cV^*$ satisfying
\begin{equation}s.\nabla^{\cV^*}_{\ttX}s_* =\rD_{\ttX}(s.s_*) - \nabla^{\cV}_Xs . s_*.\label{eq:definDual}
\end{equation}
for a vector field $\ttX$ and sections $s$ of $\cV$, $s_*$ of $\cV^*$. For example, if $\cV$ and $\cV^*$ are identified as subbundles of $\cQ\times\cF$ as before, $\rC^{\cV}$ is the connection map $\rC^{\cV}_{q, f}(\Delta_q, \Delta_f) = \Delta_f + \Gamma(q; \Delta_q, f)$ and $\Pi$ is the projection of $\cV$, then the conjugate connection map satisfies
\begin{equation}
    \rC_{q, f^*}^{\cV^*}(\Delta_q, \Delta_{f_*})= \Delta_{f_*} -\Pi'(q, \Delta_q)^{\ft}f_* - \Pi(q)^{\ft}\{h \mapsto \Gamma(q; \Delta_q, h)\}^{\ft}f_*\in \cV^*.
\end{equation}    
%with $\Gamma(q; \Delta)^{\ft}f_*:=\{h \mapsto \Gamma(q; \Delta_q, h)\}^{\ft}f_*$ is the adjoint of $\Gamma(q; \Delta): f_* \mapsto \Gamma(q; \Delta, f_*)$
Denote by $\Rc^{\cV}$ and $\Rc^{\cV^*}$ the respective curvatures.
\begin{theorem}[affine Gau{\ss}-Codazzi]\label{theo:GaussCod}Assume $\nabla^{\cV}$ and $\nabla^{\cV^*}$ are conjugate connections of $\cV$ and $\cV^*$. Let $\Pi_{\cW}$ be an affine projection from $\cV$ to a subbundle $\cW$, so $\Pi_{\cW}$ is a bundle map from $\cV$ to itself such that $\Imag(\Pi_{\cV}(q)) = \cW_q$, $q\in\cQ$. Then $\Pi_{\cW}^{\ft}$ is an affine projection on $\cV^*$, the pairing between $\cW^*:=\Imag(\Pi_{\cW}^{\ft})$ and $\cW$ is nondegenerate. Define $\Pi_{\cW^*}:=\Pi_{\cW}^{\ft}$. For a vector field $\ttX$ and sections $s$, of $\cW\subset\cV$, $s_*$ of $\cW^*\subset\cV^*$ define
\begin{gather}  
  \nabla_{\ttX}^{\cW}s := \Pi_{\cW}\nabla_{\ttX}^{\cV}s,\label{eq:defineNablaTW}\\
  \Two_{\cW}(\ttX, s) :=\nabla_{\ttX}^{\cV}s - \Pi_{\cW}\nabla_{\ttX}^{\cV}s,\label{eq:defineTwo}\\
  \nabla_{\ttX}^{\cW^*}s_* := \Pi_{\cW^*}\nabla_{\ttX}^{\cV^*}s_*,\label{eq:defineConj}\\
  \Two_{\cW^*}(\ttX, s_*) :=\nabla_{\ttX}^{\cV^*}s_* - \Pi_{\cW^*}\nabla_{\ttX}^{\cV^*}s_*,  \label{eq:defineTwoSt}
\end{gather}
Then $\nabla^{\cW}$ and $\nabla^{\cW^*}$ are conjugate connections on $\cW$ and $\cW^*$, $\Two_{\cW}$ and $\Two_{\cW^*}$ are tensors, called the affine second fundamental forms. Let $\cN_{\cW}$ and $\cN_{\cW^*}$ be the kernels of $\Pi_{\cW}$ and $\Pi_{\cW^*}$ respectively, then for fixed $\ttX$, $\Two_{\cW}$ could be considered a map from $\cW$ to $\cN_{\cW}$, $\Two_{\cW^*}$ is a map from $\cW^*$ to $\cN_{\cW^*}$. Let $\Rc^{\cW}$ be the curvature of $\nabla^{\cW}$, then for two vector fields $\ttX, \ttY$
\begin{gather}
  \Rc^{\cV}_{\ttX\ttY}s . s_*= - s . \Rc^{\cV^*}_{\ttX\ttY}s_*,\label{eq:RcT}\\
\Rc^{\cV}_{\ttX,\ttY}s . s_*= \Rc^{\cW}_{\ttX\ttY}s . s_* + \Two_{\cW}(\ttX, s) . \Two_{\cW^*}(\ttY, s_*)-\Two_{\cW}(\ttY, s) . \Two_{\cW^*}(\ttX, s_*).\label{eq:affineGauss}
\end{gather} 
\end{theorem}
\begin{proof}A number of statements have been proved for $\cV=\cQ\times\cF$ above, the proofs extend to general $\cV$, and we will prove the new statements. We have $\nabla_{\ttX}^{\cW}s.s_* = \nabla_{\ttX}^{\cV}s.s_*$ if $s, s_*$ are sections of $\cW, \cW^*$, which implies \cref{eq:definDual} is satisfied for $\nabla^{\cW}$ and $\nabla^{\cW^*}$, thus they are conjugate connections. We verify $\Pi_{\cW}\Two_{\cW}=0=\Pi_{\cW^*}\Two_{\cW^*}$ directly from \cref{eq:defineTwo,eq:defineTwoSt}. Equation \ref{eq:RcT} is well-known \cite{Nielsen} and implies the well-known fact in information geometry that the conjugate of a flat (zero-curvature) connection is flat. The proof is below
  $$\begin{gathered}\Rc^{\cV}_{\ttX\ttY}s . s_*=\nabla^{\cV}_{\ttX}\nabla^{\cV}_{\ttY}s.s_*-\nabla^{\cV}_{\ttY}\nabla^{\cV}_{\ttX}s.s_* - \nabla^{\cV}_{[\ttX,\ttY]}s.s_*=
    \rD_{\ttX}(\nabla^{\cV}_{\ttY}s.s_*) -\nabla^{\cV}_{\ttY}s. \nabla^{\cV^*}_{\ttX}s_*\\
    -\rD_{\ttY}(\nabla^{\cV}_{\ttX}s.s_*) + \nabla^{\cV}_{\ttX}s. \nabla^{\cV^*}_{\ttY}s_* - \rD_{[\ttX,\ttY]}(s.s_*) + s.\nabla^{\cV^*}_{[\ttX,\ttY]}s_*\\
    =\rD_{\ttX}\rD_{\ttY}(s.s_*)
    - \rD_{\ttX}(s.\nabla^{\cV^*}_{\ttY}s_*)
    -\rD_{\ttY}( s. \nabla^{\cV^*}_{\ttX}s_*)
    +s. \nabla^{\cV^*}_{\ttY}\nabla^{\cV^*}_{\ttX}s_*
    \\
    -\rD_{\ttY}\rD_{\ttX}(s.s_*)
    + \rD_{\ttY}(s.\nabla^{\cV^*}_{\ttX}s_*)
    +\rD_{\ttX}( s. \nabla^{\cV^*}_{\ttY}s_*)
    -s. \nabla^{\cV^*}_{\ttX}\nabla^{\cV^*}_{\ttY}s_*
    - \rD_{[\ttX,\ttY]}(s.s_*) + s.\nabla^{\cV^*}_{[\ttX,\ttY]}s_*\\
    =s. \nabla^{\cV^*}_{\ttY}\nabla^{\cV^*}_{\ttX}s_*
    -s. \nabla^{\cV^*}_{\ttX}\nabla^{\cV^*}_{\ttY}s_*
     + s.\nabla^{\cV^*}_{[\ttX,\ttY]}s_*=-s.\Rc^{\cV^*}_{\ttX\ttY}s_*.
  \end{gathered}$$
Using the fact that $\Two_{\cW}$ is sent to zero by $\Pi_{\cW}$, thus paired by zero with $\cW_q^*$, $q\in\cQ$, the first two lines in the proof above give
    $$\begin{gathered}\Rc^{\cV}_{\ttX\ttY}s . s_*=
    \rD_{\ttX}(\nabla^{\cW}_{\ttY}s.s_*) -(\nabla^{\cW}_{\ttY}s +\Two_{\cW}(\ttY, s)). (\nabla^{\cW^*}_{\ttX}s_*+\Two_{\cW^*}(\ttX, s_*))\\
    -\rD_{\ttY}(\nabla^{\cW}_{\ttX}s.s_*) + (\nabla^{\cW}_{\ttX}s+\Two_{\cW}(\ttX, s)). (\nabla^{\cW^*}_{\ttY}s_* +\Two_{\cW^*}(\ttY, s_*)) -  \nabla^{\cW}_{[\ttX,\ttY]} s.s_*\\
    =\rD_{\ttX}(\nabla^{\cW}_{\ttY}s.s_*) -\nabla^{\cW}_{\ttY}s . \nabla^{\cW^*}_{\ttX}s_* -\Two_{\cW}(\ttY, s). \Two_{\cW^*}(\ttX, s_*)\\
    -\rD_{\ttY}(\nabla^{\cW}_{\ttX}s.s_*) + \nabla^{\cW}_{\ttX}s. \nabla^{\cW^*}_{\ttY}s_* + \Two_{\cW}(\ttX, s). \Two_{\cW^*}(\ttY, s_*) -  \nabla^{\cW}_{[\ttX,\ttY]} s.s_*\\
    =    \nabla^{\cW}_{\ttX}\nabla^{\cW}_{\ttY}s.s_*-\nabla^{\cW}_{\ttY}\nabla^{\cW}_{\ttX}s.s_* - \nabla^{\cW}_{[\ttX,\ttY]}s.s_*  -\Two_{\cW}(\ttY, s). \Two_{\cW^*}(\ttX, s_*)
+ \Two_{\cW}(\ttX, s). \Two_{\cW^*}(\ttY, s_*),
\end{gathered}
$$
which gives us \cref{eq:affineGauss}.
\end{proof}  
If $\cV=\cV^*=\cQ\times \cF$, then $\Gamma^{\cV}=\Gamma^{\cF}$ is given by a $\Lin(\cE\otimes\cF, \cF)$-valued function on $\cQ$. Write $\Gamma^{\cF}(q;\xi)w$ for $\Gamma(q;\xi, w)$for $\xi\in T_q\cQ, w\in \cW_q, w_*\in\cW_q^*$, using the section $\Pi_{\cW}w$, it is easy to get expressions for $\Gamma^{\cW}$, $\Two^{\cW}$ and the conjugates, for example
\begin{gather}
  \Gamma^{\cW}(q; \xi, w)=-\Pi'_{\cW}(q;\xi)w + \Pi_{\cW}(q)\Gamma^{\cF}(q;\xi, w),\\
  \Two_{\cW}(q;\xi, w) = \Pi'_{\cW}(q;\xi)w + (\dI_{\cF}-\Pi_{\cW}(q))\Gamma^{\cF}(q;\xi, w).\label{eq:TwocW}
%  \Gamma^{\cW^*}(q; \xi, w_*)=-\Pi'_{\cW}(q;\xi)^{\ft}w_* - \Pi_{\cW}^{\ft}(q)\Gamma^{\cF}(q;\xi)^{\ft} w_*,\\
%  \Two_{\cW^*}(q;\xi, w_*) = \Pi'_{\cW}(q;\xi)^{\ft}w_* - (\dI_{\cF}-\Pi_{\cW}^{\ft}(q))\Gamma^{\cF}(q;\xi)^{\ft}w_*,\\
%  \Two_{\cW^*}^{\ft}(q;\xi, w_{\perp})=\Pi'_{\cW}(q;\xi)^{\ft}w_{\perp} - \Pi_{\cW}(q)\Gamma^{\cF}(q, \xi, w_{\perp})\text{ for }w_{\perp}\in\cN_{\cW}.
\end{gather}
%In the next section, we show the affine Gauss-Codazzi equation implies the classical one.
\subsection{The metric-potential Hamiltonian and the Levi-Civita connection}
We now consider a {\it metric operator}, a function $\sfg$ from $\cQ$ evaluates to invertible symmetric operators on $\cE$, thus for $q\in\cQ$, $\sfg(q)\in \Lin(\cE, \cE)$ is invertible and $\sfg(q)^{\ft} =\sfg(q)$. This allows us to define a new pairing, for $\omega_1, \omega_2\in \cE$
\begin{equation}
  \langle \omega_1, \omega_2\rangle_{\sfg} := \langle \omega_1,\sfg(q) \omega_2\rangle_{\cE} = \omega_1 . \sfg(q)\omega_2. 
\end{equation}
If this pairing is non-degenerate on $T_q\cQ$ (for example, if $\sfg(q)$ is positive definite), this gives $\cQ$ a semi-Riemannian metric \cite[section 3.7]{AbrahamMarsden}. We show in \ref{appx:Extend} any semi-Riemannian metric on $\cQ$ arises this way. A connection $\nabla$ is called {\it metric compatible} (self-conjugate under $\sfg$-pairing) if for three vector fields $\ttX, \ttY, \ttZ$
\begin{equation}\rD_{\ttX}\langle \ttY, \ttZ\rangle_{\sfg} =
  \langle\nabla_{\ttX} \ttY, \ttZ\rangle_{\sfg} + \langle \ttY,\nabla_{\ttX}\ttZ\rangle_{\sfg}.
\end{equation}
The fundamental theorem of (semi)Riemannian geometry says for a (semi)Riemannian metric, there exists a unique torsion-free connection, the Levi-Civita connection $\Gamma$, that is metric compatible. The geodesic is the curve on $\cQ$ satisfying the equation $\ddot{q} + \Gamma(q; \dot{q}, \dot{q}) = 0$. From the same reference, it is also the $\cQ$-component of the Hamilton flow for the metric Hamiltonian $\frac{1}{2}p.\sfg(q)^{-1} p$. We will derive an expression for the geodesic equation, hence the Levi-Civita connection.

The projection $\Pi_{\sfg}(q)$ in \cref{eq:ProjG2} satisfies
\begin{equation}
  \Pi_{\sfg}^{\ft}(q)= \sfg(q)\Pi_{\sfg}(q)\sfg^{-1}(q).\label{eq:PiTg}
\end{equation}  
This gives us the description of $T_q^*\cQ$ as $\Imag(\sfg\Pi_{\sfg}\sfg^{-1})=\Imag(\sfg\Pi_{\sfg})=\sfg(q)T_q\cQ$. This is the relationship between momentum ($p$), mass ($\sfg$) and velocity ($v$).% Velocity lives in $T_q\cQ$ while momentum in $T^*_q\cQ=\sfg(q)T_q\cQ$.

Let $f$ be a smooth function on $\cQ$, extended to $\cE$, playing the role of the potential, with gradient $\egrad_f$ in $\cE$. With $\frac{1}{2} p.\sfg(q)^{-1}p$ as the kinetic energy we consider the classical {\it metric-potential} Hamiltonian on $T^*\cQ$
\begin{equation}\label{eq:classic}
  H(q, p) = \frac{1}{2} p.\sfg(q)^{-1}p + f(q)\text{ for }(q, p)\in T^*\cQ.
\end{equation}
For $\xi, \eta\in T_q\cQ$, let $\cX_{\sfg}(\xi, \eta) = \cX_{\sfg}(q; \xi, \eta) =\{\Delta_q\mapsto \sfg'(q; \Delta_q)\eta\}^{\ft}\xi\in\cE$, so for $\Delta_q\in T_q\cQ$, 
\begin{equation}\cX_{\sfg}(\xi, \eta).\Delta_q = \xi. \sfg'(q;\Delta_q)\eta.\end{equation}
Following \cite{Edelman_1999}, define the {\it Riemannian-gradient} to be $\rgrad_f=\Pi_{\sfg}\sfg^{-1}\egrad_f$.

  \begin{proposition}\label{prop:metricPotential}
    The Hamilton equations for the Hamiltonian in \cref{eq:classic} are
\begin{gather}\label{eq:dotqge}\dot{q}=  \Pi(q)\sfg^{-1}(q)p =\sfg^{-1}(q)p,\\
\label{eq:dotpge}\dot{p} =  \Pi'(q; \sfg(q)^{-1}p)^{\ft}p +\Pi(q)^{\ft}\{\frac{1}{2}\cX_{\sfg}(\sfg^{-1}(q)p, \sfg^{-1}(q)p)- \egrad_f(q)\}
\end{gather}
where $\Pi = \Pi_{\sfg}$ in \cref{eq:ProjG}. They are equivalent to the Euler-Lagrange equation
\begin{equation}\label{eq:geodesic}\ddot{q} - \Pi'(q, \dot{q})\dot{q} + \Pi(q)\sfg(q)^{-1}\{\sfg'(q;\dot{q}) - \frac{1}{2}\cX_{\sfg}(\dot{q}, \dot{q})\} + \rgrad_f(q)= 0.
\end{equation}
\end{proposition}
\begin{proof}
  We have $H_p(q, p)=\sfg(q)^{-1} p, H_q(q, p) = -\frac{1}{2}\cX_{\sfg}(\sfg^{-1}(g)p, \sfg^{-1}p)+\egrad_f(q)$. From here, \cref{eq:dotqge} follows from \cref{eq:XH} as $\sfg^{-1}(q)p\in T_q\cQ$, or $p = \sfg(q)\dot{q}$.

Equation (\ref{eq:dotpge}) follows from \cref{eq:XH,eq:WeinX}, as $H_p=\sfg(q)^{-1}p\in T\cQ$. To derive \cref{eq:geodesic}, differentiate $\Pi^{\ft} = \sfg\Pi\sfg^{-1}$
$$\Pi'(q, \dot{q})^{\ft} = \sfg'(q, \dot{q})\Pi(q)\sfg^{-1}(q) + \sfg(q)\Pi'(q; \dot{q})\sfg^{-1}(q) - \sfg(q)\Pi(q)\sfg^{-1}(q)\sfg'(q, \dot{q})\sfg^{-1}(q).
$$
We have $\dot{p} = \sfg(q, \dot{q})\dot{q} + \sfg(q)\ddot{q}$ from $p =\sfg(q)\dot{q}$. Drop the variable $q$ in $\Pi, \sfg, \sfg^{-1}, \egrad_f$ for brevity, \cref{eq:dotpge} becomes the below, which simplifies to \cref{eq:geodesic}:
$$\begin{gathered}\sfg'(q, \dot{q})\dot{q} + \sfg\ddot{q} = \sfg'(q, \dot{q})\Pi\sfg^{-1}p + \sfg\Pi'(q; \dot{q})\sfg^{-1}p
- \sfg\Pi\sfg^{-1}(q)\sfg'(q, \dot{q})\sfg^{-1}p\\
+ \sfg\Pi\sfg^{-1}\{\frac{1}{2}\cX(\sfg^{-1}p, \sfg^{-1}p) - \egrad_f\}
\label{eq:dotpgeo}
\end{gathered}$$
using $\sfg^{-1}p=\dot{q}$, canceling the first term on each side then apply $\sfg^{-1}$.

The Euler-Lagrange equation in \cite[Proposition 3.7.4]{AbrahamMarsden} is thus \cref{eq:geodesic}.
\end{proof}
Therefore, the Christoffel function of the Levi-Civita connection satisfies
$$\Gamma(\dot{q}, \dot{q}) = -\Pi'(q, \dot{q})\dot{q} + \Pi(q)\sfg(q)^{-1}(\sfg'(q;\dot{q}) - \frac{1}{2}\cX_{\sfg}(\dot{q}, \dot{q})).$$
Equation (\ref{eq:LeviCivita}) follows by polarization, as $\Gamma$ is torsion-free.

We now prove the Gau{\ss}-Codazzi equation from the affine version. Let $\nabla^{\cE}$ be the metric-compatible connection with Christoffel function $\frac{1}{2}\sfg^{-1}(\rD_{\ttX}\sfg\ttY+\rD_{\ttY}\sfg\ttX -\chi_{\sfg}(\ttX, \ttY))$ on $\cQ\times\cE$. Let $\Two_{T\cQ}$ and $\Two_{T^*\cQ}$ be the second fundamental forms of $T\cQ$ and $T^*\cQ$ as subbundles of $\cQ\times\cE$. For a vector field $\ttX$, write $\Two(\ttX)$ for the map $\ttY\mapsto\Two_{T\cQ}(\ttX,\ttY)$ and $\Two_*(\ttX)$ for $s_*\mapsto \Two_{T\cQ^*}(\ttX, s_*)$, where $\ttY$ is a vector field and $s_*$ is a one-form. We show
\begin{equation}\Two_*(\ttX) = \sfg\Two(\ttX)\sfg^{-1}.
\end{equation}
They both map one-forms to sections of $\cN_{T^*\cQ}=\cN^*_{T\cQ}$ (recall $\cN_{T\cQ}$ and $\cN^*_{T\cQ}$ are kernels of $\Pi_{T\cQ}$ and $\Pi_{T^*\cQ}$). Let $s_{\perp}$ be a section of $\cN_{T^\cQ}$ then
$$\begin{gathered}s_{\perp}.\sfg\Two(\ttX)\sfg^{-1}s_* = s_{\perp}.\sfg(\dI_{\cE} -\Pi)\nabla^{\cE}_{\ttX} (\sfg^{-1}s_*)=s_{\perp}.(\dI_{\cE} -\Pi)^{\ft}\sfg\nabla^{\cE}_{\ttX} (\sfg^{-1}s_*)\\
=  s_{\perp}.\sfg\nabla^{\cE}_{\ttX} (\sfg^{-1}s_*)=\rD_{\ttX}(s_{\perp}.\sfg\sfg^{-1}s_*) - \nabla^{\cE}_{\ttX} s_{\perp}.s_* = s_{\perp}. \nabla^{\cE^*}_{\ttX}s_* =s_{\perp}.\Two_*(\ttX)s_*.
\end{gathered}$$
Let $\ttZ,\ttW=\sfg^{-1}s_*$ be vector fields, the Gau{\ss}-Codazzi equation follows from \cref{eq:affineGauss}
\begin{equation}
\langle\Rc^{\cE}_{\ttX,\ttY}\ttZ,\ttW\rangle_{\sfg}= \langle\Rc_{\ttX\ttY}\ttZ, \ttW\rangle_{\sfg} + \langle\Two(\ttX, \ttZ), \Two(\ttY, \ttW)\rangle_{\sfg}-\langle\Two(\ttY, \ttZ), \Two(\ttX, \ttW)\rangle_{\sfg}.\label{eq:GaussCodazzi}
\end{equation}
If we use a projection $\Pi_1$ that is not metric compatible and construct a connection using \cref{eq:defineNablaTW}, the resulting connection $\nabla^1$ is also not metric compatible. The cotangent bundle $T^*_1\cQ$ using $\Pi_1^{\ft}$ is not $\sfg T\cQ$. We can verify $\Pi_1^{\ft}\sfg$, considered as a map from $T\cQ$ to $T^*_1\cQ$ is invertible, denoted by $\ttP$, and the conjugate connection to $\nabla^1$ on $T\cQ$ is $\ttP^{-1}\nabla^{T_1^*\cQ}\ttP$, where $\nabla^{T_1^*\cQ}$ is given by \cref{eq:defineConj}. %The conjugate second fundamental form is transformed similarly.

\section{Horizontal bundles and lifts}\label{sec:HLift}
When studying systems with symmetries, the dynamic happens on a quotient manifold, which may not have a convenient embedded coordinate. To study quotients in the embedded method developed here, following \cite{Edelman_1999}, we use the concept of horizontal bundles. The symmetries (or the kernel of a submersion in general) generate vertical vector fields, we are interested in the dynamic transversal to the vertical vector fields.
\subsection{Vertical integrability and horizontal lifts of Hamilton vector fields.}
Assume $T\cQ$ and $T^*\cQ$ decompose to subbundles, $T\cQ =\cH \oplus \cV$, $T^*\cQ =\cH^* \oplus \cV^*$, where both $(\cH, \cH^*)$ and $(\cV, \cV^*)$ are pairs of dual bundles. Thus, the projection $\Pi$ to $T\cQ$ has a decomposition $\Pi = \ttH \oplus \ttV, \Pi^{\ft} = \ttH^{\ft}\oplus\ttV^{\ft}$, with $\ttH, \ttV$ are projection functions from $\cQ\times\cE$ to $\cH, \cV$, respectively, and $\ttH\ttV = \ttV\ttH = 0$. We call $\cV$ and $\cH$ the vertical and horizontal bundles.

Identify $T\cH$ and $T\cH^*$ with submanifolds of $\cE^4$ from \cref{eq:TW,eq:TWs},
  \begin{gather}T\cH=\{(q, h, \Delta_q, \Delta_h)|\;(q, h)\in \cH, \Delta_q\in T_q\cQ, \Delta_h=\ttH(q)\Delta_h + \ttH'(q, \Delta_q)h\}\\
    T\cH^*=\{(q, p_h, \Delta_q, \Delta_{p_h})|\;(q, p_h)\in \cH^*, \Delta_q\in T_q\cQ, \Delta_{p_h}=\ttH^{\ft}(q)\Delta_{p_h}+ \ttH'(q, \Delta_q)^{\ft}p_h\}\label{eq:TcHs}
  \end{gather}
$\cH^*$ is a {\it presymplectic submanifold} of $T^*\cQ$. In \cref{eq:TcHs}, decompose $\Delta_q =\epsilon_q +\delta_q$ with $\epsilon_q=\ttV(q)\Delta_q$, $\delta_q= \ttH(q)\Delta_q$, we have $T_{q, p_h}\cH^*= \cV_{q, p_h}\cH^*\oplus \cH_{q, p_h}\cH^*$ with
  \begin{gather}\cV_{q, p_h}\cH^*:=\{(\epsilon_q,  \ttB(q, p_h)\epsilon_q) | \epsilon_q\in \cV_q \},\label{eq:VHs}\\
\text{ where }  \ttB(q, p_h)\epsilon_q := - \ttH(q)^{\ft}\{X\mapsto \ttV'(q; X)^{\ft}p_h \}^{\ft}\epsilon_q +\ttH'(q; \epsilon_q)^{\ft}p_h,\label{eq:ttB}\\
    \cH_{q, p_h}\cH^*:=\{(\delta_q, \delta_{p_h}) |  \delta_q\in \cH_q, \delta_{p_h} = \ttH(q)^{\ft}\delta_{p_h} + \ttH'(q; \delta_q)^{\ft}p_h\},\\
    (\Delta_q, \Delta_{p_h}) = (\epsilon_q, \ttB(q, p_h)\epsilon_q)
    + (\delta_q, \Delta_{p_h} - \ttB(q, p_h)\epsilon_q).
  \end{gather}
Here, $\cV_{q, p_h}\cH^*$ is constructed as the kernel of the presymplectic form on $T_{q, p_h}\cH^*$:
\begin{proposition}Assume $\cV$ is {\bf integrable}, i.e. the Lie bracket of two vertical vector fields is vertical, then $\Omega(\cV_{q, p_h}\cH^*, T_{q, p_h}\cH^*)$ vanishes and
  $\cV_{q, p_h}\cH^*$ is the kernel of $\Omega$ restricted to $T_{q, p_h}\cH^*$. Restricted to $\cH_{q, p_h}\cH^*$, $\Omega$ is nondegenerated. \label{prop:Hnondegen}

The map $\epsilon_q\mapsto (\epsilon_q, \ttB(q, p_h)\epsilon_q)$ is bijective between $\cV_q$ and $\cV_{q, p_h}\cH^*$.
\end{proposition}
\begin{proof}For $(\epsilon_q, \epsilon_p), (\Delta_q, \Delta_p)\in T_{q, p_h}\cH^*$, expand
  $$\begin{gathered}\epsilon_q. \Pi(q)^{\ft}\Delta_p - \Delta_q. \Pi(q)^{\ft}\epsilon_p =
    \epsilon_q. (\ttH(q)^{\ft}\Delta_p + \Pi(q)\ttH'(q, \Delta_q)^{\ft}p_h) - \Delta_q. \Pi(q)^{\ft}\epsilon_p\\
    = \epsilon_q. \ttH(q)^{\ft}\Delta_p + \epsilon_q . (\rD_{\Delta_q}(\Pi(q)^{\ft}\ttH(q)^{\ft})p_h - \Pi'(q, \Delta_q)^{\ft}\ttH(q)^{\ft}p_h)  - \Delta_q. \Pi(q)^{\ft}\epsilon_p\\
    =\epsilon_q. \ttH(q)^{\ft}\Delta_p + \epsilon_q.( \ttH'-\Pi' )(q, \Delta_q)^{\ft}p_h - \Delta_q. \Pi(q)^{\ft}\epsilon_p=
    \\ \epsilon_q . \ttH(q)^{\ft}\Delta_p -\Delta_q . (\{X\mapsto \ttV'(q; X)^{\ft}p_h \}^{\ft}\epsilon_q + \Pi(q)^{\ft}\epsilon_p).
  \end{gathered}$$
Thus, $(\epsilon_q, \epsilon_p)$ belongs to the kernel if and only if $\epsilon_q\in\cV_q$ and
\begin{equation}\Pi(q)^{\ft}\epsilon_p = - \Pi(q)^{\ft}\{X\mapsto \ttV'(q; X)^{\ft}p_h \}^{\ft}\epsilon_q.\label{eq:VerticalHam}\end{equation}
Since $\ttH\Pi=\ttH$, this implies $\ttH(q)^{\ft}\epsilon_p = - \ttH(q)^{\ft}\{X\mapsto \ttV'(q; X)^{\ft}p_h \}^{\ft}\epsilon_q$, hence
  $$\epsilon_p =
\ttH(q)^{\ft}\epsilon_p + \ttH'(q, \epsilon_q)^{\ft}p_h=
  - \ttH(q)^{\ft}\{X\mapsto \ttV'(q; X)^{\ft}p_h \}^{\ft}\epsilon_q + \ttH'(q, \epsilon_q)^{\ft}p_h.$$
To show this choice satisfies \cref{eq:VerticalHam}, assume $\epsilon_q\in\cV_q$,  we need
  $$\begin{gathered}\Pi(q)^{\ft}\{- \ttH(q)^{\ft}\{X\mapsto \ttV'(q; X)^{\ft}p_h \}^{\ft}\epsilon_q + \ttH'(q, \epsilon_q)^{\ft}p_h\} = - \Pi(q)^{\ft}\{X\mapsto \ttV'(q; X)^{\ft}p_h \}^{\ft}\epsilon_q,\\
\Leftrightarrow  \Pi(q)^{\ft}\ttH'(q, \epsilon_q)^{\ft}p_h + \ttV(q)^{\ft}\{X\mapsto \ttV'(q; X)^{\ft}p_h \}^{\ft}\epsilon_q=0,\\
\Leftrightarrow  \ttV(q)^{\ft}\ttH'(q, \epsilon_q)^{\ft}p_h + \ttV(q)^{\ft}\{X\mapsto \ttV'(q; X)^{\ft}p_h \}^{\ft}\epsilon_q=0
\end{gathered}$$
where between the first and second line we use $-\Pi^{\ft}\ttH^{\ft} +\Pi^{\ft} = \ttV^{\ft}$, between the second and third line, we expand $\Pi(q)^{\ft} = \ttH(q)^{\ft} + \ttV(q)^{\ft}$ then use \cref{eq:Wein}. This is equivalent to the pairing below is zero for all $\phi_q\in \cV_q$
$$\begin{gathered} \phi_q . (
\ttH'(q, \epsilon_q)^{\ft}p_h + \{X\mapsto \ttV'(q; X)^{\ft}p_h \}^{\ft}\epsilon_q) =   \ttH'(q, \epsilon_q) \phi_q . p_h + \ttV'(q; \phi_q)^{\ft}p_h . \epsilon_q \\
=\ttH'(q, \epsilon_q) \phi_q . p_h + p_h . \ttV'(q; \phi_q)\epsilon_q \\
= \Pi'(q, \epsilon_q) \phi_q . p_h- \ttV'(q, \epsilon_q) \phi_q . p_h + p_h . \ttV'(q; \phi_q)\epsilon_q\\
=  p_h . (\ttV'(q; \phi_q)\epsilon_q - \ttV'(q, \epsilon_q) \phi_q) =  p_h.[\kappa_{\phi_q}, \kappa_{\epsilon_q}]_{y=q} =0
\end{gathered}$$
where $\Pi'(q, \epsilon_q) \phi_q . p_h=0$ from \cref{eq:Wein}, and $\kappa_{\omega}$ denotes the vector field $y\mapsto \ttV(y)\omega, y\in\cQ$ for $\omega\in\cE$. The final equality follows from the integrability of $\cV$, as $[\kappa_{\phi_q}, \kappa_{\epsilon_q}]_{y=q}$ is vertical and $p_h$ is horizontal.
%$$ p_h . \ttV'(q; \phi_q)\epsilon_q-p_h.\ttV'(q, \epsilon_q) \phi_q  =.\qedhere$$
\end{proof}
Elements of $\cV_{q, p_h}\cH^*$ are called {\it $\cH^*$-vertical vectors}, of $\cH_{q, p_h}\cH^*$ are {\it $\cH^*$-horizontal vectors}. This decomposition of $T_{q, p_h}\cH$ induces a bundle decomposition $T\cH^*= \cV\cH^*\oplus \cH\cH^*$. The kernel of the pairing $\Omega$ on $T\cH^*$ consists of $\cH^*$-vertical vectors. A {\it $\cH^*$-vertical vector field} is a section of $\cV\cH^*$ and a {\it $\cH^*$-horizontal vector field} is a section of $\cH\cH^*$.

For a function $G$ is on $\cH^*$, as in \cref{thm:HVF}, for a $\cH^*$-vector field $\Phi$, we want to realize $dG. \Phi$ by $\Omega(\ttZ_G, \Phi)$ for a $\cH^*$-horizontal vector field $\ttZ_G$. Since the $\Omega$-pairing is zero on $\cH^*$-vertical vectors, we will assume $dG.\Phi=0$ for $\cH^*$-vertical vectors.  We call such a function $G$ on $\cH^*$ a {\it $\cH^*$-Hamiltonian}. In \cref{eq:horep} below, given $q\in\cQ, p_h\in \cH_q, \omega_p\in\cE$, we have an element $\Psi\in \cE$ such that for all $\delta_q\in \cH_q$
$$\Psi .\delta_q = \ttH'(q; \delta_q)^{\ft}p_h . \omega_p.$$
We write $\{\cH_q\ni X\mapsto \ttH'(q; X)^{\ft}p_h\}^{\ft}\omega_p$ for $\Psi$ to indicate $\delta_p\in\cH_q$, instead of $T_q\cQ$.
\begin{theorem}For $(\omega_q, \omega_p)\in\cE^2$, there exists a unique vector $(e_q, e_p)\in \cH_{q, p_h}\cH^*$ 
  \begin{gather} e_q := \ttH(q)\omega_p,\\
    e_p := -\ttH^{\ft}(q)\left(\omega_q + \{\cH_q\ni X\mapsto \ttH'(q; X)^{\ft}p_h\}^{\ft}\omega_p\right) + \ttH'(q, e_q)^{\ft}p_h\label{eq:horep}
    \end{gather}    
  such that for all $\xi_q, \xi_p\in \cH_{q, p_h}\cH^*$
\begin{equation}  
  \omega_q .\xi_q + \omega_p . \xi_p = \Omega((e_q, e_p), (\xi_q, \xi_p)) = e_q.\xi_p -e_q . \xi_p.\label{eq:Hpairing}\end{equation}
%Here, $\{X\mapsto \ttH'(q; X)p_h\}^{\ft}\omega_p\in \cE$ satisfies the below for $\delta_q\in \cH_q$ 
%\begin{equation} \delta_q. \{X\mapsto \ttH'(q; X)^{\ft}p_h\}^{\ft}\omega_p = \omega_p . \ttH'(q; \delta_q)^{\ft}p_h
%\end{equation}  
For a $\cH^*$-Hamiltonian $G$, there exists a unique $\cH^*$-horizontal vector field $\ttZ_{G}$
\begin{equation}\label{eq:ZG}
\ttZ_G(q, p_h) := \begin{bmatrix}\ttH(q)G_{p}\\
    \ttH'(q; \ttH(q)G_p)^{\ft}p_h - \ttH(q)^{\ft}\left(G_q +  \{\cH_q\ni X\mapsto \ttH'(q; X)^{\ft}p_h\}^{\ft}G_p \right)\end{bmatrix}
\end{equation}
such that $\rD_{\Phi}G = dG.\Phi = G_q.\Phi_q + G_{p}.\Phi_p = \Omega(\ttZ_G, \Phi)$ for all $\cH^*$-horizontal vector field $\Phi= (\Phi_q, \Phi_p)$, where $(G_q, G_{p})$ are the partial gradients of $G$ in $\cE^2$.
\end{theorem}
\begin{proof}Uniqueness follows from \cref{prop:Hnondegen}. Similar to \cref{thm:HVF},
      $$\begin{gathered}e_q. \xi_p - e_p . \xi_q = \ttH(q)\omega_p .\xi_p -
(\ttH'(q; e_q)^{\ft}p_h - \ttH(q)^{\ft}\{\omega_q +  \{X\mapsto \ttH'(q; X)^{\ft}p_h\}^{\ft}\omega_p \}).\xi_q\\
      = \omega_p . \ttH(q)^{\ft} \xi_p +\ttH(q)^{\ft} \omega_q . \xi_q + \{X\mapsto \ttH'(q; X)^{\ft}p_h\}^{\ft}\omega_p .\ttH(q)\xi_q\\
      = \omega_p . \ttH(q)^{\ft} \xi_p + \omega_q . \xi_q + \omega_p.\ttH'(q; \xi_q)^{\ft}p_h = \omega_p .\xi_p + \omega_q .\xi_q
    \end{gathered}$$
  where we use $\ttH(q)^{\ft}\ttH'(q; e_q)^{\ft}p_h =0$, the definition of $\{X\mapsto \ttH'(q; X)^{\ft}p\}^{\ft}h_p$ next, and $\xi_p = \ttH(q)^{\ft} \xi_p + \ttH'(q; \xi_q)^{\ft}p$ in the last, giving us  \cref{eq:Hpairing}.

It remains to apply the above to $(G_q, G_p)$ at any point $(q, p_h)\in \cH^*$.
\end{proof}
The $\cH^*$-Hamilton flow is a solution curve $\gamma(t) = (q(t), p_h(t))$ on $\cH^*$ of
\begin{equation}(\dot{q}, \dot{p_h}) = \ttZ_G(q(t), p_h(t)).\end{equation}
We have $\dot{\gamma}\in T\cH^*$. By construction, $\dot{q}= \ttH(q)G_p(q, p_h) \in \cH_q$, so $\dot{\gamma}\in \cH\cH^*$.

The decomposition $T\cQ=\cH\oplus\cV$ typically arises when we have a submersion $\qq:\cQ\mapsto\cB$ from $\cQ$ onto a manifold $\cB$. The inverse image $\qq^{-1}(b)\subset\cQ$, (called a fiber, as fibration is a special case) for $b\in\cB$ is a submanifold, and the integrable vertical bundle $\cV$ comes from tangent spaces of $\qq^{-1}(b)$. The splitting $T\cQ = \cH\oplus\cV$ usually follows from choosing a Riemannian or semi-Riemannian metric on $\cQ$ (that is nondegenerate on $\cV$), and define $\cH$ as the orthogonal complement.

\begin{proposition}\label{prop:hlift}Let $\qq:\cQ\to\cB$ be a differentiable submersion, thus, for $b\in \cB$, $\qq^{-1}(b)$ is a submanifold of dimension $\dim\cQ-\dim\cB$, and the tangent bundles $T\qq^{-1}(\qq(q))$ ($q\in\cQ$) form a subbundle $\cV\subset T\cQ$, the vertical bundle. Let $\Pi$ be the projection function onto $T\cQ\subset \cE^2$. Assume there is a subbundle $\cH$ such that $T\cQ = \cH\oplus\cV$ and $\Pi=\ttH\oplus\ttV$ is the corresponding decomposition of $\Pi$. Then $\qq'(q)$ maps $\cH_q$ bijectively onto $T_{\qq(q)}\cB$, and $\ttH(q)^{\ft}\qq'(q)^*$ is an invertible map from $T^*_{\qq(q)}\cB$ onto $\cH_q^*$. For $(q, p_h)\in\cH^*$ define
\begin{equation}  \qq_{\cH^*}(q, p_h) := (\qq(q), w)\in T^*_{\qq(q)}\cB\quad\text{ for }w\in T_{\qq(q)}^*\cB,\quad \ttH(q)^{\ft}\qq'(q)^*w = p_h.\label{eq:qqHs}
\end{equation}
Then $\qq_{\cH^*}$ is a differentiable submersion, a presymplectic map from $\cH^*$ onto $T^*\cB$, the kernel of $\qq_{\cH^*}'(q, p_h)$ is $\cV_{q, p_h}\cH^*$ and $\cH_{q, p_h}\cH^*$ maps bijectively onto $T_{\qq_{\cH^*}(q, p_h)}T^*\cB$.

If $F$ is a Hamiltonian on $T^*\cB$, then $G = F\circ \qq_{\cH^*}$ is a $\cH^*$-Hamiltonian. Via $\qq_{\cH^*}$, the $\cH^*$-Hamiltonian vector field $\ttZ_G$ maps to the $T^*\cB$ Hamilton vector field $\ttX_F$ and horizontal $\ttZ_G$ flows maps to $\ttX_F$ flows.
\end{proposition}
\begin{proof}At $q\in\cQ$, $\cV_q$ is the kernel of $\qq'(q)$, thus $\qq'(q)$ maps $\cH_q$ bijectively onto $T_{\qq(q)}\cB$, the decomposition of $T^*\cQ$ means $\ttH(q)^{\ft}\qq'(q)^*$ is bijective between $T^*_{\qq(q)}\cB$ and $\cH^*_q$. Thus, \cref{eq:qqHs} shows $\qq_{\cH^*}$ is well-defined and smooth on local charts, hence is smooth. The inverse image of a point $(b, w)$ is the submanifold of $\cH^*$, diffeomorphic to $\qq^{-1}(b)$
  $$\qq_{\cH^*}^{-1}(b, w) = \{(q, \ttH(q)^{\ft}\qq'(q)^*w)| q\in \qq^{-1}(b)\}.$$
We will work with a local chart of $\cB$, identified with an open subset of $\R^d$, $d=\dim\cB$ and use the coordinates in $\R^d$ to consider $\qq$ and $\qq'$ as $\R^d$-valued functions. Let $K(q)$ be the operator $K(q) = \qq'(q)\ttH(q)$ from $\cE$ to $\R^d$. Then $K(q)^{\ft}w = p_h$ and for $(\Delta_q, \delta_p)\in T_{(q, p_h)}\cH^*$,
  $$\qq_{\cH^*}'(q, p_h): (\Delta_q, \delta_p) \mapsto (\qq'(q)\Delta_q, \Delta_w)$$
with $\Delta_w\in\R^d$ is the unique element satisfying $K'(q, \Delta_q)^{\ft}w + K(q)^{\ft}\Delta_w = \delta_p$. Clearly, $\qq'_{\cH^*}(q, p_h)$ is onto as given $(\Delta_b, \Delta_w)\in T_{b, w}T^*\cB$, there is $\delta_q\in \cH_q^*$ such that $\qq'(q)\delta_q = \Delta_b$. Set $\delta_p :=K'(q, \delta_q)^{\ft}w + K(q)^{\ft}\Delta_w$ then tautologically $(\delta_q, \delta_p)$ maps to $(\Delta_b, \Delta_w)$. Thus, $\qq_{\cH^*}$ is a submersion.

To verify $\qq_{\cH^*}(q)$ is presymplectic, let $\Omega_{\cB}$ and $\Omega_{\cQ}$ be the symplectic forms on $T^*\cB$ and $T^*\cQ$, respectively, note $\qq'(q)=\qq'(q)\ttH(q)$ on $T_q\cQ$ as $\qq'(q)\ttV(q)=0$. If $(\Phi_q, \phi_p)$ with $K'(q, \Phi_q)^{\ft}w + K(q)^{\ft}\Phi_w = \phi_p$ is another tangent vector for $\Phi_w\in\R^d$
  $$\begin{gathered}
    \Omega_{\cB}(\qq_{\cH^*}'(q, p_h)(\Delta_q, \delta_p), \qq_{\cH^*}'(q, p_h)(\Phi_q, \phi_p))
    =\qq'(q)\Delta_q . \Phi_w - \qq'(q)\Phi_q . \Delta_w\\
    =\qq'(q)\ttH(q)\Delta_q . \Phi_w - \qq'(q)\ttH(q)\Phi_q . \Delta_w
    =\Delta_q . K(q)^{\ft}\Phi_w - \Phi_q . K(q)^{\ft}\Delta_w   \\
    = \Delta_q . (\phi_p - K'(q, \Phi_q)^{\ft}w) - \Phi_q . (
    \phi_p-K'(q, \Delta_q)^{\ft}w )  \\
    = \Omega_{\cQ}((\Delta_q, \delta_p), (\Phi_q, \phi_p)) - \{K'(q, \Phi_q)\Delta_q - K'(q, \Delta_q)\Phi_q\} . w.
  \end{gathered}$$
We want the last group to vanish. Let $\qq^{(2)}$ be the Hessian of $\qq$, expand
$$K'(q, \Phi_q)\Delta_q =  \qq^{(2)}(q; \Phi_q, \ttH(q)\Delta_q) + \qq'(q)\ttH'(q; \Phi_q)\Delta_q,$$
then since both $q^{(2)}(q; \Phi_q, \Delta_q)$ and $\Pi'$ are symmetric bilinear, we need
  $$\begin{gathered}
  q^{(2)}(q; \Phi_q, \ttH(q)\Delta_q) - q^{(2)}(q; \Delta_q, \ttH(q)\Phi_q)
  + \qq'(q)\{\ttH'(q; \Phi_q)\Delta_q - \ttH'(q; \Delta_q)\Phi_q \}= 0\\
\Leftrightarrow q^{(2)}(q; \Phi_q, \ttV(q)\Delta_q) - q^{(2)}(q; \Delta_q, \ttV(q)\Phi_q)
  + \qq'(q)\{\ttV'(q; \Phi_q)\Delta_q - \ttV'(q; \Delta_q)\Phi_q \}= 0.
\end{gathered}$$
The last line follows from $\rD_{\Delta_q}(\qq'(q)\ttV(q)\Phi_q) = 0$ and $\rD_{\Phi_q}(\qq'(q)\ttV(q)\Delta_q) = 0$.

The description of the kernel of $\qq'_{\cH^*}(q, p_h)$ follows from symplecticity, if $\qq'_{\cH^*}(q, p_h)(\kappa_q, \kappa_p) = 0$, for $(\kappa_q, \kappa_p)\in T_{q, p_h}\cH^*$ then for $(\Delta_q, \delta_p)\in T_{q, p_h}\cH^*$, the symplectic pairing of the images in $\cB$ is zero, hence $\Omega_{\cQ}((\kappa_q, \kappa_p), (\Delta_q, \delta_p))=0$, or $(\kappa_q, \kappa_p)$ is vertical. Conversely, if $(\kappa_q, \kappa_p)$ is vertical, then the symplectic pairing on $\cH^*$ with any $(\Delta_q, \delta_p)$ is zero, so the pairing of the images is zero. Since $\Omega_{\cB}$ is nondegenerate and $\qq'_{\cH^*}$
is onto, it maps $(\kappa_q, \kappa_p)$ to zero. From here, the chain rule shows the directional derivative of $G = F\circ\qq_{\cH^*}$ in a $\cH^*$-vertical direction is zero, for $G$ is a $\cH^*$-Hamiltonian. The correspondence between $\ttZ_G$ and $\ttX_h$ follow from symplecticity and the chain rule.
\end{proof}
Assume the conditions of \cref{prop:hlift}, then a tangent vector to $\cB$ at $b=\qq(q)$ is the image of the unique horizontal vector at $q$, the {\it horizontal lift}. We can lift vector fields to horizontal vector fields and Hamilton flows to $\cH^*$-Hamilton flows.

An application is the lift of metric-potential systems, including the usual results on lifts of geodesics. If $\sfg_{\cB}$ and $\sfg_{\cQ}$ are metrics on $\cB$ and $\cQ$ such that $\qq$ is a {\it Riemannian submersion} \cite{ONeill1966}, that is, $\cH$ is orthogonal to $\cV$ via the metric $\sfg_{\cQ}$ and $\qq'(q)$ is an isometry between $\cH_q$ to $T_{\qq(q)}\cB$, set $F(b, p_b) = \frac{1}{2}p_b.\sfg_{\cB}(b)^{-1} p_b$ for $(b, p_b)\in T^*_b\cB$, we get a lifted version of \cref{prop:metricPotential}, with $\ttH$ in place of $\Pi$.

We next compute the lift of curvature. For $\omega\in\cE$, let $\tth_{\omega}$ be the horizontal vector field $\tth_{\omega}:y \mapsto \ttH(y)\omega$ for $y\in\cQ$. Let $\sfg$ be the metric operator for $\cQ$ on $\cE$ and let $\langle, \rangle_{\sfg}$ denote the $\sfg$-pairing, so for $\omega_1,\omega_2\in \cE$, $\langle\omega_1, \omega_2\rangle_{\sfg, q} = \omega_1. \sfg(q)\omega_2 $.
\begin{proposition}\label{prop:curvlift}Let $\qq:\cQ\mapsto \cB$ be a Riemannian submersion with $\cQ\subset\cE$. For  $q\in\cQ, \xi \in \cH_q, \omega\in\cE$, define $\sA_{\xi}(q), \sAd_{\xi}(q)$ by the first equality below, where $\Gamma$ is the Christoffel function of the Levi-Civita connection on $\cQ\subset \cE$.
\begin{gather}
  \sA_{\xi}(q)\omega := \ttV(q)(\nabla_{\tth_{\xi}}\tth_{\omega})_{y=q} = -\ttV'(q, \xi)\ttH(q)\omega + \ttV(q)\Gamma(q; \xi, \ttH(q)\omega),\\
  \sAd_{\xi}(q)\omega :=-\ttH(q)(\nabla_{\xi}\ttV(y)\omega)_{y=q} = \ttH'(q, \xi)\ttV(q)\omega - \ttH(q)\Gamma(q; \xi, \ttV(q)\omega).\label{eq:sAdomega}
\end{gather}
Then $\sA_{\xi}$ and $\sA_{\xi}^{\dagger}$ are adjoints under the $\sfg$-pairing. For $\omega\in\cH_q$, $\sA_{\xi}$ evaluates to the O'Neill's tensor. If $\xi,\eta\in \cH_q, \epsilon\in\cV_q$ %and $\dagger$ denotes the adjoint under $\sfg$ then
   \begin{gather}\label{eq:oneilLie}
     \sA_{\xi}(q)\eta = \frac{1}{2}\ttV(q)[\tth_{\xi}, \tth_{\eta}]_{y=q} = \frac{1}{2}\{\ttH'(q; \xi)\eta - \ttH'(q, \eta)\xi\}.
%     \sAd_{\xi}(q)\epsilon = \frac{1}{2}\ttH(q)\{\ttH'(q; \xi)\epsilon - \{X\mapsto\ttH'(q, X)\xi\}^{\dagger}\epsilon\}.\label{eq:sAdVert}
   \end{gather}
If $\nabla$ is the Levi-Civita connection on $T\cQ$, then $\nabla^{\cH}_{\ttX}\ttY := \ttH\nabla_{\ttX}\ttY$ for a vector field $\ttX$ and horizontal vector field $\ttY$ on $\cQ$ is a connection on the horizontal bundle $\cH$. Let $\GammaH$ be its Christoffel function, extended bilinearly to $\cE$. At $q\in\cQ, \xi, \eta, \phi\in \cH_q$, let $\bRcB_{\xi\eta}\phi$ be the horizontal lift of the Riemannian curvature $\RcB_{\qq(\xi), \qq(\eta)}\qq(\phi)$ on $\cB$ then  
\begin{gather}      
%     \GammaH(q; \Delta, \eta) = -\ttH'(q; \Delta)\eta + \ttH(q)\mrGamma(q; \Delta, \eta),\label{eq:GammaH}\\
  \bRcB_{\xi\eta}\phi = 
        \rD_{\xi}\GammaH(\eta, \phi) -\rD_{\eta}\GammaH(\xi, \phi)
   +\GammaH(\xi, \GammaH(\eta, \phi)) - \GammaH(\eta, \GammaH(\xi, \phi)) - 2\sAd_{\phi}\sA_{\xi}\eta,\label{eq:curGammaH}\\
   \bRcB_{\xi\eta}\phi = \ttH(q) \RcQ_{\xi\eta}\phi - 2 \sAd_{\phi} \sA_{\xi}\eta + \sAd_{\xi} \sA_{\eta}\phi +  \sAd_{\eta} \sA_{\phi}\xi.\label{eq:ONeil13}
      \end{gather}
\end{proposition}
\begin{proof}By the metric invariance of $\nabla$ and since $\ttV$ and $\ttH$ are self-adjoint under $\sfg$
  $$\begin{gathered}0 = \rD_{\xi}\langle \ttH(y)\omega_1, \ttV(y)\omega_2 \rangle_{\sfg} =
  \langle \nabla_{\tth_{\xi}}\ttH(y)\omega_1, \ttV(y)\omega_2 \rangle_{\sfg} +
  \langle \ttH(y)\omega_1, \nabla_{\tth_{\xi}}\ttV(y)\omega_2 \rangle_{\sfg}\\
  \Rightarrow  \langle \nabla_{\tth_{\xi}}\ttH(y)\omega_1, \ttV(y)\omega_2 \rangle_{\sfg} =-  \langle \ttH(y)\omega_1, \nabla_{\tth_{\xi}}\ttV(y)\omega_2 \rangle_{\sfg}\\
\Rightarrow \langle \ttV(q)(\nabla_{\tth_{\xi}}\tth_{\omega_1})_{y=q}, \omega_2 \rangle_{\sfg} = - \langle \omega_1, \ttH(q)(\nabla_{\tth_{\xi}}\ttV(y)\omega_2)_{y=q} \rangle_{\sfg}
  \end{gathered}$$
Thus, $\sAd_{\xi}$ and $\sA_{\xi}$ are adjoint in the $\sfg$-pairing. Next,
  $$\ttV(q)(\nabla_{\tth_{\xi}}\tth_{\omega})_{y=q} = \ttV(q)(\ttH'(q, \xi)\eta + \Gamma(\xi, \eta)) = - \ttV'(q; \xi)\eta +  \ttV(q)\Gamma(\xi, \eta)$$
  since $0= \ttV'(q; \xi)\ttH(q)\eta + \ttV(q)\ttH'(q, \xi)\eta$ by product rule. Equation (\ref{eq:oneilLie}) follows from \cite[lemma 2]{ONeill1966}, expanding the Lie bracket,  then use $ \ttV(q)\ttH'(q, \xi)\eta = - \ttV'(q; \xi)\eta$
$$\sA_{\xi}\eta = \frac{1}{2}\ttV(q)[\tth_{\xi}, \tth_{\eta}]_{y=q}=\frac{1}{2}\ttV(q)(\ttH'(q; \xi)\eta - \ttH'(q, \eta)\xi)=-\frac{1}{2}(\ttV'(q; \xi)\eta - \ttV'(q, \eta)\xi)
$$
  and expand $\ttV' = \Pi'- \ttH'$, simplifying $\Pi'(q; \xi)\eta = \Pi'(q; \eta)\xi$.% From \cref{eq:oneilLie}, since $\ttH$ is self-adjoint under $\sfg$, we get \cref{eq:sAdVert} using the definition of $\dagger$.

Note $\nabla^{\cH}$ is of the form in \cref{eq:defineNablaTW}. From \cite{ONeill1966}, the horizontal lift of a vector field $\ttX_{\cB}$ on $T\cB$ is called a {\it basic} vector field on $\cQ$, which we denote by $\tdX$. For a Riemannian submersion, $(\nabla_{\ttX_{\cB}}\ttY_{\cB})^{\widetilde{}} = \ttH\nabla_{\tdX}\tdY = \nabla^{\cH}_{\tdX}\tdY, [\ttX_{\cB}, \ttY_{\cB}]^{\widetilde{}} = \ttH[\tdX, \tdY]$. Thus,
\begin{equation}\bRcB_{\ttX_{\cB},\ttY_{\cB}}\ttZ_{\cB} = \nabla^{\cH}_{\tdX}\nabla^{\cH}_{\tdY}\tdZ-\nabla^{\cH}_{\tdY}\nabla^{\cH}_{\tdX}\tdZ - \nabla^{\cH}_{\ttH[\tdX, \tdY]}\tdZ=\Rc^{\nabla^{\cH}}_{\tdX\tdY}\tdZ + \nabla^{\cH}_{\ttV[\tdX, \tdY]}\tdZ\label{eq:RcBexpand}
\end{equation}
where $\Rc^{\nabla^{\cH}}$ is the curvature of $\nabla^{\cH}$ on the bundle $\cH$, given by \cref{eq:rc1} using $\GammaH$ , or by \cref{eq:affineGauss} where $\sA$ is the second fundamental form of $\ttH$ operating on $T\cQ$. From \cite[lemma 3]{ONeill1966}, for three basic vector fields $\tdX, \tdY, \tdZ$ evaluated to $\xi, \eta, \phi$ at $q$
$$\nabla^{\cH}_{\ttV[\tdX, \tdY]}\tdZ = \ttH\nabla_{\tdZ}\ttV[\tdX, \tdY]=-2\sAd_{\tdZ}\sA_{\tdX}\tdY$$
from \cref{eq:sAdomega,eq:oneilLie}. Equations (\ref{eq:curGammaH}) and (\ref{eq:ONeil13}) follow from the two ways (\cref{eq:rc1,eq:affineGauss}) to compute $\Rc^{\nabla^{\cH}}$. Note \cref{eq:ONeil13} is \cite[theorem 2]{ONeill1966}.
\end{proof}
We note to use \cref{eq:curGammaH}, $\GammaH(\Delta, \phi)$ must be valid for {\it tangent} $\Delta$ while $\phi$ is horizontal.
\section{Examples}\label{sec:examples}
\subsection{Rigid body dynamics}
This is well-studied in \cite{LeeLeokMBook,Arnold,LeeLeokM}, but we show our global formulas make the geometric calculation easy. It is well-known that the mechanics is closely related to left-invariant metrics on the special Euclidean group $\SE(n)$. Let $\cQ = \SE(n)$ be identified with the submanifold $\SOO(n)\times \R^n\subset\cE = \R^{n\times n}\times \R^n$, where the special orthogonal group $\SOO(n)$ consists of orthogonal matrices of determinant $1$. The group structure on $\SE(n)$ is defined by
\begin{equation}(U_1, z_1).(U_2, z_2) = (U_1U_2, U_1z_2 + z_1)\text{ for }(U_1, z_1), (U_2, z_2) \in \cQ.
\end{equation}
The $z$ component represents an anchor point of the rigid body, the $U$-component represents the rotation of the body around the anchor point. We assume the anchor point is the center of mass. The identity of this group is $\dI_{\SE(n)} = (\dI_n, 0)$ and $(U, z)^{-1} = (U^{\ft}, -U^{\ft}v)$. A tangent vector at $q=(U, z)\in\cQ$ is a pair $(\Delta_U, \Delta_z)\in \cE=\R^{n\times n}\times \R^n$ satisfying $(U^{\ft}\Delta_U)_{\sym} = 0$ with $A_{\sym} = \frac{1}{2}(A+A^{\ft})$ for $A\in \R^{n\times n}$.

The constraint is given by $\rC(U, z) = U^{\ft}U - \dI_n\in \cE_L = \Herm{n}$, $\Herm{n}$ is the space of symmetric matrices. Let $\oo(n)$ be the Lie algebra of antisymmetric matrices. If $v = (\Delta_U, \Delta_z)$ is a velocity (tangent) vector, then $\rC(U, z)'(\Delta_U, \Delta_z) = (U^{\ft}\Delta_U)_{\sym}$, so a tangent vector satisfies $U^{\ft}\Delta_U\in \oo(n)$.

The (mass) metric has two components, the total mass $m>0$ and the {\it inertia} $\cI$, an invertible operator on $\R^{n\times n}$ satisfying $\cI^{\ft} = \cI$ in the Frobenius inner product, and $\cI(\oo(n))=\oo(n)$. We do not make further assumptions, but $\cI$ is often given by a symmetric matrix in $\R^{n\times n}$, also denoted $\cI$, with positive entries, where the operator $\cI$ acts by Hadamard product,  $\cI(e) =\cI \odot e$ for $e\in\oo(n)$, or $\cI(e)_{ij} = \cI_{ij}e_{ij}$ for $i,j=1\cdots k$ (the diagonal entries have no effect, set to $1$). For $n=3$, the coefficients $\cI_{23}, \cI_{31}, \cI_{12}$ are often denoted $\cI_x, \cI_y, \cI_z$. For the mass metric operator, define $\sfg(\dI, 0)(e, \Delta_z) = (\cI(e), m\Delta_z)$ for $(e, \Delta_z)\in\R^{n\times n}\times \R^n$ and extend it {\it left-invariantly} to $(U, z)\in \SE(n)$ 
\begin{gather}\sfg(U, z)( e, \Delta_z) := (U, z)\sfg(\dI_n, 0)((U, z)^{-1}( e, \Delta_z)) = (U\cI(U^{\ft} e), m \Delta_z),\\
\sfg(U, z)^{-1}( e, \Delta_z) = (U\cI^{-1}(U^{\ft} e), \frac{1}{m}\Delta_z).
\end{gather}  
 It is clear $\sfg$ is a positive definite operator on $\cE$. A cotangent vector $p = (p_U, p_z)$ satisfies $(\cI^{-1}(U^{\ft}p_U))_{\sym} = 0$, thus $T_q^*\cQ$ is identified with $T_q\cQ$.

 In \cref{eq:ProjG}, at $(U, z)\in\SE(n)$, $(\rC')^{\ft}a  = (Ua, 0)\in\cE$ for $a\in\Herm{n}$, $\rC'\sfg^{-1}(\rC')^{\ft}a = \cI^{-1}(a)$. The $\SOO(n)$ component of the projection is
$$e\mapsto e - U\cI^{-1}(U^{\ft}(U\cI(U^{\ft}e)_{\sym})) = e - U(U^{\ft}e)_{\sym} = 
U(U^{\ft}e)_{\asym}
$$
where $A_{\asym}= \frac{1}{2}(A-A^{\ft})$ for $A\in \R^{n\times n}$. The projection $\Pi=\Pi_{\sfg}$ is
\begin{equation}
\Pi(U, z)(e, \Delta_z) = (e - U(U^{\ft}e)_{\sym} , \Delta_z)= (U(U^{\ft}e)_{\asym} , \Delta_z)\label{eq:SEnproj}.
\end{equation}
We are interested in the Hamiltonian $\frac{1}{2}p. \sfg(U, z)^{-1}p + f(q)$ for $q = (U, z)\in \SE(n)$. For tangent vectors $\xi = (\xi_U, \xi_z), \eta=(\eta_U, \eta_z)$, by direct calculations
\begin{gather}(\rD_{\xi_U, \xi_z}\sfg(q)(e, \Delta_z))_U = \xi_U\cI(U^{\ft}e) + U\cI(\xi_U^{\ft}e)\\
  (\chi_{\sfg}(q; \xi, \eta))_U =  \xi_U\cI(\eta_U^{\ft}U) + \eta_U\cI(\xi_U^{\ft}U) \label{eq:chiSE} \\
  (\rD_{\xi}\Pi(q)(e, z))_U = -\xi_U\sym(U^{\ft}e) - U\sym(\xi_U^{\ft}e)\label{eq:DPiRigid}\\
  \Pi(q)\sfg(q)^{-1}(e, \Delta_z) = (U\cI^{-1}(U^{\ft}e)_{\asym}, \frac{1}{m}\Delta_z)\\
  \mrGamma(q; \xi, \eta)_U = \frac{1}{2}U\cI^{-1}\{[U^{\ft}\xi_U,\cI(U^{\ft}\eta_U)] +
          [U^{\ft}\eta_U,\cI(U^{\ft}\xi_U)]\}\label{eq:mrGammaRigid}\\
          \Gamma(q; \xi, \eta)_U =  U\sym(\xi_U^{\ft}\eta_U) + \xi_U\sym(U^{\ft}\eta_U) +
          \mrGamma(q; \xi, \eta) \label{eq:GammaRigid}
\end{gather}
with $[A,B]= AB-BA$ denoting the Lie bracket. The $z$-components are zero otherwise. The derivatives are routine. To verify \cref{eq:chiSE}, $\cI(\oo(n))=\oo(n)$ implies $\cI(U^{\ft}\eta_U)^{\ft} =-\cI(U^{\ft}\eta_U)=\cI(\eta_U^{\ft} U)$; since $\cI$ is self-adjoint, $\Tr\xi_U^{\ft}U\cI(\phi_U^{\ft}\eta_U) =\Tr\cI(\xi_U^{\ft}U)\phi_U^{\ft}\eta_U$ for $\phi = (\phi_U, \phi_z)\in T_{(U, z)}\SE(n)$, thus
$$\begin{gathered}\xi_U . (\rD_{\phi_U, \phi_z}\sfg(q)(\eta_U, \eta_z))_U = \Tr\xi_U^{\ft}
\phi_U\cI(U^{\ft}\eta_U) + \Tr\xi_U^{\ft}U\cI(\phi_U^{\ft}\eta_U)\\
=\Tr\phi_U^{\ft}\xi_U\cI(\eta_U^{\ft}U) + \Tr\phi_U^{\ft}\eta_U\cI(\xi_U^{\ft}U)\}
= (\chi_{\sfg}(q; \xi, \eta))_U . \phi_U.
\end{gathered}
$$
Expansions give us the rest. The Euler-Lagrange equation (\ref{eq:geodesic}) is
\begin{gather}\ddot{U} + U\dot{U}^{\ft}\dot{U} +
U\cI^{-1}[U^{\ft}\dot{U},\cI(U^{\ft}\dot{U})] 
+  U\cI^{-1}(U^{\ft}f_U(q))_{\asym}=0,\\
  m\ddot{z} + f_z(q) =0.
\end{gather}
The angular velocity is $\omega = U^{\ft}\dot{U}$, the first equation reduces to \cite[Equation 7.21]{LeeLeokMBook}
\begin{equation}\cI\dot{\omega} + [\omega, \cI\omega] + (U^{\ft}f_U(q))_{\asym}=0,\end{equation}
For the curvature, we showed in \cite{NguyenJLT} with $A=U^{\ft}\xi_U, B=U^{\ft}\eta_U, C=U^{\ft}\phi_U$
\begin{equation}\begin{gathered}(\Rc_{\xi, \eta}\phi)_U =U\{ -\frac{1}{2}[[A, B], C]_{\cI} +\frac{1}{4}[A, [B, C]_{\cI}]_{\cI} - \frac{1}{4}[B[A, C]_{\cI}]_{\cI}\}\\
\text{ with }  [A, B]_{\cI} := [A, B] + \cI^{-1}[A, \cI(B)] + \cI^{-1}[B, \cI(A)]
\end{gathered}\label{eq:rigidCurv} \end{equation}
using Arnold's method \cite{Arnold}.  We need the below and a permutation of $\eta$ and $\xi$ for \cref{eq:rc2}
$$\begin{gathered}
  (\rD_{\xi}\mrGamma(\eta, \phi))_U = \frac{1}{2}\xi_U\cI^{-1}\{[U^{\ft}\eta_U,\cI(U^{\ft}\phi_U)] +
     [U^{\ft}\phi_U,\cI(U^{\ft}\eta_U)]\}+\\
     \frac{1}{2}U\cI^{-1}\{[\xi_U^{\ft}\eta_U,\cI(U^{\ft}\phi_U)]
     + [U^{\ft}\eta_U,\cI(\xi_U^{\ft}\phi_U)]
       + [\xi_U^{\ft}\phi_U,\cI(U^{\ft}\eta_U)]
       + [U^{\ft}\phi_U,\cI(\xi_U^{\ft}\eta_U)]\}.
\end{gathered}     $$
There is no $z$-component. Instead of an algebraic comparison of \cref{eq:rc2} with \cref{eq:rigidCurv}, we compare them numerically in \cite{NguyenVerify}. We compare them algebraically for $\cI = Id$ below.

If $\cI = Id$, $\Gamma$ simplifies to $\Gamma^a(\xi, \eta)_U = U(\xi_U^{\ft}\eta_U)_{\sym}$. From \cref{eq:rc1}
$$\begin{gathered}(\Rc_{\xi,\eta}\phi)_U= (\rD_{\xi}\Gamma^a(\eta, \phi) - \rD_{\eta}\Gamma^a(\xi, \phi)+ \Gamma^a(\xi, \Gamma^a(\eta, \phi)) - \Gamma^a(\eta, \Gamma^a(\xi, \phi)))_U\\
  = UA(-BC)_{\sym} - UB(-AC)_{\sym} + U((A(BC)_{\sym})_\sym - (B(AC)_{\sym})_{\sym}\\
  =  \frac{1}{4}U(
  -2ABC  - 2ACB + 2BAC +2BCA \\
  + ABC + ACB - BCA - CBA
- BAC - BCA + ACB + CAB)\\
=\frac{1}{4}U(-ABC +BAC -CBA + CAB )
=\frac{1}{4}U( C[A, B] - [A, B]C ) = \frac{1}{4}U[C, [A, B]]
\end{gathered}$$
which is well-known, agreeing with \cref{eq:rigidCurv} as in this case $[\;]_{\cI}=[\;]$ and the Jacobi's identity holds. Alternatively, to use \cref{eq:rc2}, in \cref{eq:GammaRigid} for $\Gamma$, we need to keep the term $\xi_U\sym(U^{\ft}\eta_U)$, since we need a formula for $\Pi'$ valid on $\cE$ as mentioned.

\subsection{Lifting Hamilton vector fields and curvature for a Grassmann manifold}
The formulas for Grassmann manifolds are similar to that of the sphere. The Grassmann manifold $\Gr{n}{k}$ could be considered as a quotient of $\SOO(n)$ by the diagonal block matrix subgroup $\SOO(k)\times \SOO(n-k)$. However, it is more convenient to consider it as a quotient of the Stiefel manifold $\St{n}{k}$ of $n\times k$ orthogonal matrices $Y$ with $Y^{\ft}Y = \dI_k$ by right multiplication by $\SOO(k)$. Using $\lb\; \rb$ to denote equivalent classes
$$\Gr{n}{k} =\{ \lb Y\rb\;|\quad Y^{\ft}Y = \dI_k, Y\sim YU \text{ for } U\in \SOO(k)\}.$$
Here, $\cE = \R^{n\times k}, \cQ = \St{n}{k}, \cB = \Gr{n}{k}$. With the metric $\sfg(q)=\dI_{\cE}$, the tangent and horizontal spaces and projections are  \cite{Edelman_1999}
\begin{gather}T_Y\cQ = T^*_Y\cQ = \{\eta\in \R^{n\times k}| (Y^{\ft}\eta)_{\sym} = 0\},\\
  \cH_Y = \cH^*_Y = \{\eta \in\R^{n\times k} | Y^{\ft}\eta = 0\},\\
  \Pi(Y)\omega = \omega - Y(Y^{\ft}\omega)_{\sym}, \omega\in \R^{n-k},\\
  \ttH(Y)\omega = \omega - YY^{\ft}\omega, \omega\in \R^{n-k}.
\end{gather}
So $\ttH'(Y, \Delta)p = -Y\Delta^{\ft}p$ for $\Delta\in T_Y\cQ, p\in \cH_Y^*$ and for $\omega\in \cE$, since $\Tr Y\Delta^{\ft}p\omega^{\ft} = \Tr\Delta^{\ft}p\omega^{\ft}Y$, we can take $\{X\mapsto \ttH'(q,X)^{\ft}p\}^{\ft} \omega = - p\omega^{\ft}Y.$
Using $Y^{\ft}p = p^{\ft}Y= 0$ to simplify
$$\ttH(q)^{\ft}\{X\mapsto \ttH'(q,X)^{\ft}p\}^{\ft} G_p = -pG_p^{\ft}Y$$
and  $\ttH'(q, \ttH(q)G_p)p = - YG_p^{\ft}p$, the lifted Hamilton equations are similar to \cref{eq:Hamiltonsphere}
\begin{equation}\begin{gathered}
    \dot{Y} = (\dI_{\cE}- YY^{\ft}) G_p,\\
    \dot{p} = pG_p^{\ft}Y -YG_p^{\ft}p  - (\dI_{\cE}-YY^{\ft})G_q.
  \end{gathered}
\end{equation}
For the curvature, with horizontal vectors $\xi, \eta,\phi$, tangent vector $\Delta$ and vertical vector $\epsilon$, $\GammaH(Y; \Delta, v) = Y\Delta^{\ft}v$, $\sA_{\xi}\eta = -\frac{1}{2}Y(\xi^{\ft}\eta - \eta^{\ft}\xi)$, $\sAd_{\phi}\epsilon = -\phi Y^{\ft}\epsilon$
$$\begin{gathered}
  \bRcB_{\xi, \eta}\phi = \xi\eta^{\ft}\phi - \eta\xi^{\ft}\phi + Y\xi^{\ft}Y\eta^{\ft}\phi - Y\eta^{\ft} Y\xi^{\ft}\phi -\phi Y^{\ft}Y(\xi^{\ft}\eta - \eta^{\ft}\xi),\\
  \bRcB_{\xi, \eta}\phi =\xi\eta^{\ft}\phi - \eta\xi^{\ft}\phi  -\phi \xi^{\ft}\eta + \phi\eta^{\ft}\xi.
\end{gathered}  $$
Accounting for a scaling factor of $\frac{1}{2}$, this is the same as \cite[equation 4]{Wong}.

\section{The Kim-McCann metric and a reflector antenna-type cost function}\label{sec:KimMcCann}
In \cite{KimMcCann}, Kim and McCann define a semi-Riemannian metric on a submanifold $\cN$
of the product $\cM\times\bcM$ two manifolds of the same dimension $\cM$ and $\bcM$. On $\cN$, we assume there is a sufficiently smooth function $c=c(\ttq) = c(x, y)$ such that at each $\ttq=(x, y)\in\cN\subset\cM\times\bcM$, the bilinear form $(\xi,\bxi)\mapsto \rD_{\xi}\brD_{\bxi}c(\ttq)$ is nondegenerate for $\xi\in T_x\cM, \bxi\in T_{y}\cM$. Here, $\rD$ denotes the directional derivative in $x$ and $\brD$ denotes the directional derivative in $y$. The pairing of two tangent vectors $\bdxi=(\xi, \bxi), \bdeta=(\eta, \bareta)$ at $\ttq\in\cN\subset \cM\times \bcM$ is defined as
\begin{equation}\langle(\bdxi, \bdeta \rangle_{KM} = -\frac{1}{2}\{(\rD_{\xi}\brD_{\bareta}c)(\ttq) + (\brD_{\bxi}\rD_{\eta}c)(\ttq)\}.\label{eq:KimMcCann}
\end{equation}  
Assuming the nondegeneracy condition, for a tangent vector $\bdxi=(\xi, \bxi)$, define $\bdxi_{x0} = (\xi, 0), \bdxi_{0y} = (0, \bxi)$, the MTW-tensor (also called cost-sectional curvature or cross-curvature) of a vector $\bdxi$ is
\begin{equation}\crosscv(\bdxi) = \langle \Rc_{\bdxi_{x0}\bdxi_{0y}}\bdxi_{0y},\bdxi_{x0}\rangle_{KM},\end{equation}
where $\Rc$ is the curvature of the metric. A condition for regularity of the optimal transport problem associated with the cost $c$ is the cross-curvature to be non-negative on null vectors $\bdxi$, $\|\bdxi\|_{KM} = 0$. Beside examples arising from cost functions related to the Riemannian metric, the reflector antenna cost function, which inspires our construction here, also satisfies this condition.

For two positive integers $k \leq n$, let $\R_*^{n\times k}\subset \R^{n\times k}$ be the manifold of matrices of full rank $k$. Recall $\OO(k)$ is the group of orthogonal matrices in $\R^{k\times k}$ and $\oo(k)$ is its Lie algebra of skew-symmetric matrices. The manifold $\Sd{n}{k}$ of $n\times n$ positive-semidefinite matrices symmetric of fixed rank $k$ could be identified with a quotient of $\R^{n\times k}_*$ by $\OO(k)$ via the map $x\mapsto xx^{\ft}$ for $x\in\R^{n\times k}_*$. It is clear $xx^{\ft}$ is in $\Sd{n}{k}$. On the other hand, if $A\in \Sd{n}{k}$, from the eigenvalue decomposition, we can write $A = xx^{\ft}$ for $x\in \R_*^{n\times k}$, and $xx^{\ft} = zz^{\ft}$ implies $z = xx^{\ft}z(z^{\ft}z)^{-1}=xU$, with $U = x^{\ft}z(z^{\ft}z)^{-1}$ is orthogonal as algebraically $U^{\ft}U = (z^{\ft}z)^{-1}z^{\ft}z= \dI_k$.

Recall the polar decomposition of an invertible matrix $M$ is the decomposition $M=\Sigma U$, where $\Sigma$ is the unique symmetric positive-definite square root of $MM^{\ft}$, $\Sigma^2= MM^{\ft}$. In that case, $U=\Sigma^{-1}M$ is orthogonal as $UU^{\ft} = \Sigma^{-1}\Sigma^2\Sigma^{-1} = I$. Note that the condition $M$ is invertible implies $\Sigma$ and $U$ are uniquely defined and are smooth functions of $M$. The following is immediate from the definition of $\Sigma$.
\begin{lemma} Fix $\alpha > 0$, for $x, y\in \R_*^{n\times k}$ such that $x^{\ft}y$ is invertible, let $x^{\ft}y = \Sigma U$ be the polar decomposition of $x^{\ft}y$. Set $K=\alpha + \Tr\Sigma$ and define
  \begin{equation}c(x, y) := - \log K =  -\log(\alpha+\Tr\Sigma)\label{eq:cAntenna}
  \end{equation}
  then $c(x U_1, y U_2^{\ft}) = c(x, y)$ for $U_1, U_2\in \OO(k)$, thus $c$ is a well-defined function on $\lb x\rb, \lb y\rb\in \R_*^{n\times k}/\OO(k)\equiv \Sd{n}{k}$, where $\lb \rb$ to denote the equivalent class. Define
  \begin{equation} \cB = \cB_{\Sd{n}{k}} := \{(\lb x\rb, \lb y\rb)\in (\Sd{n}{k})^2| x^{\ft}y \text{ is invertible} \}.\label{eq:cBFR}
  \end{equation}
Then $c$ induces a well-defined function on $\cB$ and could be used as a cost function.
\end{lemma}
Note that on $\R_*^{n\times k}$, the pairing in \cref{eq:KimMcCann} with this function $c$ is degenerated, as any vertical vector $\mathring{\psi} = (xa, yb)$ for $a, b\in \oo(k)$ satisfies $\rD_{xa}c = 0, \brD_{yb}c=0$ as $c$ is equivariance under $\OO(k)$. To facilitate computation, we introduce a semi-Riemannian pairing on the open submanifold
\begin{equation} \cQ = \cQ_{\R_*^{n\times k}} := \{ (x, y)\in (\R_*^{n\times k})^2 |x^{\ft}y \text{ is invertible}\}\subset \cE := (\R^{n\times k})^2.
\label{eq:cQFR}
\end{equation}
It induces the Kim-McCann pairing on $\cB$ corresponding to $c$, and we show the later pairing is in fact nondegenerate. In the following, we use $\mathring{}$ to signify a pair of vectors, while the second component is denoted with $\bar{}$, as in $\bdxi = (\xi, \bxi)$. We will denote the map $(x, y)\mapsto (\lb x\rb, \lb y\rb)$ by $\qq$.

We will focus on the cost in \cref{eq:cAntenna}, but generally the lift could be defined for cost functions of the form $\bbf(\Sigma)$ where $x^{\ft}y = \Sigma U$ is the polar decomposition, for a smooth, scalar function $\bbf$ invariant under the adjoint action of $\OO(k)$ on the submanifold of positive definite matrices $\Sp{k}$ of $\R^{k\times k}$. We only provide the proof for $\bbf(\Sigma) =-\log(\alpha + \Tr\Sigma)$ in proposition \ref{prop:antennaLift} below. The sphere and the Grassmann square-Riemannian distance cost \cite{KimMcCann,Loeper} correspond to $\bbf(\Sigma)=\Tr\arccos^2\Sigma$. Let $\Hess_{\bbf}$ and $\egrad_{\bbf}$ denote the Hessian and gradient of $\bbf$. The general form of the pairing is
  \begin{equation}\begin{gathered}
\langle\bdxi, \bdxi \rangle_{\bbf} = - \Hess_{\bbf}(\Sigma; (U\bxi^{\ft}x)_{\sym}, (\xi^{\ft}yU^{\ft})_{\sym})
- \egrad_{\bbf}(\Sigma).(U\bxi^{\ft}\xi)
\end{gathered}
  \end{equation}
for $\bdxi=(\xi, \bxi)\in \cE := (\R^{n\times k})^2$. If some technical conditions on $\bbf$ are satisfied, then the horizontal condition is $(U\bxi^{\ft}x)_{\asym} = (\xi^{\ft}yU^{\ft})_{\asym} = 0$, hence 
$\rD_{\xi} U = \brD_{\bxi} U=0$ in \cref{eq:DU} below.

Recall a symmetric Lyapunov equation is an equation of the form $\Sigma X + X\Sigma = B$ for given $\Sigma$ and $B$, with $\Sigma$ is a symmetric matrix and $X$ and $B$ are square matrices. The equation has a unique solution if $\Sigma$ is positive definite, which is the case considered here. We have $X = L^{-1}_{\Sigma} B$ if we define the Lyapunov operator as $L_{\Sigma}X := \Sigma X + X\Sigma$. From $X + \Sigma^{-1} X \Sigma = \Sigma^{-1} B$
\begin{equation}\Tr X = \Tr L^{-1}_{\Sigma}B = \frac{1}{2}\Tr(\Sigma^{-1}B).
\end{equation}
Note $L_{\Sigma}$ is a symmetric operator on $\oo(k)$, therefore
\begin{equation}\Tr A L_{\Sigma}^{-1}B_{\asym} = \Tr A_{\asym} L_{\Sigma}^{-1}B_{\asym}
=  \Tr (L_{\Sigma}^{-1}A_{\asym}) B_{skew} = \Tr (L_{\Sigma}^{-1}A_{\asym}) B.\label{eq:LSigSym}\end{equation}
If $Y = L^{-1}_{\Sigma} B_{\asym}$ then $Y$ is antisymmetric and $\Tr B^{\ft}Y = \Tr B^{\ft}_{\asym}Y = -(\Tr\Sigma YY + \Tr Y\Sigma Y)$, or
\begin{equation}\Tr B^{\ft}Y = 2\Tr Y^{\ft} \Sigma Y \geq 0.\label{eq:BY}
\end{equation}
\begin{lemma}For $\ttq = (x, y)\in \cQ$, let $x^{\ft}y = \Sigma U$ be the polar decomposition. For $\bdomg = (\omega, \bomg)\in \cE$
  \begin{gather}\rD_{\omega}\Sigma = \omega^{\ft}y U^{\ft} - 2\Sigma L^{-1}_{\Sigma}(\omega^{\ft}y U^{\ft})_{\asym} ;    \quad \brD_{\bomg}\Sigma = x^{\ft} \bomg U^{\ft} - 2\Sigma L^{-1}_{\Sigma}(x^{\ft} \bomg U^{\ft})_{\asym},\label{eq:DSigma}   \\
\rD_{\omega}(\alpha + \Tr\Sigma) =  \Tr \rD_{\omega}\Sigma =  \Tr \omega^{\ft}y U^{\ft} ;  \quad \Tr \brD_{\bomg}\Sigma =  \Tr x^{\ft} \bomg U^{\ft}, \label{eq:DK}\\
    \rD_{\omega} U = 2L^{-1}_{\Sigma}(\omega^{\ft}y U^{\ft})_{\asym} U;\quad \brD_{\bomg} U = 2L^{-1}_{\Sigma} (x^{\ft}\bomg U^{\ft})_{\asym} U.\label{eq:DU}
  \end{gather}
\end{lemma}  
\begin{proof}
Apply $L_{\Sigma}$ to the proposed expression for $\rD_{\omega}\Sigma$
$$\begin{gathered}
\Sigma \omega^{\ft}y U^{\ft} + \omega^{\ft}y U^{\ft}\Sigma
- 2\Sigma^2 L^{-1}_{\Sigma}(\omega^{\ft}y U^{\ft})_{\asym}
- 2L^{-1}_{\Sigma}(\omega^{\ft}y U^{\ft})_{\asym}\Sigma\\
=\Sigma \omega^{\ft}y U^{\ft} + \omega^{\ft}y U^{\ft}\Sigma
-2\Sigma L_{\Sigma}L^{-1}_{\Sigma}(\omega^{\ft}y U^{\ft})_{\asym}\\
= \Sigma \omega^{\ft}y U^{\ft} + \omega^{\ft}y U^{\ft}\Sigma
- \Sigma(\omega^{\ft}y U^{\ft} - U y^{\ft}\omega)
= \omega^{\ft} y y^{\ft} x + x^{\ft}yy^{\ft} \omega
=(\rD_{\omega}\Sigma)\Sigma + \Sigma(\rD_{\omega}\Sigma),
\end{gathered}
$$
(the last equality is by differentiating $\Sigma^2 = x^{\ft}yy^{\ft}x$), which proves the first part of \cref{eq:DSigma}. The second follows similarly. Equation (\ref{eq:DU}) follows by differentiating $\Sigma U = x^{\ft}y$
$$\begin{gathered}\rD_{\omega}(\Sigma U) = (\rD_{\omega}\Sigma) U + \Sigma (\rD_{\omega} U) = \omega^{\ft}y\\
  \Rightarrow (\omega^{\ft}y U^{\ft} - 2\Sigma L^{-1}_{\Sigma} (\omega^{\ft}y U^{\ft})_{\asym})U + \Sigma (\rD_{\omega} U) = \omega^{\ft}y \\
  \Rightarrow \rD_{\omega} U = 2 L^{-1}_{\Sigma} (\omega^{\ft}y U^{\ft})_{\asym}U,
\end{gathered}$$
 and similarly for $(\brD_{\bomg}\Sigma) U + \Sigma (\brD_{\bomg} U) = x^{\ft}\bomg$.
\end{proof}
\begin{proposition}\label{prop:antennaLift}
  With $(x, y)\in \cQ$ (defined by \cref{eq:cQFR}), let $x^{\ft}y = \Sigma U$ be the polar decomposition. Let $K = \alpha + \Tr\Sigma$. For $\bdxi_i = (\xi_i,\bxi_i), i=1,2$, the pairing
  \begin{equation}
\langle \bdxi_1, \bdxi_2 \rangle_* :=  - \frac{1}{2K^2}(\Tr(\xi_1^{\ft}yU^{\ft})\Tr(\bxi_2^{\ft}xU) + \Tr(\xi_2^{\ft}yU^{\ft})\Tr(\bxi_1^{\ft}xU)) 
            + \frac{1}{2K}(\Tr(\xi_1^{\ft}\bxi_2 U^{\ft}) + \Tr(\xi_2^{\ft}\bxi_1U^{\ft}))
\end{equation}
  is nondegenerated, with $\langle \bdxi_1, \bdxi_2\rangle_* = \bdxi_1 . \sfg \bdxi_2$ where ``$.$'' is the trace inner product on $\cE=(\R^{n\times k})^2$ and for $\bdomg =(\omega, \bomg) \in \cE$
  \begin{gather}
    \sfg \bdomg := (- \frac{1}{2K^2}\Tr (\bomg^{\ft}x U)yU^{\ft} + \frac{1}{2K}\bomg U^{\ft}, - \frac{1}{2K^2}(\Tr(\omega^{\ft}yU^{\ft})xU + \frac{1}{2K}\omega U),\\
    \sfg^{-1}\bdomg := (2K \bomg U^{\ft} + \frac{2K}{\alpha} \Tr(y^{\ft}\bomg)x,
    2K\omega U + \frac{2K}{\alpha}\Tr(x^{\ft}\omega)y).
  \end{gather}
With this metric, the map $(x, y)\mapsto (\lb x\rb, \lb y\rb)$ is a semi-Riemannian submersion onto $\cB$ defined by \cref{eq:cBFR}, 
%  \begin{equation} \cB := \{(\lb x\rb, \lb y\rb)\in \Sd{n}{k}\;| x^{\ft}y \text{ is invertible }\}.\end{equation}
where $\cB$ is equipped with the Kim-McCann metric defined by $c(\lb x\rb, \lb y\rb) = -\log(\alpha + \Tr\Sigma)$. A vector $\bdxi=(\xi, \bxi) \in (\R^{n\times k})^2$ is horizontal if and only if $x^{\ft}\bxi U^{\ft}$ and $y^{\ft}\xi U$ are symmetric. In particular, $\rD_{\xi} U$ and $\brD_{\bxi} U$ vanish for horizontal $\bdxi$.
\end{proposition}
\begin{proof}It is clear $\langle .\rangle_*$ is symmetric, and given by $\sfg$. Let
$(\sfg \circ \sfg^{-1}\bdomg) = (\mu_x, \mu_y)$ then  
  $$\mu_x =
    - \frac{1}{2K^2}\Tr ((2K\omega U + \frac{2K}{\alpha}\Tr(x^{\ft}\omega)y)^{\ft}xU) )yU^{\ft} + \frac{1}{2K}(2K\omega U + \frac{2K}{\alpha}\Tr(x^{\ft}\omega)y) U^{\ft}.\\
    $$
    The first term is
    $$\begin{gathered}
      - \frac{1}{2K^2}(2K\Tr (U^{\ft}\omega^{\ft} xU) + \frac{2K}{\alpha}\Tr(x^{\ft}\omega)\Tr(y^{\ft}xU))yU^{\ft}
      = \frac{-1}{K\alpha}\Tr (x^{\ft}\omega)(\alpha + \Tr\Sigma)yU^{\ft} = \frac{-1}{\alpha}\Tr (x^{\ft}\omega)yU^{\ft}
    \end{gathered}$$
which offsets the opposite expression in the second term, leaving $\mu_x = \omega$. We get $\mu_y = \bomg$ similarly. From here, $\sfg$ is invertible and $\langle .\rangle_*$ is a semi-Riemannian pairing.

    Observe $\langle .\rangle_*$ is nondegenerate on the vertical space, as if $\mathring{\psi}_1=(xa_1, yb_1), \mathring{\psi}_2=(xa_2, yb_2)$ are vertical vectors ($a_i, b_i\in \oo(k), i=1,2$), then since $\Tr b_1 y^{\ft}xU = \Tr b_1 U^{\ft}\Sigma U = 0, \Tr a_1 x^{\ft}y U^{\ft} = \Tr a_1 \Sigma=0$ (symmetric and antisymmetric matrices are orthogonal under trace inner product) 
$$\langle \mathring{\psi}_1, \mathring{\psi}_2 \rangle_* :=  
    \frac{1}{2K}(\Tr(-a_1 x^{\ft}y b_2 U^{\ft}) + \Tr(-a_2x^{\ft} yb_1U^{\ft}))
    = \frac{1}{2K}(\Tr(-a_1 \Sigma U b_2 U^{\ft}) + \Tr(-a_2\Sigma U b_1U^{\ft})).
$$
    Since $U b_2 U^{\ft}, Ub_1U^{\ft}$ are antisymmetric, we need to show the pairing $-\Tr a\Sigma b$ is nondegenerate for two antisymmetric matrices $a, b\in\oo(k)$. But this pairing is positive definite, so $\langle\rangle_*$ is nondegenerate on the vertical space, {\it hence on the horizontal space}.
    
A horizontal vector $\bdxi$ is orthogonal to all vectors of the form $(xa, yb)$ for $a, b\in \oo(k)$, or
$$\begin{gathered}
0 = - \frac{1}{2K^2}(\Tr(\xi^{\ft}yU^{\ft})\Tr((-by^{\ft}xU) + \Tr(-ax^{\ft}yU^{\ft})\Tr(\bxi^{\ft}xU)) 
            + \frac{1}{2K}(\Tr(\xi^{\ft}yb U^{\ft}) + \Tr(-ax^{\ft}\bxi U^{\ft}))\\
\end{gathered}
$$
which reduces to $\Tr(\xi^{\ft}yb U^{\ft}) + \Tr(-ax^{\ft}\bxi U^{\ft})=0$ for $a, b\in \oo(k)$. Thus,$(U^{\ft}\xi^{\ft} y)_{\asym} = (x^{\ft}\bxi U^{\ft})_{\asym} = 0$, giving us the description of horizontal space.

Finally, for a horizontal vector $\bdxi$
$$\begin{gathered}
\langle \qq(\bdxi), \qq(\bdxi)\rangle_{KM} = -\frac{\partial^2}{\partial s\partial t}c(x+s\xi, y+ t\bxi)_{s=0,t=0} = \rD_{\xi}\frac{\Tr(x^{\ft}\bxi U^{\ft})}{\alpha + \Tr\Sigma}\\
  =\frac{\Tr(\xi^{\ft}\bxi U^{\ft})}{\alpha + \Tr\Sigma}
  - \frac{\Tr(x^{\ft}\bxi U^{\ft})
\Tr (y^{\ft}\xi U)
  }{(\alpha + \Tr\Sigma)^2} = \langle \bdxi, \bdxi\rangle_*.
\end{gathered}$$
Thus, the differential submersion $\qq:\cQ\to\cB$ is a semi-Riemannian submersion.
\end{proof}
\begin{theorem}\label{thm:crossSecFR} Let $\ttq = (x, y)\in \cQ$ and $\bdomg=(\omega,\bomg)\in \cE$. The projection to the horizontal bundle is
  \begin{equation}\ttH(\ttq)\bdomg = (\omega - 2xL_{\Sigma}^{-1} (U y^{\ft}\omega)_{\asym}, 
  \bomg - 2yU^{\ft} L^{-1}_{\Sigma} (x^{\ft}\bomg U^{\ft})_{\asym} U).\label{eq:fixrankHproj}
  \end{equation}
  The Christoffel function $\Gamma^{\cQ}(\ttq; \bdomg_1, \bdomg_2) = (\Gamma^{\cQ}_x, \Gamma^{\cQ}_y)$ of the Levi-Civita connection of the metric $\sfg$ on $\cQ$ evaluated at
two vectors $\bdomg_1=(\omega_1, \bomg_1), \bdomg_2=(\omega_2, \bomg_2)\in \cE$ is given by
  \begin{equation}\begin{gathered}
      \Gamma^{\cQ}_x =-\frac{1}{K}\left(\Tr(\omega_2^{\ft}yU^{\ft})\omega_1 
    + \Tr(\omega_1^{\ft}yU^{\ft})\omega_2 \right)
    + xL^{-1}_{\Sigma}(U (\bomg_1^{\ft} \omega_2 + \bomg_2^{\ft}\omega_1))_{\asym} \\
    + \omega_1 L^{-1}_{\Sigma}(x^{\ft}\bomg_2 U^{\ft}
    +\omega_2^{\ft}y U^{\ft})_{\asym}    
    + \omega_2 L^{-1}_{\Sigma}
    (x^{\ft}\bomg_1 U^{\ft} + \omega_1^{\ft} y U^{\ft})_{\asym}
    \\    
    + \frac{1}{K}\left(\Tr(x^{\ft} \bomg_2 U^{\ft}
    - \omega_2^{\ft} yU^{\ft})xL^{-1}_{\Sigma} (\omega_1^{\ft}yU^{\ft})_{\asym}
    + \Tr(x^{\ft} \bomg_1 U^{\ft}
    - \omega_1^{\ft} yU^{\ft})xL^{-1}_{\Sigma} (\omega_2^{\ft}yU^{\ft})_{\asym}
    \right)
    ,    \label{eq:GammaxFR}
\end{gathered}    
  \end{equation}
  \begin{equation}
    \begin{gathered}
    \Gamma^{\cQ}_y = -\frac{1}{K}\left(\Tr(\bomg_2^{\ft}xU)\bomg_1
    + \Tr(\bomg_1^{\ft}x U) \bomg_2 \right)
    - yU^{\ft}L^{-1}_{\Sigma} (U(\bomg_1^{\ft} \omega_2 + \bomg_2^{\ft}\omega_1) )_{\asym} U \\
    + \bomg_1 U^{\ft} L^{-1}_{\Sigma} (Uy^{\ft}\omega_2 +  U\bomg_2^{\ft}x)_{\asym}U     
    + \bomg_2 U^{\ft} L^{-1}_{\Sigma} (Uy^{\ft}\omega_1 + U\bomg_1^{\ft}x)_{\asym}U \\
    + \frac{1}{K}\left(\Tr(\omega_2^{\ft}yU^{\ft} - x^{\ft}\bomg_2 U^{\ft})yU^{\ft} L^{-1}_{\Sigma} (U\bomg_1^{\ft}x)_{\asym} U
    + \Tr(\omega_1^{\ft}y U^{\ft} - x^{\ft}\bomg_1 U^{\ft})yU^{\ft}L^{-1}_{\Sigma} (U\bomg_2^{\ft}x)_{\asym} U
        \right).\label{eq:GammayFR}
    \end{gathered}
  \end{equation}
For a horizontal vector $\bdeta$, let $\bdeta_{x0} = (\eta, 0), \bdeta_{0y} = (0, \bareta)$. The O'Neill tensor and the cross-curvature of $c(x, y)$ on $\cB$ are  
  \begin{gather}\rA_{\bdeta_{x0}}\bdeta_{0y}= (x L^{-1}_{\Sigma} (U \bareta^{\ft}\eta)_{\asym},
    - y U^{\ft} L^{-1}_{\Sigma} (\eta^{\ft}\bareta U^{\ft}))_{\asym}U),\label{eq:ONeilFr}\\
  \langle\rA_{\bdeta_{x0}}\bdeta_{0y}  ,\rA_{\bdeta_{x0}}\bdeta_{0y}\rangle_* = \frac{1}{K}\Tr(
    L^{-1}_{\Sigma}(U\bareta^{\ft}\eta)_{\asym}
    \Sigma L^{-1}_{\Sigma} (\eta^{\ft}\bareta U^{\ft})_{\asym}) \geq 0,\label{eq:ONeilFrNorm}\\
    \crosscv(\bdeta) = 4\langle\rA_{\bdeta_{x0}}\bdeta_{0y}  ,\rA_{\bdeta_{x0}}\bdeta_{0y}\rangle_* - \langle\bdeta, \bdeta \rangle_*^2.\label{eq:crossFr}
  \end{gather}
If $\langle \bdeta,\bdeta\rangle_*=0$, the cross-curvature is nonnegative, it is zero if and only if $U \bareta^{\ft}\eta$ is symmetric.
\end{theorem}
\begin{proof}With $\ttH$ given in \cref{eq:fixrankHproj}, $\ttV(\ttq) = I_{\cE} - \ttH(\ttq)$ maps to the vertical space, and for $a \in \oo(k)$,
  $$2L^{-1}_{\Sigma}(Uy^{\ft}xa)_{\asym} = 2L^{-1}_{\Sigma}(\Sigma a)_{\asym} = L^{-1}_{\Sigma}(\Sigma a + a\Sigma) = a,
  $$
similarly, $2U^{\ft}L^{-1}_{\Sigma}(x^{\ft}ya U^{\ft})_{\asym}U = a$, hence $\ttV(\ttq)$ is idempotent. It remains to show $\ttH(q)\bdomg$ is horizontal by first verifying $y^{\ft}(\omega - 2xL_{\Sigma}^{-1} (U y^{\ft}\omega)_{\asym}) U$ is symmetric, by expanding to
  $$\begin{gathered}
  y^{\ft}\omega U - 2U^{\ft}\Sigma L_{\Sigma}^{-1} (U y^{\ft}\omega)_{\asym} U
  =y^{\ft}\omega U - 2U^{\ft}(
  (U y^{\ft}\omega)_{\asym} -    L_{\Sigma}^{-1} (U y^{\ft}\omega)_{\asym} \Sigma) U\\
  = y^{\ft}\omega U - y^{\ft}\omega U + U^{\ft}\omega^{\ft} y  + 2U^{\ft}L_{\Sigma}^{-1} (U y^{\ft}\omega)_{\asym} \Sigma U\\
  = U^{\ft}\omega^{\ft} y  + 2U^{\ft}L_{\Sigma}^{-1} (U y^{\ft}\omega)_{\asym} x^{\ft}y
  = U^{\ft}(\omega - 2xL_{\Sigma}^{-1} (U y^{\ft}\omega)_{\asym})^{\ft} y,
\end{gathered}
$$
which is symmetric. A similar computation for $\bomg $ confirms $\ttH(q)\bdomg$ is horizontal.
The Christoffel function is {\it derived} from $\frac{1}{2}\sfg^{-1}(\rD_{\bdomg}\sfg\eta + \rD_{\eta}\sfg\bdomg - \chi_{\sfg}(\bdomg, \bdeta))$, the derivation is straightforward, given in detail in \cite{NguyenVerify}. We will verify it here. The formula is torsion free, hence, we need to show
\begin{equation} \rD_{\omega_1}\langle \bdomg_2, \bdomg_2\rangle_* + \brD_{\bomg_1}\langle \bdomg_2,\bdomg_2\rangle_*=
2\langle \bdomg_2, \Gamma^{\cQ}(\ttq; \bdomg_1, \bdomg_2)\rangle_*.\label{eq:checkGammaKM}
\end{equation}
For a function $F$ in $(x, y)$, denote $\bdD_{\bdomg}F =\rD_{\omega}F + \brD_{\bomg}F$, the left-hand side is rearranged to
$$\begin{gathered}2 \frac{\bdD_{\bdomg}K}{K^3}(
            \Tr(U y^{\ft} \omega_2)\Tr(U \bomg_2^{\ft}x)) \
            - \frac{\bdD_{\bdomg}K}{K^2}\Tr(U\bomg_2^{\ft}\omega_2) \
            \\
            - \frac{1}{K^2}\Tr( y^{\ft}\omega_2\bdD_{\bdomg} U)\Tr(U\bomg_2^{\ft}x) \
            - \frac{1}{K^2}\Tr(U\bomg_1^{\ft}\omega_2)\Tr(U\bomg_2^{\ft}x) \\
            - \frac{1}{K^2}\Tr(Uy^{\ft}\omega_2)\Tr(\bomg_2^{\ft}x \bdD_{\bdomg} U) \
            - \frac{1}{K^2}\Tr(U y^{\ft}\omega_2)\Tr(U\bomg_2^{\ft}\omega_1) \
            + \frac{1}{K}\Tr(\bomg_2^{\ft}\omega_2\bdD_{\bdomg} U).
\end{gathered}            
$$
%The right-hand side is $2\Tr(\sfg \bdomg_2)_x^{\ft} \Gamma_x + 2\Tr(\sfg \bdomg_2)_y^{\ft} \Gamma_y$ where for $A\in \cE$ we write $A_x, A_y$ for its two components. It expands to
The right-hand side expands to
$$  - \frac{1}{K^2}(\Tr (\bomg_2^{\ft}x U)\Tr (Uy^{\ft}\Gamma^{\cQ}_x)
+ \Tr (\omega_2^{\ft}y U^{\ft})\Tr (U^{\ft}x^{\ft}\Gamma^{\cQ}_y))
+ \frac{1}{K}(\Tr (U \bomg_2^{\ft}\Gamma^{\cQ}_x)
 + \Tr (U^{\ft} \omega_2^{\ft}\Gamma^{\cQ}_y)).
 $$
 When expanding $\Tr Uy^{\ft}\Gamma^{\cQ}_x$, terms of the form $\Tr Uy^{\ft}x a =\Tr \Sigma a $ vanish for $a\in \oo(k)$, and similar for $\Tr U^{\ft}x^{\ft}ya$. After some manipulations, using \cref{eq:LSigSym}
$$\begin{gathered}
\Tr Uy^{\ft} \Gamma^{\cQ}_x= -\frac{2}{K}\Tr(\omega_2^{\ft}yU^{\ft})\rD_{\omega_1}K
    + \Tr Uy^{\ft}\omega_1 L^{-1}_{\Sigma}(x^{\ft}\bomg_2 U^{\ft}
    +\omega_2^{\ft}y U^{\ft})_{\asym}   + \frac{1}{2}\Tr y^{\ft}\omega_2(\rD_{\omega_1}U +\brD_{\bomg_1}U),\\
\Tr Uy^{\ft} \Gamma^{\cQ}_x=  -\frac{2}{K}\Tr(\omega_2^{\ft}yU^{\ft})\rD_{\omega_1}K
    + \frac{1}{2}\Tr ((\bomg_2^{\ft}x +  y^{\ft}\omega_2)\rD_{\omega_1}U)   + \frac{1}{2}\Tr y^{\ft}\omega_2\bdD_{\bdomg}U,\\  
\Tr U^{\ft}x^{\ft} \Gamma^{\cQ}_y = -\frac{2}{K}\Tr(\bomg_2^{\ft}xU)\brD_{\bomg_1}K
+\frac{1}{2}\Tr((\bomg_2^{\ft}x + y^{\ft}\omega_2  )\brD_{\bomg_1}U)   + \frac{1}{2}\Tr\bomg_2^{\ft} x\bdD_{\bdomg}U.
%(\rD_{\omega_1}U +\brD_{\bomg_1}U).
 \end{gathered}$$
 %\Tr U^{\ft}x^{\ft}\bomg_1 L^{-1}_{\Sigma}(U(y^{\ft}\omega_2   +\bomg_2^{\ft}x))_{\asym}
On the other hand,
$$\begin{gathered}\Tr U\bomg_2^{\ft} \Gamma^{\cQ}_x =
 -\frac{1}{K}\left(\Tr(\omega_2^{\ft}yU^{\ft})\Tr U\bomg_2^{\ft} \omega_1 
    + \Tr(\omega_1^{\ft}yU^{\ft})\Tr U\bomg_2^{\ft} \omega_2 \right)
    + \Tr U\bomg_2^{\ft} xL^{-1}_{\Sigma}(U (\bomg_1^{\ft} \omega_2 + \bomg_2^{\ft}\omega_1))_{\asym} \\
    + \Tr U\bomg_2^{\ft} \omega_1 L^{-1}_{\Sigma}(x^{\ft}\bomg_2 U^{\ft}
    +\omega_2^{\ft}y U^{\ft})_{\asym}    
    + \Tr U\bomg_2^{\ft} \omega_2 L^{-1}_{\Sigma}
    (x^{\ft}\bomg_1 U^{\ft} + \omega_1^{\ft} y U^{\ft})_{\asym}
    \\    
    + \frac{1}{K}\left(\Tr(x^{\ft} \bomg_2 U^{\ft}
    - \omega_2^{\ft} yU^{\ft})\Tr U\bomg_2^{\ft} xL^{-1}_{\Sigma} (\omega_1^{\ft}yU^{\ft})_{\asym}
    + \Tr(x^{\ft} \bomg_1 U^{\ft}
    - \omega_1^{\ft} yU^{\ft})\Tr U\bomg_2^{\ft} xL^{-1}_{\Sigma} (\omega_2^{\ft}yU^{\ft})_{\asym}
    \right)\\
    = -\frac{1}{K}\Tr(\omega_2^{\ft}yU^{\ft})\Tr U\bomg_2^{\ft} \omega_1 
    -\frac{\rD_{\omega_1}K}{K} \Tr U\bomg_2^{\ft} \omega_2 
    + \Tr L^{-1}_{\Sigma}(U\bomg_2^{\ft} x )_{\asym} U(\bomg_1^{\ft} \omega_2 + \bomg_2^{\ft}\omega_1) \\
    + \Tr U\bomg_2^{\ft} \omega_1 L^{-1}_{\Sigma}(x^{\ft}\bomg_2 U^{\ft}
    +\omega_2^{\ft}y U^{\ft})_{\asym}    
    + \frac{1}{2}\Tr \bomg_2^{\ft} \omega_2 (\rD_{\omega_1}U + \brD_{\bomg_1}U)
    \\    
    + \frac{1}{2K}\Tr(x^{\ft} \bomg_2 U^{\ft}
    - \omega_2^{\ft} yU^{\ft})\Tr \bomg_2^{\ft} x\rD_{\omega_1}U
    + \frac{1}{K}\Tr(x^{\ft} \bomg_1 U^{\ft}
    - \omega_1^{\ft} yU^{\ft})\Tr U\bomg_2^{\ft} xL^{-1}_{\Sigma} (\omega_2^{\ft}yU^{\ft})_{\asym}.    
\end{gathered}$$
$$\begin{gathered}
  \Tr U^{\ft}\omega_2^{\ft}\Gamma^{\cQ}_y = - \frac{1}{K}\Tr(\bomg_2^{\ft}xU)\Tr U^{\ft}\omega_2^{\ft}\bomg_1 
            - \frac{\brD_{\bomg_1}K}{K}\Tr U^{\ft} \omega_2^{\ft}\bomg_2  
            + \Tr(L_{\Sigma}^{-1}(Uy^{\ft}\omega_2)_{\asym}U(\bomg_2^{\ft}\omega_1 + \bomg_1^{\ft}\omega_2))\\
            + \Tr(U\bomg_1^{\ft}\omega_2 L_{\Sigma}^{-1} (x^{\ft}\bomg_2U^{\ft}+\omega_2^{\ft}yU^{\ft})_{\asym}) \
            + \frac{1}{2}\Tr(\bomg_2^{\ft}\omega_2(\rD_{\omega_1}U + \brD_{\bomg_1}U)) \\
            + \frac{1}{2K}\Tr(\omega_2^{\ft}yU^{\ft}
                               - x^{\ft}\bomg_2U^{\ft})\Tr(y^{\ft}\omega_2\brD_{\bomg2}U) 
            + \frac{1}{K}\Tr(\omega_1^{\ft}yU^{\ft}
            - x^{\ft}\bomg_1U^{\ft})\Tr(\omega_2^{\ft}yU^{\ft}L_{\Sigma}^{-1}(U\bomg_2^{\ft}x))_{\asym}).
\end{gathered}$$
To verify \cref{eq:checkGammaKM}, we show terms with  $L_{\Sigma^{-1}}$ from
$\Tr U\bomg_2^{\ft} \Gamma^{\cQ}_x$ and $\Tr U^{\ft}\omega_2^{\ft}\Gamma^{\cQ}_y$ 
 vanish:
$$\begin{gathered}
    \Tr L^{-1}_{\Sigma}(U\bomg_2^{\ft} x )_{\asym} U(\bomg_1^{\ft} \omega_2 + \bomg_2^{\ft}\omega_1) 
    + \Tr U\bomg_2^{\ft} \omega_1 L^{-1}_{\Sigma}(x^{\ft}\bomg_2 U^{\ft}  +\omega_2^{\ft}y U^{\ft})_{\asym}    \\
    + \frac{1}{K}\Tr(x^{\ft} \bomg_1 U^{\ft}
    - \omega_1^{\ft} yU^{\ft})\Tr U\bomg_2^{\ft} xL^{-1}_{\Sigma} (\omega_2^{\ft}yU^{\ft})_{\asym}    \\
    + \Tr(L_{\Sigma}^{-1}(Uy^{\ft}\omega_2)_{\asym}U(\bomg_2^{\ft}\omega_1 + \bomg_1^{\ft}\omega_2))
    + \Tr(U\bomg_1^{\ft}\omega_2 L_{\Sigma}^{-1} (x^{\ft}\bomg_2U^{\ft}+\omega_2^{\ft}yU^{\ft})_{\asym}) \\
    + \frac{1}{K}\Tr(\omega_1^{\ft}yU^{\ft}
    - x^{\ft}\bomg_1U^{\ft})\Tr(\omega_2^{\ft}yU^{\ft}L_{\Sigma}^{-1}(U\bomg_2^{\ft}x))_{\asym})=0.  
\end{gathered}
$$
This follows if we use \cref{eq:LSigSym} to cancel the terms with $\frac{1}{K}$, and group the remaining terms
$$\begin{gathered}
    \Tr L^{-1}_{\Sigma}(U\bomg_2^{\ft} x + Uy^{\ft}\omega_2)_{\asym} U(\bomg_1^{\ft} \omega_2 + \bomg_2^{\ft}\omega_1) \\
    + \Tr (U\bomg_2^{\ft} \omega_1 + U\bomg_1^{\ft}\omega_2)L^{-1}_{\Sigma}(x^{\ft}\bomg_2 U^{\ft}  +\omega_2^{\ft}y U^{\ft})_{\asym} = 0.
\end{gathered}
$$

The O'Neill's tensor in \cref{eq:ONeilFr} follows from \cref{eq:fixrankHproj} and \cref{eq:oneilLie}, using the horizontal condition, while  \cref{eq:ONeilFrNorm} follows from the definition and the below, noting $x^{\ft}y U^{\ft}=\Sigma$
$$\begin{gathered}\langle \rA_{\eta_{x0}}\eta_{0y},\rA_{\eta_{x0}}\eta_{0y}\rangle_* 
  = \frac{1}{K}\Tr L^{-1}_{\Sigma} (U \bareta^{\ft}\eta)_{\asym}x^{\ft}y U^{\ft} L^{-1}_{\Sigma} (\eta^{\ft}\bareta U^{\ft})_{\asym}U U^{\ft}.
  \end{gathered}
$$
It remains to compute the ambient cross-curvature, $\langle \RcQ_{\bdeta_{x0},\bdeta_{0y}}\bdeta_{0y},\bdeta_{x0}\rangle_*$ then use O'Neill's formula. Using \cref{eq:rc1}, note we only need to compute the $y$-component of $\RcQ$, and we can remove the vertical summands. For any ambient vector $\bdomg = (\omega, \bomg)$, we write $\bdomg_{x0} = (\omega, 0), \bdomg_{0y} = (0, \bomg)$ then
$$\Gamma^{\cQ}_y(\bdomg_{0y}, \bdomg_{0y}) =  - \frac{2}{K}\Tr(\bomg^{\ft}xU)\bomg
               + 2\bomg U^{\ft} L^{-1}_{\Sigma} (U \bomg^{\ft}x)_{\asym}U
               - \frac{2}{K}\Tr(x^{\ft}\bomg U^{\ft})y U^{\ft} L^{-1}_{\Sigma}(U \bomg^{\ft}x)_{\asym}U,$$
$$    \Gamma^{\cQ}_y(\bdomg_{x0}, \bdomg_{0y})  = - y U^{\ft} L^{-1}_{\Sigma} (U\bomg^{\ft}\omega)_{\asym}U \
        + \bomg U^{\ft} L^{-1}_{\Sigma}(U y^{\ft}\omega)_{\asym}U 
        + \frac{1}{K}\Tr(\omega^{\ft}yU^{\ft})y U^{\ft} L^{-1}_{\Sigma}(U\bomg^{\ft} x)_{\asym} U.
        $$
From here, for horizontal $\bdeta$, the $y$-component of $\bdD_{\bdeta_{x0}}\Gamma^{\cQ}_y(\bdeta_{0y}, \bdeta_{0y})$ is
$$
\frac{2}{K^2}\Tr(y^{\ft}\eta U)\Tr(\bareta^{\ft} x U)\bareta
- \frac{2}{K}\Tr(\bareta^{\ft}\eta U) \bareta              
              + 2\bareta U^{\ft}L^{-1}_{\Sigma}(U\bareta^{\ft}\eta)_{\asym}U
              - \frac{2}{K}\Tr(x^{\ft} \bareta U^{\ft})y U^{\ft}L^{-1}_{\Sigma} (U\bareta^{\ft}\eta)_{\asym}U.\\
$$
The last term is vertical, thus, the contribution of $\langle \bdD_{\eta_{x0}}\Gamma^{\cQ}_y(\bdeta_{0y}, \bdeta_{0y}), \bdeta_{x0}\rangle_*$ is
$$\begin{gathered}
                -\frac{1}{K^4}\Tr(U y^{\ft}\eta)^2\Tr(\bareta^{\ft} x U)^2  
                +\frac{1}{K^3}\Tr(U y^{\ft}\eta) \Tr(\bareta^{\ft}\eta U)\Tr(  \bareta^{\ft} x U) \\
                +\frac{1}{K^3}\Tr(y^{\ft}\eta U)\Tr(\bareta^{\ft} x U)\Tr\eta^{\ft} \bareta U^{\ft}
                -\frac{1}{K^2}\Tr(\bareta^{\ft}\eta U)\Tr \eta^{\ft}  \bareta U^{\ft}
                +\frac{1}{K}\Tr \eta^{\ft} \bareta U^{\ft}L^{-1}_{\Sigma}(U\bareta^{\ft}\eta)_{\asym}\\
                = - \langle \bdeta, \bdeta \rangle_*^2 + \frac{1}{K}\Tr\eta^{\ft} \bareta U^{\ft}L^{-1}_{\Sigma}(U\bareta^{\ft}\eta)_{\asym}
              \end{gathered}$$
as $\Tr(U^{\ft} x^{\ft} \bareta U^{\ft}L^{-1}_{\Sigma}(U\bareta^{\ft}\eta)_{\asym}U)$ vanishes. There is no contribution from $\bdD_{\bdeta_{0y}}\Gamma^{\cQ}_y(\bdeta_{x0}, \bdeta_{0y})_y$, since after taking derivative then leave out terms of the form $ya$ for $a\in \oo(k)$, the remain terms are
$$- \bareta U^{\ft} L^{-1}_{\Sigma} (U\bareta^{\ft}\eta)_{\asym}U \
        + \bareta U^{\ft} L^{-1}_{\Sigma}(U \bareta^{\ft}\eta)_{\asym}U =0.
$$
Next, using horizontal condition
$$\begin{gathered}\Gamma(\bdeta_{0y}, \bdeta_{0y}) =  (0, - \frac{2}{K}\Tr(\bareta^{\ft}xU)\bareta),\\
\Gamma(\eta_{x0}, \eta_{0y})  = (x L^{-1}_{\Sigma} (U\bareta^{\ft}\eta)_{\asym},
- y U^{\ft} L^{-1}_{\Sigma} (U\bareta^{\ft}\eta)_{\asym}U).
\end{gathered}
$$
$\Gamma(\bdeta_{x0}, \Gamma(\bdeta_{0y}, \bdeta_{0y}))_y$ is also vertical as $(Uy^{\ft}\eta)_{\asym} = 0$, while after simplifying vertical summands, $\Gamma(\bdeta_{0y}, \Gamma(\bdeta_{x0}, \bdeta_{0y}))_y$ is left with the first term in the second line of \cref{eq:GammayFR},
$$\bareta U^{\ft}L^{-1}_{\Sigma} \left(Uy^{\ft}x L^{-1}_{\Sigma} (U\bareta^{\ft}\eta)_{\asym} +UU^{\ft}
 L^{-1}_{\Sigma} (U\bareta^{\ft}\eta)_{\asym}Uy^{\ft}x\right)_{\asym}U = \bareta U^{\ft}L^{-1}_{\Sigma} (U \bareta^{\ft} \eta)_{\asym}U$$
using $Uy^{\ft}x = \Sigma$ then $L_{\Sigma}L_{\Sigma}^{-1}=\dI_k$. The contribution $-\langle \Gamma(\bdeta_{0y}, \Gamma(\bdeta_{x0}, \bdeta_{0y})), \bdeta_{x0}\rangle_*$ is
$$
\frac{1}{2K^2}\Tr(\eta^{\ft}y U^{\ft})\Tr U^{\ft}x^{\ft} \bareta U^{\ft}L^{-1}_{\Sigma} (U \bareta^{\ft} \eta)_{\asym}U
              - \frac{1}{2K}\Tr \eta^{\ft}\bareta U^{\ft}L^{-1}_{\Sigma} (U \bareta^{\ft} \eta)_{\asym},
              $$
where $\Tr U^{\ft}x^{\ft} \bareta U^{\ft}L^{-1}_{\Sigma} (U \bareta^{\ft} \eta)_{\asym}U$ vanishes by the horizontal condition, thus
\begin{equation}
\langle \RcQ_{\eta_{x0},\eta_{0y}}\eta_{0y},\eta_{x0}\rangle_* = - \langle \bdeta, \bdeta \rangle_*^2 + \frac{1}{2K}\Tr\eta^{\ft} \bareta U^{\ft}L^{-1}_{\Sigma}(U\bareta^{\ft}\eta)_{\asym}
= - \langle \bdeta, \bdeta \rangle_*^2 + \langle\rA_{\bdeta_{x0}}\bdeta_{0y}  ,\rA_{\bdeta_{x0}}\bdeta_{0y}\rangle_*\label{eq:RcQFR}
\end{equation}
using \cref{eq:BY} . Equation (\ref{eq:crossFr}) follows from the O'Neill's formula. 
\end{proof}
\subsection{The reflector antenna for Grassmann manifolds}\label{sec:reflectGrass}
The Stiefel manifold $\St{n}{k}$ defined by $x^{\ft}x = \dI_k$ is a submanifold of $\R_*^{n\times k}$. Its quotient under the right action of $\OO(k)$ is the Grassmann manifold
$\Gr{n}{k}$. We now study the reflector antenna cost on $\Gr{n}{k}$. The case $k=1$ is the projective form of the classical reflector antenna \cite{KimMcCann,MTW,Loeper}.
\begin{theorem}\label{theo:crossgrass}Restricted to $\cQ_{\Sto} =\cQ_{\St{n}{k}}:=(\St{n}{k}\times \St{n}{k})\cap \cQ_{\R_*^{n\times k}}$, the metric $\sfg$ is nondegenerate. Its tangent bundle is charaterized by the condition $(x^{\ft}\omega)_{\sym} = (y^{\ft}\bomg)_{\sym}=0$, for $(x, y)\in\cQ_{\St{n}{k}}, \bdomg=(\omega, \bomg)\in\cE=(\R^{n\times k})^2$. The projection from $\cE$ to $T_{x,y}\cQ_{\Sto}$ is given by
  \begin{equation}\begin{gathered}
\Pi^{\Sto}(\ttq)\bdomg =  (\omega - 2yU^{\ft}L^{-1}_{\Sigma} (x^{\ft}\omega)_{\sym} \
    - (\alpha+\Tr\Sigma^{-1})^{-1} \Tr(\Sigma^{-1}x^{\ft}\omega)(x - yU^{\ft}\Sigma^{-1}),\\
    \bomg - 2x L^{-1}_{\Sigma} (U y^{\ft}\bomg U^{\ft})_{\sym} U \
    - (\alpha+\Tr \Sigma^{-1})^{-1}\Tr(U^{\ft} \Sigma^{-1}U y^{\ft}\bomg)(y - x\Sigma^{-1} U)).\end{gathered}
\end{equation}  
  The action of $\OO(k)^2$ on $\cQ_{\R_*^{n\times k}}$ restricts to an action on $\cQ_{\Sto}$. The corresponding horizontal bundle $\cH_{\Gro}$ is characterized by the additional conditions $(\bomg^{\ft}xU)_{\asym} = (\omega^{\ft}y U^{\ft})_{\asym}=0$. On this bundle, $\sfg$ is nondegenerate, and $\qq: \cQ_{\Sto}\to\cB_{\Gro}:= \cQ_{\Sto}\slash \OO(k)^2$ is a semi-Riemannian submersion. The (Grassmann) horizontal projection is
  \begin{equation}\begin{gathered}
\ttH^{\Gro}(\ttq)\bdomg = (    \omega
            - 2x L^{-1}_{\Sigma} (Uy^{\ft}\omega)_{\asym}
            - 2y U^{\ft} L^{-1}_{\Sigma}(x^{\ft}\omega)_{\sym} - (\alpha+\Tr\Sigma^{-1})^{-1}\Tr(\Sigma^{-1} x^{\ft} \omega) (x - y U^{\ft} \Sigma^{-1}),\\
            \bomg  - 2y U^{\ft} L^{-1}_{\Sigma}(x^{\ft} \bomg U^{\ft})_{\asym} U 
            - 2x L^{-1}_{\Sigma}(Uy^{\ft} \bomg U^{\ft})_{\sym} U - (\alpha+\Tr\Sigma^{-1})^{-1}\Tr(U^{\ft} \Sigma^{-1} U y^{\ft} \bomg) (y - x\Sigma^{-1} U)
            ).\label{eq:HGro}
\end{gathered}          
  \end{equation}
  If $x_{\perp}, y_{\perp}\in \St{n}{n-k}$ are complement orthogonal basis to $x, y$ then
for $B, \bB\in \R^{(n-k)\times k}$, the following vector is horizontal at $(x, y)$, and any horizontal vector is of this form
\begin{equation}\bdeta = (x_{\perp}B - 2xL^{-1}_{\Sigma} (U y^{\ft} x_{\perp} B)_{\asym},
            y_{\perp}\bB - 2y U^{\ft}L^{-1}_{\Sigma} (x^{\ft}y_{\perp}\bB U^{\ft})_{\asym}U).\label{eq:GrassTanKM}
\end{equation}  
For a horizontal vector $\bdeta=(\eta, \bareta)$ on $\cQ_{\St{n}{k}}$, the cross-curvature of the Kim-McCann metric with the reflector antenna cost in \cref{eq:cAntenna} for a pair of Grassmann manifolds $(\Gr{n}{k})^2$ is given by
\begin{equation}
  \begin{gathered}
    \crosscv_{\Gro}(\bdeta) = - \langle \bdeta, \bdeta\rangle_*^2 \
            + \frac{4}{K}\Tr(
                L^{-1}_{\Sigma}(U\bareta^{\ft} \eta)_{\asym}
                \Sigma L^{-1}_{\Sigma}(\eta^{\ft} \bareta U^{\ft})_{\asym} \\
                + \frac{2}{K}\Tr(L^{-1}_{\Sigma}(\eta^{\ft} \eta)\Sigma L^{-1}_{\Sigma}(U \bareta^{\ft} \bareta U^{\ft}))
                - \frac{1}{2K(\alpha+\Tr\Sigma^{-1})}\Tr(\Sigma^{-1}\eta^{\ft}\eta)\Tr(\Sigma^{-1} U \bareta^{\ft}\bareta U^{\ft}).   \label{eq:crosscvGr}
\end{gathered}
\end{equation}
For $k = 1$, $\crosscv_{\Gro}(\bdeta)$ is positive on null vectors if $\eta\neq 0$ and $\bareta \neq 0$. For $1< k <n-1$, the cross-curvature admits both positive and negative values on null vectors.
\end{theorem}
A tedious calculation also shows that for $k=n-1$, $\crosscv_{\Gro}(\bdeta)$ is positive on null vectors with $\eta\neq 0 \neq \bareta$, using the relationship between $\Gr{n}{k}$ and $\Gr{n}{n-k}$. We will not treat it here.

\begin{proof}The condition $(x^{\ft}\omega)_{\sym} = (y^{\ft}\bomg)_{\sym}=0$ for the tangent bundle of the Stiefel manifold follows from $x^{\ft}x = y^{\ft}y = \dI_k$. We can verify $(x^{\ft}(\Pi^{\Sto}(\ttq)\bdomg)_x)_{\sym}=0$ from
  $$x^{\ft}\omega - 2\Sigma L^{-1}_{\Sigma} (x^{\ft}\omega)_{\sym}
+  \omega^{\ft}x - 2 L^{-1}_{\Sigma} (x^{\ft}\omega)_{\sym}\Sigma = 2(x^{\ft}\omega)_{\sym} -2 L_{\Sigma}L^{-1}_{\Sigma} (x^{\ft}\omega)_{\sym}=0,
$$
and similarly $(y^{\ft}(\Pi^{\Sto}(\ttq)\bdomg)_y)_{\sym}=0$. If $(x^{\ft}\omega)_{\sym}=0$ then $\Tr\Sigma^{-1}x^{\ft}\omega=0$, hence $\Pi^{\Sto}(\ttq)$ is idempotent. For metric compatibility, if $\bdeta$ is tangent then $2K^2\langle \bdomg - \Pi^{\Sto}(\ttq)\bdomg, \bdeta\rangle_*$ expands to
$$\begin{gathered}
-\Tr Uy^{\ft}\left( 2yU^{\ft}L^{-1}_{\Sigma} (x^{\ft}\omega)_{\sym}
+ (\alpha+\Tr\Sigma^{-1})^{-1} \Tr(\Sigma^{-1}x^{\ft}\omega)(x - yU^{\ft}\Sigma^{-1})\right)\Tr x^{\ft}\bareta U^{\ft} \\
- \Tr\eta^{\ft} yU^{\ft}\Tr x^{\ft}\left(
     2x L^{-1}_{\Sigma} (U y^{\ft}\bomg U^{\ft})_{\sym} U \
    + (\alpha+\Tr \Sigma^{-1})^{-1}\Tr(U^{\ft} \Sigma^{-1}U y^{\ft}\bomg)(y - x\Sigma^{-1}U\right) U^{\ft}\\
+ K\Tr U \bareta^{\ft}\left(2yU^{\ft}L^{-1}_{\Sigma} (x^{\ft}\omega)_{\sym} \
+ (\alpha+\Tr\Sigma^{-1})^{-1} \Tr(\Sigma^{-1}x^{\ft}\omega)(x - yU^{\ft}\Sigma^{-1})\right) +\\
K\Tr\eta^{\ft}\left(
  2x L^{-1}_{\Sigma} (U y^{\ft}\bomg U^{\ft})_{\sym} U \
  + (\alpha+\Tr \Sigma^{-1})^{-1}\Tr(U^{\ft} \Sigma^{-1}U y^{\ft}\bomg)(y - x\Sigma^{-1}U)    \right)U^{\ft}\\
=  -\left(  \Tr\Sigma^{-1}x^{\ft}\omega
+ (\alpha+\Tr\Sigma^{-1})^{-1} \Tr(\Sigma^{-1}x^{\ft}\omega) (\Tr\Sigma - \Tr\Sigma^{-1})\right)\Tr x^{\ft}\bareta U^{\ft} \\
- \Tr\eta^{\ft} yU^{\ft}\left(
\Tr \Sigma^{-1} U y^{\ft}\bomg U^{\ft} \
     + (\alpha+\Tr \Sigma^{-1})^{-1}\Tr(U^{\ft} \Sigma^{-1}U y^{\ft}\bomg)\Tr (\Sigma - \Sigma^{-1}\right) \\
+ K\left((2\Tr U \bareta^{\ft}yU^{\ft}L^{-1}_{\Sigma} (x^{\ft}\omega)_{\sym} \
+ (\alpha+\Tr\Sigma^{-1})^{-1} \Tr(\Sigma^{-1}x^{\ft}\omega)\Tr (U \bareta^{\ft}x - U \bareta^{\ft}yU^{\ft}\Sigma^{-1})\right) +\\
K\left(
  2\Tr\eta^{\ft}x L^{-1}_{\Sigma} (U y^{\ft}\bomg U^{\ft})_{\sym} \
  + (\alpha+\Tr \Sigma^{-1})^{-1}\Tr(U^{\ft} \Sigma^{-1}U y^{\ft}\bomg)\Tr(\eta^{\ft}yU^{\ft} - \eta^{\ft}x\Sigma^{-1})    \right).
\end{gathered}
$$
Since $\eta^{\ft}x$ and $\bareta^{\ft}y$ are antisymmetric, their trace inner product with symmetric matrices, eg  $\Sigma^{-1}$ and $L^{-1}_{\Sigma} U (y^{\ft}\bomg) U^{\ft})_{\sym}$ are zero. The above becomes
$$\begin{gathered}
=  -(\alpha+\Tr\Sigma^{-1})^{-1}\left((\alpha+\Tr\Sigma^{-1})  \Tr \Sigma^{-1} x^{\ft}\omega
+  \Tr(\Sigma^{-1}x^{\ft}\omega) (\Tr\Sigma - \Tr\Sigma^{-1})\right)\Tr x^{\ft}\bareta U^{\ft} \\
- (\alpha+\Tr \Sigma^{-1})^{-1}\Tr\eta^{\ft} yU^{\ft}\left(
(\alpha+\Tr \Sigma^{-1})\Tr\Sigma^{-1} U y^{\ft}\bomg U^{\ft} \
     + \Tr(U^{\ft} \Sigma^{-1}U y^{\ft}\bomg)\Tr (\Sigma - \Sigma^{-1})\right) \\
+ K
(\alpha+\Tr\Sigma^{-1})^{-1} \Tr(\Sigma^{-1}x^{\ft}\omega)\Tr U \bareta^{\ft}x
+ K  
  (\alpha+\Tr \Sigma^{-1})^{-1}\Tr(U^{\ft} \Sigma^{-1}U y^{\ft}\bomg)\Tr\eta^{\ft}yU^{\ft}   
\end{gathered}
$$
which reduces to $0$ after grouping, with $K=\alpha + \Tr\Sigma$. The existence of the metric projection implies $\sfg$ is nondegenerate on $T\cQ_{\Sto}$. The restriction of $\sfg$ to the horizontal subspace of $T\cQ_{\Sto}$ is nondegenerate as it is the case for the vertical subspace (similar proof to the $\R_*^{n\times k}$ case). The expression for $\ttH^{\Gro}$ is the composition of $\Pi^{\Sto}$ and the horizontal projection of $\Sd{n}{k}$. For \cref{eq:GrassTanKM}, express $\eta = x_{\perp}B + xC$ then $C$ is antisymmetric, and $Uy^{\ft}\eta$ is symmetric is equivalent to
$$Uy^{\ft}x_{\perp}B + Uy^{\ft}xC = -C(Uy^{\ft}x)^{\ft} + (Uy^{\ft}x_{\perp}B)^{\ft}.$$
With $Uy^{\ft}x=\Sigma$, we get the equation for $C$. A similar computation for $y$ and $\bB$ gives \cref{eq:GrassTanKM}. It is clear that a vector of that form is horizontal.

One way to compute $\crosscv_{\Gro}$ is to use the affine Gauss-Codazzi equation. On $\cQ_{\Sto}$, we consider three bundles: $\cH_{\Gro}\subset T\cQ_{\Sto} \subset\cQ_{\Sto}\times\cE$. Let $\Two^*$ be the second fundamental form  of the bundle projection $\ttH^{\Gro}$ of the inclusion $\cH_{\Gro}\subset \cQ_{St}\times\cE$, $\sA$ be the O'Neill's tensor of the quotient $\qq:\cQ_{\Sto}\mapsto\cB_{\Gro}$, $\crosscv_{\R_*}=\langle \RcQ_{\eta_{x0},\eta_{0y}}\eta_{0y},\eta_{x0}\rangle_*$ be  given in \cref{eq:RcQFR}. For a horizontal $\bdeta$
$$\begin{gathered}
\crosscv_{\Gro}(\bdeta) = \langle \Rc_{\bdeta_{x0}, \bdeta_{0y}}^{\nabla^{\cH_{\Gro}}}\bdeta_{0y}, \bdeta_{x0}\rangle_* - 2 \langle \sA_{\bdeta_{x0}}\bdeta_{0y}, \sA_{\bdeta_{0y}}\bdeta_{x0} \rangle_*\\
=\crosscv_{\R_*} - \langle \Two^*(\bdeta_{x0}, \bdeta_{0y}), \Two^* (\bdeta_{0y}, \bdeta_{x0}) \rangle_* + \langle \Two^* (\bdeta_{0y}, \bdeta_{0y}), \Two^* (\bdeta_{x0}, \bdeta_{x0}) \rangle_* + 2 \langle \sA_{\bdeta_{x0}}\bdeta_{0y}, \sA_{\bdeta_{x0}}\bdeta_{0y} \rangle_*. \\
\end{gathered}$$
For two horizontal vectors $\bdxi, \bdeta\in\cH_{\Gro}$, from \cref{eq:oneilLie} and differentiating \cref{eq:HGro}
$$\sA_{\bdxi}\bdeta  = (
x L^{-1}_{\Sigma} (U (\bareta^{\ft}\xi-\bxi^{\ft}\eta))_{\asym}),
 y U^{\ft}L^{-1}_{\Sigma} ((\eta^{\ft}\bxi
 - \xi^{\ft}\bareta) U^{\ft})_{\asym}) U 
 ).
 $$
Set $t_S =\Tr\Sigma, t_I = \Tr\Sigma^{-1}$. Note $(\dI_{\cE} - \ttH^{\Gro})\Gamma^{\cQ_{\R_*^{n\times k}}}(\bdxi, \bdeta)$ in \cref{eq:TwocW} is the vertical component of \cref{eq:GammaxFR,eq:GammayFR}, which we can use to simplify $\Two^*$ as below
 $$(\dI_{\cE} - \ttH^{\Gro})\Gamma^{\cQ}(\bdxi, \bdeta) = (x L^{-1}_{\Sigma} (U\bxi^{\ft}\eta + U\bareta^{\ft}\xi)_{\asym},
            - yU^{\ft}L^{-1}_{\Sigma} (U\bareta^{\ft}\xi + 
            U\bxi^{\ft}\eta)_{\asym}U),
                                          $$
 $$\begin{gathered}\Two^*(\bdxi, \bdeta) = (
 - x L^{-1}_{\Sigma} (U \bxi^{\ft}\eta - U\bareta^{\ft}\xi)_{\asym}
 - 2y U^{\ft}L^{-1}_{\Sigma} (\xi^{\ft}\eta)_{\sym}
 - (\alpha+t_I)^{-1}\Tr(\Sigma^{-1}\xi^{\ft}\eta)(x - y U^{\ft}\Sigma^{-1}),\\
 - y U^{\ft} L^{-1}_{\Sigma} (U\bxi^{\ft}\eta
 - U\bareta^{\ft}\xi)_{\asym}U
 - 2x L^{-1}_{\Sigma} (U \bxi^{\ft}\bareta U^{\ft})_{\sym} U
 - (\alpha+t_I)^{-1}\Tr(U^{\ft}\Sigma^{-1}U\bxi^{\ft}\bareta)(y - x\Sigma^{-1}U)
            ),
\end{gathered}$$
 $$\begin{gathered}
\Two^*(\bdeta_{x0}, \bdeta_{0y}) =(x L^{-1}_{\Sigma} (U\bareta^{\ft}\eta)_{\asym}
,
y U^{\ft}L^{-1}_{\Sigma} (U\bareta^{\ft}\eta)_{\asym}) U),\\
\Two^*(\bdeta_{x0}, \bdeta_{x0}) = (
            - 2yU^{\ft} L^{-1}_{\Sigma} \eta^{\ft}\eta
            - (\alpha+t_I)^{-1}\Tr(\Sigma^{-1}\eta^{\ft}\eta)(x - yU^{\ft}\Sigma^{-1}),
            0
            ),\\
   \Two^*(\bdeta_{0y}, \bdeta_{0y}) = (0,
- 2xL^{-1}_{\Sigma} (U \bareta^{\ft}\bareta U^{\ft}) U
- (\alpha+t_I)^{-1}\Tr (U^{\ft}\Sigma^{-1}U\bareta^{\ft}\bareta)(y - x\Sigma^{-1}U)
            )).
\end{gathered}$$
From here, $\langle\sA_{\bdeta_{x0}}\bdeta_{0y},\sA_{\bdeta_{x0}}\bdeta_{0y}\rangle_* = \langle \Two^*(\bdeta_{x0}, \bdeta_{0y}), \Two^*(\bdeta_{x0}, \bdeta_{0y}) \rangle_* $ are both given by \cref{eq:ONeilFrNorm}. As
%            $$\langle (\omega, 0), (0 , \bomg) = - \frac{1}{2K^2}\Tr(Uy^{\ft}\omega)\Tr(x^{\ft}\bomg U^{\ft}) + \frac{1}{2K}\Tr(\omega^{\ft}\bomg U^{\ft})$$
$$\begin{gathered}\langle (yU^{\ft} L^{-1}_{\Sigma} \eta^{\ft}\eta), 0),
  (0,   xL^{-1}_{\Sigma} (U \bareta^{\ft}\bareta U^{\ft}) U)\rangle_* =
  - \frac{1}{2K^2}\Tr(Uy^{\ft}yU^{\ft} L^{-1}_{\Sigma} \eta^{\ft}\eta))\Tr(x^{\ft}xL^{-1}_{\Sigma} (U \bareta^{\ft}\bareta U^{\ft}) U) U^{\ft}) \\
  + \frac{1}{2K}\Tr((yU^{\ft} L^{-1}_{\Sigma} \eta^{\ft}\eta))^{\ft}xL^{-1}_{\Sigma} (U \bareta^{\ft}\bareta U^{\ft}) U) U^{\ft})\\
= - \frac{1}{8K^2}\Tr(\Sigma^{-1} \eta^{\ft}\eta) \
            \Tr(\Sigma^{-1} U \bareta^{\ft}\bareta U^{\ft}) 
            + \frac{1}{K}\Tr\{L^{-1}_{\Sigma}(\eta^{\ft}\eta)\Sigma L^{-1}_{\Sigma} (U\bareta^{\ft}\bareta U^{\ft})\},
\end{gathered}          $$
$$\begin{gathered}
\langle ( x-y U^{\ft}\Sigma^{-1}, 0),
(0,   xL^{-1}_{\Sigma} (U \bareta^{\ft}\bareta U^{\ft}) U)\rangle_*
=- \frac{1}{2K^2}\Tr(Uy^{\ft}( x-y U^{\ft}\Sigma^{-1}))\Tr(x^{\ft}(xL^{-1}_{\Sigma} (U \bareta^{\ft}\bareta U^{\ft}) U)) U^{\ft}) \\
+ \frac{1}{2K}\Tr(( x-y U^{\ft}\Sigma^{-1})^{\ft}(xL^{-1}_{\Sigma} (U \bareta^{\ft}\bareta U^{\ft}) U)) U^{\ft})
= - \frac{t_S - t_I}{4K^2}\Tr(\Sigma^{-1} U \bareta^{\ft} \bareta U^{\ft}),
\end{gathered}$$
$$\begin{gathered}
\langle (yU^{\ft} L^{-1}_{\Sigma} \eta^{\ft}\eta), 0),
(0,   y-x\Sigma^{-1}U\rangle_* =  - \frac{t_S - t_I}{4K^2}\Tr(\Sigma^{-1}\eta^{\ft}\eta),\\
\langle ( x-yU^{\ft}\Sigma^{-1}, 0) , (0,   y-x\Sigma^{-1}U\rangle_*
=  -\frac{(t_S-t_I)^2}{2K^2} +\frac{(t_S-t_I)}{2K}= \frac{(t_S-t_I)(\alpha + t_I)}{2K^2},
\end{gathered}$$
expanding $\langle
\Two^*(\bdeta_{x0}, \bdeta_{x0}) ,\Two^*(\bdeta_{0y}, \bdeta_{0y})
\rangle_*$, in addition to $\frac{4}{K}\Tr\{L^{-1}_{\Sigma}(\eta^{\ft}\eta)\Sigma L^{-1}_{\Sigma} (U\bareta^{\ft}\bareta U^{\ft})\}$, the remaining terms collect to $f\Tr(\Sigma^{-1}\eta^{\ft}\eta)\Tr (U^{\ft}\Sigma^{-1}U\bareta^{\ft}\bareta)$, with the total coefficient 
$$f = -\frac{4}{8K^2} -2\frac{t_S- t_I}{4K^2(\alpha + t_I)} -2\frac{t_S-t_I}{4K^2(\alpha + t_I)} +\frac{t_S-t_I}{2K^2(\alpha+t_I)}=-\frac{1}{2K(\alpha+t_I)},
$$
which gives us \cref{eq:crosscvGr}.

When $k=1$, $\Sigma = \sigma = |x^{\ft}y|$ is a scalar, the cross-curvature reduces to
$$\crosscv(\bdeta) = - \langle \bdeta, \bdeta\rangle_*^2 \
+ \frac{\eta^{\ft}\eta \bareta^{\ft}\bareta\alpha}{2(\alpha + \sigma)(1+\alpha\sigma)},$$
which is positive for null vectors with $\eta^{\ft}\eta \bareta^{\ft}\bareta\neq 0$.

When $1 < k <n-1$, for simplicity, consider the case $x = y$, then $\Sigma = \dI_k = U$. Let $\eta =x_{\perp}B, \bareta = x_{\perp}\bB$. The cross-curvature for a null vector reduces to
$$\frac{1}{\alpha +k}\Tr(\bareta^{\ft}\eta)_{\asym}(\eta^{\ft}\bareta)_{\asym} + \frac{1}{2(\alpha + k)}\Tr\eta^{\ft}\eta \bareta^{\ft}\bareta- \frac{1}{2(\alpha + k)^2}\Tr\eta^{\ft}\eta \Tr\bareta^{\ft}\bareta.
$$
Let $k_m = \min(k, n-k)$, and for some $t\in \R$, set $D_1$ and $D_2$ to be diagonal matrices in $\R^{k_m\times k_m}$ with diagonals $(1, t,\cdots t)$ and $(-t - (k_m-2)t^2, t,\cdots t, 1)$, respectively. If $k \leq n/2$, let $B=\begin{bmatrix}D_1\\0_{((n-2k)\times k}\end{bmatrix}, \bar{B}=\begin{bmatrix}D_2\\0_{(n-2k)\times k}\end{bmatrix}$ , and $\eta = x_{\perp}B, \bareta = x_{\perp}\bar{B}$. Then $\langle \bdeta, \bdeta\rangle = 1/K\Tr D_1D_2=0$, while the cross-curvature is
$$  \frac{(t+(k_m-2)t^2)^2 + (k_m-2)t^4 + t^2}{2(\alpha + k)} \
        - \frac{(1+(k_m-1)t^2)((t+(k_m-2)t^2)^2 + (k_m-2)t^2 + 1)}{2(\alpha + k)^2}.
$$
It is negative if $t$ is small, and equal to $\frac{1}{2(k+\alpha)^2}k_m(k_m-1)(\alpha+k-k_m)>0$ if $t=1$. The case $n/2 < k < n-1$ is similar, with the zero matrices appended on $D_1, D_2 \in\R^{k_m\times k_m}$ to the right.
\end{proof}
\section{Discussion and future works}\label{sec:Discussion}
We provide global formulas for several geometric expressions, which we believe facilitate computations for embedded manifolds in both pure and applied mathematics. We expect the method will have new applications in mechanics, gravity, and optimal transport. The nonnegativity cross-curvature on the reflector antenna-type cost function on the semidefinite matrices is likely to extend to certain representations of semisimple Lie groups of non-compact type, and other cost functions defined via the polar decomposition could also be considered, in particular, we hope to study the cross-curvature of square Riemannian cost function for off-diagonal pairs, presumably on certain symmetric spaces.
\begin{appendix}
\section{Extension of a metric tensor to a metric operator}\label{appx:Extend}
%While an extension of a metric tensor on $\cQ\subset \cE$ to a metric operator $\sfg$ operating on $\cE$ is not unique, it always exists.
\begin{proposition}\label{prop:emb_exists}If a manifold $\cQ$ is differentiably embedded in $\cE$, for $q\in\cQ$, let $\Pi^{\cE}: q\mapsto \Pi^{\cE}(q)$ be the projection from $\cE$ to $T_q\cQ$ under Euclidean inner product $\langle, \rangle_{\cE}$ on $\cE$. Assume $\cQ$ is equipped with a semi-Riemannian metric $\langle,\rangle_R$, the operator-valued function $\sfg$ from $\cQ$ to $\Lin(\cE, \cE)$
\begin{equation}\label{eq:stdg}
  \sfg(q)\omega = (\dI_{\cE} - \Pi^{\cE}(q))\omega + \sfg_{R, q}(\Pi^{\cE}(q)\omega),
\end{equation}  
where $\sfg_{R, q}$ is the unique self-adjoint operator on $T_q\cQ$ such that
\begin{equation}\langle \xi, \eta\rangle_{R, q} = \langle \xi, \sfg_{R, q}\eta\rangle_{\cE}.\label{eq:construct}\end{equation}
extends the pairing from $T_q\cQ$ to $\cE$, that means
  \begin{equation}\langle \eta, \xi\rangle_{R, q} = \langle\eta, \sfg(q)\xi\rangle_{\cE}.\label{eq:consistentExtend}
  \end{equation}
If $\sfg$ is defined in \cref{eq:stdg}, the operator $\Pi^{\cE}$ satisfies $\Pi^{\cE}\sfg =\sfg\Pi^{\cE}$, and the cotangent bundle is identified with the tangent bundle.

If $\langle,\rangle_R$ is a Riemannian metric then $\sfg$ is positive-definite.
\end{proposition}
\begin{proof} To show $\sfg$ defined in \cref{eq:stdg} is self-adjoint, observe for each $q$, $T_q\cQ$ and $N_q\cQ :=\Imag(\dI_{\cE} - \Pi^{\cE}(q))$ are orthogonal complement invariant subspaces of $\sfg(q)$, restricting to each subspace $\sfg(q)$ is self-adjoint. Equation (\ref{eq:consistentExtend}) follows from \cref{eq:construct}. We can see $\sfg^{-1}$ is given by
  $$\sfg(q)^{-1}\omega = (\dI_{\cE} - \Pi^{\cE}(q))\omega + \sfg_{R, q}^{-1}(\Pi^{\cE}(q)\omega).$$
  where we use $\Pi^{\cE}(q)\sfg_{R,q}\eta =\sfg_{R,q}\eta$ for $\eta\in T_q\cQ$ as $\sfg_R(q)$ maps $T_q\cQ$ to itself.

For the next statement, apply $\Pi^{\cE}(q)$ to both sides of \cref{eq:stdg}, we get
$$\Pi^{\cE}(q)\sfg(q)\omega = \Pi^{\cE}(q)\sfg_{R, q}(\Pi^{\cE}(q)\omega) = \sfg_{R, q}(\Pi^{\cE}(q)\omega).$$
If $\sfg_R$ is Riemannian, $\sfg$ is positive-definite as it is positive-definite on both $T_q\cQ$ and $\Imag(\dI_{\cE}-\Pi^{\cE}(q))$.
\end{proof}
For the sphere $S^{n-1}\subset \cE=\R^n$, assume $\Lambda: q\mapsto\Lambda(q)\in \R^{n\times n}$ is a positive-definite metric operator. Then $\xi . \Lambda(q)\eta$ for $\xi, \eta\in T_q\cQ$ define a metric on $S^{n-1}$. Equation (\ref{eq:stdg}) extends it to
$$\sfg(q)\omega = qq^T\omega + (\dI_n - qq^{\ft})\Lambda(q)(\dI_n - qq^{\ft})\omega$$
for $\omega\in\cE$, and $T_q^*\cQ$ is identified with $T_q\cQ$. Another extension of this metric to an operator on $\cE$ is to define $\sfg_1(q)\omega = \Lambda(q)\omega$, with $T^*\cQ$ identified with $\Lambda(q) T_q\cQ$. The Levi-Civita connection in \cref{eq:LeviCivita} for either extension is the same on tangent vectors, but to verify it algebraically even when $\Lambda(q)$ is constant, diagonal is quite an exercise.
\end{appendix}

\bibliographystyle{elsarticle-num-names}
\bibliography{embedded_hamilton}
%\printaddress

\end{document}